\documentclass{amsart}

\usepackage[hmargin=3.5cm,top=3cm,bottom=3cm,includehead]{geometry}

\usepackage{amsmath, amssymb, amsthm, mathrsfs}
\usepackage{graphicx,color,enumerate}
\usepackage{dsfont}

\usepackage{hyperref}

\newcommand\N{{\mathbb N}}
\newcommand\R{{\mathbb R}}
\newcommand\T{{\mathbb T}}
\newcommand\C{{\mathbb C}}

\newcommand\Z{{\mathbb Z}}

\def\AA{{\mathcal A}}
\def\BB{{\mathcal B}}
\def\CC{{\mathcal C}}

\def\EE{{\mathcal E}}
\def\FF{{\mathcal F}}
\def\GG{{\mathcal G}}
\def\HH{{\mathcal H}}

\def\LL{{\mathcal L}}

\def\SS{{\mathcal S}}

\def\BBB{{\mathscr B}}

\def\e{{\epsilon}}
\def\eps{{\varepsilon}}

\newcommand{\la}{\langle}
\newcommand{\ra}{\rangle}

\def\Nt{|\hskip-0.04cm|\hskip-0.04cm|}

\newtheorem{thm}{Theorem}[section]
\newtheorem{theo}[thm]{Theorem}
\newtheorem{prop}[thm]{Proposition}
\newtheorem{lem}[thm]{Lemma}
\newtheorem{cor}[thm]{Corollary}
\theoremstyle{remark}
\newtheorem{rem}[thm]{Remark}

\theoremstyle{definition}

\numberwithin{equation}{section}

\newcommand{\beqn}{\begin{equation}}
\newcommand{\eeqn}{\end{equation}}
\newcommand{\bal}{\begin{aligned}}
\newcommand{\eal}{\end{aligned}}

\newcommand{\Black}{\color{black}}

\begin{document}

\title[Cauchy problem and stability for the Landau equation]{Cauchy problem and exponential stability for the inhomogeneous Landau equation}
\date{\today}

\author[K. Carrapatoso]{Kleber Carrapatoso}
\address[K. Carrapatoso]{\'Ecole Normale Sup\'erieure de Cachan, CMLA (UMR 8536), 61 av. du pr\'esident Wilson, 94235 Cachan, France.}
\email{carrapatoso@cmla.ens-cachan.fr}

\author[I. Tristani]{Isabelle Tristani}
\address[I. Tristani]{Universit\'e Paris Dauphine, Ceremade (UMR 7534), Place du Mar\'echal de Lattre de Tassigny, 75775 Paris, France.}
\email{tristani@ceremade.dauphine.fr}

\author[K.-C. Wu]{Kung-Chien Wu}
\address[K.-C. Wu]{Department of Mathematics, National Cheng Kung University, 70101 Tainan, Taiwan.}
\email{kungchienwu@gmail.com}

\begin{abstract}
This work deals with the inhomogeneous Landau equation on the torus in the cases of hard, Maxwellian and moderately soft potentials. We first investigate the linearized equation and we prove exponential decay estimates for the associated semigroup. We then turn to the nonlinear equation and we use the linearized semigroup decay in order to construct solutions in a close-to-equilibrium setting. Finally, we prove an exponential stability for such a solution, with a rate as close as we want to the optimal rate given by the semigroup decay.
\end{abstract}



\maketitle

\tableofcontents


\section{Introduction}

\subsection{The model} In this paper, we investigate the Cauchy theory associated to the \emph{spatially inhomogeneous Landau equation}. This equation is a kinetic model in plasma physics that describes the evolution of the density function $F=F(t,x,v)$ in the phase space of position and velocities of the particles. In the torus, the equation is given by, for $F=F(t,x,v) \geq 0$ with $t\in \R^+$, $x\in \T^3=\R^3/\Z^3$ (that we assume without loss of generality to have volume one $|\T^3|=1$) and $v\in \R^3$,
\beqn\label{eq:landau}
\left\{
\bal
& \partial_t F  +v\cdot \nabla_x F =  Q(F,F)  \\
& F_{|t=0} = F_0
\eal
\right.
\eeqn
where the Landau operator $Q$ is a bilinear operator that takes the form
\beqn\label{eq:oplandau0}
Q(G,F)(v) = \partial_{i} \int_{\R^3} a_{ij}(v-v_*) \left[ G_* \partial_j F - F \partial_{j}G_*\right]\, dv_*,
\eeqn
and we use the convention of summation of repeated indices, and the derivatives are in the velocity variable, i.e. $\partial_i = \partial_{v_i}$. Hereafter we use the shorthand notations $G_* = G(v_*)$, $F= F(v)$, $\partial_j G_* = \partial_{v_{*j}} G(v_*)$, $\partial_j F = \partial_{v_j} F(v)$, etc.

The matrix $a_{ij}$ is symmetric semi-positive, depends on the interaction between particles and is given by
\beqn\label{eq:aij}
a_{ij}(v) = |v|^{\gamma+2}\left( \delta_{ij} - \frac{v_i v_j}{|v|^2}\right).
\eeqn

We define (see \cite{Vi2}) in $3$-dimension the following quantities
\beqn\label{eq:bc}
\bal
& b_i(v) = \partial_j a_{ij}(v) = - 2 \, |v|^\gamma \, v_i, \\
& c(v) =  \partial_{ij} a_{ij}(v)  = - 2 (\gamma+3) \, |v|^\gamma \quad \text{or}\quad c = 8\pi \delta_0 \; \text{ if }\; \gamma=-3.
\eal
\eeqn
We can rewrite the Landau operator \eqref{eq:oplandau0} in the following way
\beqn\label{eq:oplandau}
Q(G,F) = ( a_{ij}*_v G) \partial_{ij} F - (c*_v G) F
= \nabla_v \cdot \{ (a *_v g) \nabla_v f - (b *_v g)f \}.
\eeqn

We have the following classification: we call hard potentials if $\gamma\in(0,1]$, Maxwellian molecules if $\gamma=0$, moderately soft potentials if $\gamma \in [-2,0)$, very soft potentials if $\gamma \in (-3,-2)$ and Coulombian potential if $\gamma=-3$.
Hereafter we shall consider the cases of hard potentials, Maxwellian molecules and moderately soft potentials, i.e. $\gamma\in[-2,1]$.

The Landau equation conserves mass, momentum and energy. Indeed, at least formally, for any test function $\varphi$, we have
$$
\int_{\R^3} Q(F,F) \varphi \, dv = - \frac12 \int_{\R^3 \times \R^3} a_{ij}(v-v_*) F F_*  \left(  \frac{\partial_i F}{F} -  \frac{\partial_{i} F_*}{F_*} \right)   \left( \partial_j \varphi - \partial_{j} \varphi_* \right) \, dv \, dv_*,
$$
from which we deduce that
\beqn \label{eq:conserv}
\frac{d}{dt} \int_{\T^3 \times \R^3} F \varphi(v) \, dx \, dv=
\int_{\T^3 \times \R^3} \left[Q(F,F) - v \cdot \nabla_x F \right] \varphi(v) \, dx \, dv = 0 \quad \text{for} \quad \varphi(v) = 1,v,|v|^2.
\eeqn
Moreover, the Landau version of the Boltzmann $H$-theorem asserts that the entropy
$$
H(F) := \int_{\T^3 \times \R^3} F \, \log F \, dx \, dv
$$
is non increasing. Indeed, at least formally, since $a_{ij}$ is nonnegative, we have the following inequality for the entropy dissipation $D(F)$:
$$
\begin{aligned}
D(F) &:= -\frac{d}{dt} H(F) \\
&=\frac12 \int_{\T^3 \times \R^3 \times \R^3} a_{ij}(v-v_*) F F_*
\left( \frac{\partial_i F}{F} -  \frac{\partial_{i} F_*}{F_*} \right)
\left( \frac{\partial_j F}{F} -  \frac{\partial_{j} F_*}{F_*} \right) \, dv \, dv_* \, dx \geq 0.
\end{aligned}
$$

It is known that the global equilibria of \eqref{eq:landau} are global Maxwellian distributions that are independent of time $t$ and position $x$. We shall always consider initial data $F_0$ verifying
$$
\int_{\T^3 \times \R^3} F_0 \, dx \, dv = 1, \quad
\int_{\T^3 \times \R^3} F_0 \, v \, dx \, dv = 0, \quad
\int_{\T^3 \times \R^3} F_0 \, |v|^2 \, dx \, dv = 3,
$$
therefore we consider the Maxwellian equilibrium
$$
\mu(v)=(2 \pi)^{-3/2} e^{-|v|^2/2}
$$
with same mass, momentum and energy of the initial data.

We linearize the Landau equation around $\mu$ with the perturbation
$$
F=\mu + f.
$$
The Landau equation \eqref{eq:landau} for $f=f(t,x,v)$ takes the form
\beqn\label{eq:landau-f}
\left\{
\bal
\partial_t f &= \Lambda f + Q(f,f) := \LL f - v\cdot \nabla_x f + Q(f,f)\\
f_{| t=0} &= f_0 = F_0-\mu,
\eal
\right.
\eeqn
where $\Lambda =\LL - v\cdot \nabla_x $ is the inhomogeneous linearized Landau operator and the homogeneous linearized Landau operator $\LL$ is given by
\beqn\label{eq:oplandaulin}
\bal
\LL f &:= Q(\mu,f)+Q(f,\mu) \\
&= ( a_{ij}* \mu) \partial_{ij} f - (c* \mu) f
+ ( a_{ij}* f) \partial_{ij} \mu - (c* f) \mu.
\eal
\eeqn
Through the paper we introduce the following notation
\beqn\label{eq:barabc}
\bal
\bar a_{ij}(v) = a_{ij} \ast \mu , \quad
\bar b_{i}(v) = b_i \ast \mu , \quad
\bar c(v) = c \ast \mu .
\eal
\eeqn
The conservation laws \eqref{eq:conserv} can then be rewritten as, for all $t \ge 0$,
\beqn\label{laws}
\int_{\T^3 \times \R^3} f(t,x,v) \varphi(v) \, dx \, dv
 = 0
\quad \text{for} \quad \varphi(v) = 1,v,|v|^2.
\eeqn

\medskip
\subsection{Notations}

Through all the paper we shall consider function of two variables $f=f(x,v)$ with $x \in \T^3$ and $v \in \R^3$.
Let $m=m(v)$ be a positive Borel weight function and $1\leq p,q \leq \infty$. 
We define the space $L^q_{x} L^p_{v} (m) $ as the Lebesgue space associated to the norm, for $f=f(x,v)$,
$$
\bal
\| f \|_{L^q_{x} L^p_{v} (m)} &:= \big\| \|  f \|_{L^p_v(m)}   \big\|_{L^q_{x}}
:= \big\| \| m \, f \|_{L^p_v}   \big\|_{L^q_{x}}  \\
& = \left(\int_{\T^3_x}  \|  f (x, \cdot )\|_{L^p_v(m)}^q \, dx \right)^{1/q} \\
& = \left(\int_{\T^3_x}  \left(\int_{\R^3_v} |f(x,v)|^p \, m(v)^p \,dv \right)^{q/p} \, dx\right)^{1/q} .
\eal
$$
We also define the high-order Sobolev spaces $ W^{n,q}_{x} W^{\ell,p}_{v} (m)$, for $n,\ell \in \N$:
$$
\| f \|_{ W^{n,q}_x W^{\ell,p}_v (m)} = \sum_{0\leq |\alpha| \leq \ell, \, 0 \leq |\beta| \leq n, \, |\alpha| + |\beta| \leq \max(\ell,n)}
 \| \partial^\alpha_v \partial^\beta_x f \|_{L^q_x L^p_v (m)}.
$$
This definition reduces to the usual weighted Sobolev space $W^{\ell,p}_{x,v}(m)$ when $p=q$ and $\ell=n$, and we recall the shorthand notation $H^\ell = W^{\ell,2}$. We shall denote
$W^{\ell,p}(m) = W^{\ell,p}_{x,v}(m)$ when considering spaces in the two variables $(x,v)$.

\medskip

Let $X,Y$ be Banach spaces and consider a linear operator $\Lambda : X \to X$. We shall denote by $\SS_{\Lambda}(t) = e^{t\Lambda}$ the semigroup generated by $\Lambda$. Moreover we denote by $\BBB(X,Y)$ the space of bounded linear operators from $X$ to $Y$ and by $\| \cdot \|_{\BBB(X,Y)}$ its norm operator, with the usual simplification $\BBB(X) = \BBB(X,X)$. 

\smallskip

For simplicity of notations, hereafter, we denote $\langle v \rangle=(1+|v|^{2})^{1/2}$; $a\sim b$ means that there exist constants $c_{1}, c_{2}>0$ such that $c_{1}b\leq a \leq c_{2}b$; we abbreviate ``{ $\leq C$} " to ``{ $ \lesssim$ }", where $C$ is a positive constant depending only on fixed number.

\subsection{Main results}
\subsubsection{Cauchy theory and convergence to equilibrium}
We develop a Cauchy theory of perturbative solutions in ``large'' spaces for $\gamma \in [-2,1]$. We also deal with the problem of convergence to equilibrium of the constructed solutions, we prove an exponential convergence to equilibrium. Let us now state our assumptions for the main result.

\bigskip
\noindent{\bf (H0)} Assumptions for Theorem \ref{main1}:
\begin{itemize}
\item {\bf Hard potentials $\gamma \in (0,1]$ and Maxwellian molecules $\gamma=0$:}

\smallskip

\begin{enumerate}[$\qquad(i)$]

\item \emph{Polynomial weight}: $m = \la v \ra^k$ with $k > \gamma + 7 + 3/2$.
\smallskip

\item \emph{Stretched exponential weight}: $m = e^{r\la v \ra^s}$ with $r>0$ and $s \in (0,2)$.

\smallskip

\item \emph{Exponential weight}: $m = e^{r\la v \ra^2}$ with $r \in (0,1/2)$.

\end{enumerate}

\medskip

\item {\bf Moderately soft potentials $\gamma \in [-2,0)$:}

\smallskip

\begin{enumerate}[$\qquad(i)$]

\item \emph{Stretched exponential weight}: $m = e^{r\la v \ra^s}$ with $r>0$, $s \in (-\gamma,2)$.

\smallskip

\item \emph{Exponential weight}: $m = e^{r\la v \ra^2}$ with $r \in (0,1/2)$.

\end{enumerate}

\end{itemize}
Through the paper, we shall use the notation $\sigma=0$ when $m=\langle v \rangle^k$ and $\sigma=s$ when $m=e^{r\langle v \rangle^s}$.


\bigskip

We define the space $\HH^3_x L^2_v(m)$ (for $m$ a polynomial or exponential weight) associated to the norm
\beqn\label{HH3xL2v}
\bal
\| h \|_{\HH^3_x L^2_v(m)}^2
&= \| h\|_{L^2_x L^2_v(m)}^2
+ \| \nabla_x h \|_{L^2_x L^2_v(m \la v \ra^{-(1 - \sigma/2)})}^2 \\
&\quad
+ \| \nabla^2_x h \|_{L^2_x L^2_v(m \la v \ra^{-2(1 - \sigma/2)})}^2
+\| \nabla^3_x h \|_{L^2_x L^2_v(m \la v \ra^{-3(1 - \sigma/2)})}^2.
\eal
\eeqn

We also introduce the velocity space $H^1_{v,*}(m)$ through the norm
\beqn\label{H1v*}
\bal
\| h \|_{H^1_{v,*}(m)}^2
&= \| h\|_{L^2_v(m \la v \ra^{(\gamma+\sigma)/2})}^2
+ \| P_v \nabla_v h \|_{L^2_v(m \la v \ra^{\gamma/2})}^2
+ \| (I-P_v) \nabla_v h \|_{L^2_v(m \la v \ra^{(\gamma+2)/2})}^2 ,
\eal
\eeqn
with $P_v$ the projection onto $v$, namely $P_v \xi = \left(\xi \cdot \frac{v}{|v|}\right) \frac{v}{|v|}$, as well as the space $\HH^3_x (H^1_{v,*}(m))$ associated to
\beqn\label{HH3xH1v*}
\bal
\| h \|_{\HH^3_x (H^1_{v,*}(m))}^2
&= \| h\|_{L^2_x (H^1_{v,*}(m))}^2
+ \| \nabla_x h \|_{L^2_x (H^1_{v,*}(m \la v \ra^{-(1 - \sigma/2)}))}^2 \\
&\quad
+ \| \nabla^2_x h \|_{L^2_x (H^1_{v,*} (m \la v \ra^{-2(1 - \sigma/2)}))}^2
+\| \nabla^3_x h \|_{L^2_x (H^1_{v,*}(m \la v \ra^{-3(1 - \sigma/2)}))}^2 \\
&= \int_{\T^3_x} \| h\|_{ H^1_{v,*}(m)}^2
+  \int_{\T^3_x} \| \nabla_x h \|_{H^1_{v,*}(m \la v \ra^{-(1 - \sigma/2)})}^2 \\
&\quad
+ \int_{\T^3_x} \| \nabla^2_x h \|_{H^1_{v,*} (m \la v \ra^{-2(1 - \sigma/2)})}^2
+\int_{\T^3_x} \| \nabla^3_x h \|_{H^1_{v,*}(m \la v \ra^{-3(1 - \sigma/2)})}^2 .
\eal
\eeqn

\Black

\medskip

Here are the main results on the fully nonlinear problem (\ref{eq:landau-f}) that we prove in what follows. For simplicity denote $X := \HH^3_x L^2_v(m)$ and $Y := \HH^3_x (H^1_{v,*}(m))$ (see \eqref{HH3xL2v} and \eqref{HH3xH1v*}).

\begin{thm}\label{main1}
Consider assumption {\bf (H0)} with some weight function $m$. We assume that $f_0$ satisfies \eqref{laws} and also that $F_0 = \mu + f_0 \ge 0$.
There is a constant $\e_0 = \e_0(m) >0$ such that if
$\| f_0 \|_{X} \le \e_0$,
then there exists a unique global weak solution $f$ to the Landau equation \eqref{eq:landau-f}, which satisfies, for some constant $C>0$,
$$
\| f \|_{L^\infty ([0,\infty); X)} + \| f \|_{L^2([0,\infty); Y)} \le C \e_0.
$$
Moreover, this solution verifies an exponential decay: for any $0<\lambda_2 < \lambda_1$ there exists $C>0$ such that
$$
\forall\, t \ge 0, \quad
\| f(t)  \|_{X} \le C \, e^{- \lambda_2 t} \, \| f_0 \|_{X},
$$
where $\lambda_1 >0$ is the optimal rate given by the semigroup decay of the associated linearized operator in Theorem~\ref{thm:extension}.

\end{thm}

Let us comment our result and give an overview on the previous works on the Cauchy theory for the inhomogeneous Landau equation.
For general large data, we refer to the papers of DiPerna-Lions \cite{DPL} for global existence of the so-called renormalized solutions in the case of the Boltzmann equation. This notion of solution have been extend to the Landau equation by Alexandre-Villani \cite{AV} where they construct global renormalized solutions with a defect measure. We also mention the work of Desvillettes-Villani~\cite{DV-boltzmann} that proves the convergence to equilibrium of a priori smooth solutions for both Boltzmann and Landau equations for general initial data.

In a close-to-equilibrium framework, Guo in \cite{Guo} has developed a theory of perturbative solutions in a space with a weight prescribed by the equilibrium of type $H^N_{x,v}(\mu^{-1/2})$, for any $N \ge 8$, and for all cases $\gamma \in [-3,1] $, using an energy method. Later, for $\gamma \in [-2,1]$, Mouhot-Neumann~\cite{mouNeu} improve this result to $H^N_{x,v}(\mu^{-1/2})$, for any $N \ge 4$.

Let us underline the fact that Theorem \ref{main1} largely improves previous results on the Cauchy theory associated to the Landau equation in a perturbative setting. Indeed, we considerably have enlarged the space in which the Cauchy theory has been developed in two ways: the weight of our space is much less restrictive (it can be a polynomial or stretched exponential weight instead of the inverse Maxwellian equilibrium) and we also require less assumptions on the derivatives, in particular no derivatives in the velocity variable.

Moreover, we also deal with the problem of the decay to equilibrium of the solutions that we construct. This problem has been considered in several papers by Guo and Strain in \cite{GS1,GS2} first for Coulombian interactions ($\gamma=-3$) for which they proved an almost exponential decay and then, they have improved this result dealing with very soft potentials ($\gamma \in [-3,-2)$) and proving a decay to equilibrium with a rate of type $e^{-\lambda t^p}$ with $p \in (0,1)$. In the case $\gamma \in [-2,1]$, Yu \cite{Yu} has proved an exponential decay in $H^N_{x,v}(\mu^{-1/2})$, for any $N \ge 8$, and Mouhot-Neumann~\cite{mouNeu} in $H^N_{x,v}(\mu^{-1/2})$, for any $N \ge 4$.

\smallskip
We here emphasize that our strategy to prove Theorem \ref{main1} is completely different from the one of Guo in \cite{Guo}. Indeed, he uses an energy method and his strategy is purely nonlinear, he directly derives energy estimates for the nonlinear problem while the first step of our proof is the study of the linearized equation and more precisely the study of its spectral properties. Then, we go back to the nonlinear problem combining the new spectral estimates obtained on the linearized equation with some bilinear estimates on the collision operator. Thanks to this method, we are able to develop a Cauchy theory in a space which is much larger than the one from the previous paper \cite{Guo}. Moreover, we obtain the convergence of solutions towards the equilibrium with an explicit exponential rate.

\smallskip
 Our strategy is thus based on the study of the linearized equation. And then, we go back to the fully nonlinear problem. This is a standard strategy to develop a Cauchy theory in a close-to-equilibrium regime. However, we have to emphasize here that our study of the nonlinear problem is very tricky. Indeed, usually (for example in the case of the non-homogeneous Boltzmann equation for hard spheres in \cite{GMM}), the gain induced by the linear part of the equation allows directly to control the nonlinear part of the equation so that the linear part is dominant and we can use the decay of the semigroup of the linearized equation. In our case, it is more difficult because the gain induced by the linear part is anisotropic and it is not possible to conclude using only natural estimates on the bilinear Landau operator. As a consequence, we establish some new very accurate estimates on the Landau operator to be able to deal with this problem. 

\smallskip

\medskip

Since the study of the linearized equation is the cornerstone of the proof of our main result, we here present the result that we obtain on it and briefly remind previous results.

\smallskip
\subsubsection{The linearized equation.}\label{sec:lin-intro}
We remind the definition of the linearized operator at first order around the equilibrium:
$$
\Lambda f =Q(\mu,f) + Q(f,\mu) - v \cdot \nabla_x f.
$$
We study spectral properties of the linearized operator $\Lambda$ in various weighted Sobolev spaces $W^{n,p}_{x} W^{\ell,p}_{v}$. Let us state our main result on the linearized operator (see Theorem \ref{thm:extension} for a precise statement), which widely generalizes previous results since we are able to deal with a more general class of spaces.

\begin{thm} \label{theo:main2}
Consider hypothesis {\bf (H1)}, {\bf (H2)} or {\bf (H3)} defined in Subsection \ref{subsec:hypm} and a weight function $m$. Let $\EE$ be one of the admissible spaces defined in \eqref{def:EE}.
Then, there exist explicit constants $\lambda_1 >0$ and $C>0$ such that
$$
\forall \, t \ge 0,  \quad \forall\, f \in \EE, \quad
\| S_{\Lambda}(t) f - \Pi_0 f \|_{\EE} \le C \, e^{-\lambda_1 t} \, \| f - \Pi_0 f \|_{\EE},
$$
where $S_\Lambda(t)$ is the semigroup associated to $\Lambda$ and $\Pi_0$ the projector onto the null space of $\Lambda$ by~(\ref{eq:Pi0}).
\end{thm}

\smallskip

We first make a brief review on known results on spectral gap properties of the homogeneous linearized operator $\LL$ defined in \eqref{eq:oplandaulin}.
On the Hilbert space $L^2_v(\mu^{-1/2})$, a simple computation gives that $\LL$ is self-adjoint and $\la \LL h, h \ra_{L^2_v(\mu^{-1/2})} \le 0$, which implies that the spectrum of $\LL$ on $L^2_v(\mu^{-1})$ is included in $\R^-$. Moreover, the nullspace is given by
$$
N(\LL) = \mathrm {Span} \{ \mu , v_1 \mu, v_2 \mu, v_3 \mu , |v|^2 \mu   \}.
$$
We can now state the existing results on the spectral gap of $\LL$ on $L^2_v(\mu^{-1/2})$.
Summarising results of Degond and Lemou~\cite{DL}, Guo~\cite{Guo}, Baranger and Mouhot~\cite{BM}, Mouhot~\cite{M}, Mouhot and Strain~\cite{MS} for all cases $\gamma \in [-3,1]$, we have: there is a constructive constant $\lambda_0>0$ (spectral gap) such that
\beqn\label{eq:GapDLandauMS}
\bal
\la -\LL h, h \ra_{L^2_v(\mu^{-1/2})}
\geq \lambda_0 \| h \|_{H^1_{v,**} (\mu^{-1/2})}^2,
\quad \forall\,  h \in N(\LL)^\perp,
\eal
\eeqn
where the anisotropic norm $\| \cdot \|_{H^1_{v,**} (\mu^{-1/2})}$ is defined by
$$
\bal
\| h \|_{H^1_{v,**} (\mu^{-1/2})}^2 &:= \|\la v \ra^{\gamma/2} P_v \nabla h  \|_{L^2_v( \mu^{-1/2})}^2
  + \|\la v \ra^{(\gamma+2)/2} (I-P_v) \nabla h  \|_{L^2_v(\mu^{-1/2})}^2 \\
  &\quad
  + \|\la v \ra^{(\gamma+2)/2}  h  \|_{L^2_v( \mu^{-1/2})}^2  ,
\eal
$$
where $P_v$ denotes the projection onto the $v$-direction, more precisely $P_v g = \left( \frac{v}{|v|}\cdot g \right) \frac{v}{|v|}$.
We also have from \cite{Guo} the reverse inequality,
which implies a spectral gap for $\LL$ in $L^2_v(\mu^{-1/2})$ if and only if $\gamma+2\geq 0$.

%

Let us now mention the works which have studied spectral properties of the full linearized operator $\Lambda = \LL - v \cdot \nabla_x $. Mouhot and Neumann \cite{mouNeu} prove explicit coercivity estimates for hard and moderately soft potentials ($\gamma \in [-2,1]$) in $H^{\ell}_{x,v}(\mu^{-1/2})$ for $\ell \ge 1$, using the known spectral estimate for $\LL$ in \eqref{eq:GapDLandauMS}.
It is worth mentioning that the third author has obtained in \cite{Wu1} an exponential decay to equilibrium for the full linearized equation in $L^2_{x,v}(\mu^{-1/2})$ by a different method, and the decay rate depends on the size of the domain.
Let us summarize results that we will use in the remainder of the paper in the following theorem.

\begin{theo}[\cite{mouNeu}]\label{theo:gapE}
Consider $\ell_0 \geq 1$ and $E:=H^{\ell_0}_{x,v}(\mu^{-1/2})$. Then, there exists a constructive constant $\lambda_0>0$ (spectral gap) such that $\Lambda$ satisfies on $E$:
\begin{enumerate}[(i)]
\item the spectrum $\Sigma(\Lambda) \subset \left\{ z \in \C : \Re e z \leq -\lambda_0 \right\} \cup \left\{0\right\}$;
\item the null space $N(\Lambda)$ is given by
\beqn \label{eq:kerlambda}
N(\Lambda) = \mathrm {Span} \{ \mu , v_1 \mu, v_2 \mu, v_3 \mu , |v|^2 \mu   \},
\eeqn
and the projection $\Pi_0$ onto $N(\Lambda)$ by
\beqn\label{eq:Pi0}
\bal
\Pi_0 f &= \left(\int_{\T^3 \times \R^3} f \, dx \, dv \right) \mu + \sum_{i=1}^3 \left(\int_{\T^3 \times \R^3} v_i f  \, dx \, dv \right) v_i \mu \\
&\quad 
+ \left(\int_{\T^3 \times \R^3} \frac{|v|^2-3}{6}\, f  \, dx \, dv\right) \, \frac{(|v|^2-3)}{6} \, \mu;
\eal
\eeqn
\item $\Lambda$ is the generator of a strongly continuous semigroup $S_\Lambda(t)$ that satisfies
\beqn \label{eq:spectralgapE}
\forall \, t \ge 0,  \, \forall\, f \in E, \quad
\| S_{\Lambda}(t) f - \Pi_0 f \|_{E} \le  e^{-\lambda_0 t} \| f - \Pi_0 f \|_{E}.
\eeqn
\end{enumerate}
\end{theo}

\smallskip
To prove Theorem \ref{theo:main2}, our strategy follows the one initiated by Mouhot in \cite{Mouhot} for the homogeneous Boltzmann equation for hard potentials with cut-off.
The latter theory has then been developed and extend in an abstract setting by Gualdani, Mischler and Mouhot~\cite{GMM}, and Mischler and Mouhot~\cite{MM}.
They have applied it to Fokker-Planck equations and the spatially inhomogeneous Boltzmann equation for hard spheres. This strategy has also been used for the homogeneous Landau equation for hard and moderately soft potentials by the first author in \cite{KC1,KC2} and by the second author for the fractional Fokker-Planck equation and the homogeneous Boltzmann equation for hard potentials without cut-off in \cite{IT1,IT2} (see also \cite{MiSch} for related works).

Let us describe in more details this strategy. We want to apply the abstract theorem of enlargement of the space of semigroup decay from \cite{GMM,MM} to our linearized operator $\Lambda$. We shall deduce the spectral/semigroup estimates of Theorem~\ref{theo:main2} on ``large spaces'' $\EE$ using the already known spectral gap estimates for $\Lambda$ on $H^\ell_{x,v}(\mu^{-1/2})$, for $\ell \ge 1$, described in Theorem~\ref{theo:gapE}. Roughly speaking,
to do that, we have to find a splitting of $\Lambda$ into two operators $\Lambda = \AA + \BB$ which satisfy some properties. The first part $\AA$ has to be bounded, the second one $\BB$ has to have some dissipativity properties, and also the semigroup $(\AA \SS_\BB(t))$ is required to have some regularization properties.

\bigskip

We end this introduction by describing the organization of the paper. In Section~\ref{sec:lin} we consider the linearized equation and prove a precise version of Theorem~\ref{theo:main2}. In Section~\ref{sec:NL} we come back to the nonlinear equation and prove our main result Theorem~\ref{main1}.

\bigskip
\noindent
{\bf Acknowledgements.}
The authors would like to thank St\'{e}phane Mischler for his help and his suggestions.
The first author is supported by the Fondation Math\'ematique Jacques Hadamard.
The second author has been partially supported by the fellowship l'Or\'{e}al-UNESCO \textit{For Women in Science}.
The third author is supported by the Ministry of Science and Technology (Taiwan) under the grant 102-2115-M-017-004-MY2 and National Center for Theoretical Science.

\section{The linearized equation}\label{sec:lin}

\subsection{Functional spaces} \label{subsec:hypm}
Let us now make our assumptions on the different potentials $\gamma$ and weight functions $m=m(v)$:


\bigskip
\noindent{\bf (H1) Hard potentials $\gamma\in(0,1]$.} For $p \in [1,\infty]$ we consider the following cases
\begin{enumerate}[$\qquad(i)$]

\item \emph{Polynomial weight}: let $m = \la v \ra^k$ with $k > \gamma + 2 + 3(1-1/p)$, and define the abscissa $\lambda_{m,p} := \infty$.

\smallskip

\item \emph{Stretched exponential weight}: let $m = e^{r\la v \ra^s}$ with $r>0$ and $s \in (0,2)$, and define the abscissa $\lambda_{m,p} := \infty$.

\smallskip

\item \emph{Exponential weight}: let $m = e^{r\la v \ra^2}$ with $r \in (0,1/2)$ and define the abscissa $\lambda_{m,p} := \infty$.

\end{enumerate}

\bigskip
\noindent{\bf (H2) Maxwellian molecules $\gamma = 0$.} For $p \in [1,\infty]$ we consider the following cases
\begin{enumerate}[$\qquad(i)$]

\item \emph{Polynomial weight}: let $m = \la v \ra^k$ with $k > \gamma + 2 + 3(1-1/p)$, and define the abscissa $\lambda_{m,p} := 2[k - (\gamma+3)(1-1/p)]$.

\smallskip

\item \emph{Stretched exponential weight}: let $m = e^{r\la v \ra^s}$ with $r>0$ and $s \in (0,2)$, and define the abscissa $\lambda_{m,p} := \infty$.

\smallskip

\item \emph{Exponential weight}: let $m = e^{r\la v \ra^2}$ with $r \in (0,1/2)$ and define the abscissa $\lambda_{m,p} := \infty$.

\end{enumerate}

\bigskip
\noindent{\bf (H3) Moderately soft potentials $\gamma\in [-2,0)$.} For $p \in [1,\infty]$ we consider the following cases
\begin{enumerate}[$\qquad(i)$]

\item \emph{Stretched exponential weight for $\gamma \in (-2,0)$}: let $m = e^{r\la v \ra^s}$ with $r>0$, $s \in (0,2)$ and $s + \gamma >0$, and define the abscissa $\lambda_{m,p} := \infty$.

\smallskip

\item \emph{Exponential weight for $\gamma \in (-2,0)$}: let $m = e^{r\la v \ra^2}$ with $r \in (0,1/2)$ and define the abscissa $\lambda_{m,p} := \infty$.

\smallskip


\item \emph{Exponential weight for $\gamma=-2$}: let $m = e^{r\la v \ra^2}$ with $r \in (0,1/2)$, and define the abscissa $\lambda_{m,p} := 4r(1-2r)$.

\end{enumerate}

\bigskip

Under these hypothesis, we shall use the following notation for the functional spaces:
\beqn\label{def:E}
E := H^{\ell_0}_{x,v} (\mu^{-1/2}), \quad \ell_0 \ge 1,
\eeqn
in which space we already know that the linearized operator $\Lambda$ has a spectral gap (Theorem~\ref{theo:gapE}), and also, under hypotheses {\bf (H1)}, {\bf (H2)} or {\bf (H3)},
\beqn\label{def:EE}
\EE :=
\left\{
\bal
L^p_{x,v}(m) ,& \quad \forall\, p \in [1,\infty]; \\
W^{n,p}_x W^{\ell,p}_v (m),& \quad \forall\, p \in [1,2], \,  n \in \N^*, \, \ell \in \N ;
\eal
\right.
\eeqn
and for each space we define the associated abscissa $\lambda_\EE = \lambda_{m,p}$.

\bigskip
The main result of this section, which is a precise version of Theorem~\ref{theo:main2}, reads
\begin{thm}\label{thm:extension}
Consider hypothesis {\bf (H1)}, {\bf (H2)} or {\bf (H3)} with some weight $m$, and let $\EE$ be one of the admissible spaces defined in \eqref{def:EE}.

Then, for any $\lambda < \lambda_{\EE}$ and any $\lambda_1 \le \min\{ \lambda_0, \lambda \}$, where we recall that $\lambda_0>0$ is the spectral gap of $\Lambda$ on $E$ (see \eqref{eq:spectralgapE}), there is a constructive constant $C>0$ such that
the operator $\Lambda$ satisfies on $\EE$:

\begin{enumerate}[(i)]

\item
$\Sigma (\Lambda) \subset \{ z \in \C \mid \Re z \le - \lambda_1 \} \cup \{ 0 \}$;

\item the null-space $N(\Lambda)$ is given by \eqref{eq:kerlambda} and the projection $\Pi_0$ onto $N(\Lambda)$ by \eqref{eq:Pi0};

\item $\Lambda$ is the generator of a strongly continuous semigroup $\SS_{\Lambda}(t)$ that verifies
$$
\forall \, t \ge 0, \;   \forall\, f \in \EE, \quad
\| \SS_{\Lambda}(t) f - \Pi_0 f \|_{\EE} \le C \, e^{-\lambda_1 t} \, \| f - \Pi_0 f \|_{\EE}.
$$

\end{enumerate}

\end{thm}

\begin{rem} (1) Observe that:

\begin{itemize}

\item Cases {\bf (H1)}, {\bf (H2)-(ii)-(iii)} or {\bf (H3)-(i)-(ii)}: we can recover the optimal estimate $\lambda_1 = \lambda_0$ since $\lambda_{m,p} = + \infty$.

\item Case {\bf (H2)-(i)}: in this case we have $m = \la v \ra^k$, and we can recover the optimal estimate $\lambda_1 = \lambda_0$ if $k>0$ is large enough such that $\lambda_{m,p} = 2k - 6(1-1/p) > \lambda_0$. Otherwise, we obtain $\lambda_1 < 2k - 6(1-1/p)$.

\item Case {\bf (H3)-(iii)}: in this case we have $\gamma=-2$, $m = e^{r \la v \ra^2}$ and $\lambda_{m,p} = 4r(1-2r)$ and the condition $0 < r < 1/2$.

\end{itemize}

(2) This theorem also holds for other choices of space, namely for a space $\EE$ that is an interpolation space of two admissible spaces $\EE_1$ and $\EE_2$ in \eqref{def:EE}. 


\Black

\end{rem}

The proof of Theorem~\ref{thm:extension} uses the fact that the properties (i)-(ii)-(iii) with $\lambda_1=\lambda_0$ hold on the small space $E$ (Theorem~\ref{theo:gapE}) and the strategy described in section~\ref{sec:lin-intro}.

\medskip

In a similar way we shall obtain Theorem~\ref{thm:extension}, we shall also deduce a regularity estimate on the semigroup $\SS_\Lambda$ that will be of crucial importance in the study of the nonlinear equation in Section~\ref{sec:NL}. For the sake of simplicity, and because it is the case that we shall use for the nonlinear equation, we only present this result for the particular case of $p=2$ and $\ell=0$ in \eqref{def:EE}.

Define the space $H^{1}_{v,*}(m)$, associated to the norm
$$
\| f \|_{H^{1}_{v,*}(m)}^2 = \| f \|_{L^2_{x,v}(m \la v \ra^{(\gamma+\sigma)/2})}^2
+  \| P_v \nabla f \|_{L^2(m \la v \ra^{\gamma/2})}^2
+ \| (I-P_v) \nabla f \|_{L^2(m \la v \ra^{(\gamma+2)/2})}^2  ,
$$
as well as the space $H^{n}_x ( H^{1}_{v,*}(m))$, with $ n \in \N$, by
\beqn\label{eq:HnxH1v*}
\| f \|_{H^{n}_x ( H^{1}_{v,*}(m))}^2
= \sum_{0\le j\le n} \| \nabla^j_x f \|_{L^2_x ( H^{1}_{v,*}(m)) }^2
= \sum_{0\le j\le n} \int_{\T^3_x} \| \nabla^j_x f \|_{H^{1}_{v,*}(m) }^2.
\eeqn
We hence define the negative Sobolev space $H^n_x (H^{-1}_{v,*} (m))$ by duality in the following way
\beqn\label{eq:HnxH-1v*}
\| f \|_{H^n_x (H^{-1}_{v,*} (m))} := \sup_{\| \phi \|_{H^n_x (H^{1}_{v,*} (m))} \le 1} \la f , \phi \ra_{H^n_x L^2_v (m)}.
\eeqn

\begin{thm}\label{thm:SL-regularity}
Consider hypothesis {\bf (H1)}, {\bf (H2)} or {\bf (H3)} with some weight $m$. Let $\EE = H^n_x L^2_v (m)$ and $\EE_{-1} = H^n_x (H^{-1}_{v,*} (m))$. 
Then, for any $\lambda < \lambda_1 $, the following regularity estimate holds
\beqn\label{eq:SLambda-reg}
\int_0^\infty e^{ 2 \lambda t} \, \| \SS_{\Lambda}(t) (I - \Pi_0) f \|_{\EE  }^2 \, dt \le C \| (I - \Pi_0) f \|_{\EE_{-1}}^2,
\eeqn
for some constant $C>0$.

\end{thm}

\Black


\subsection{Splitting of the linearized operator}
We decompose the linearized Landau operator $\LL$ defined
in \eqref{eq:oplandaulin} as $\LL =  \AA_0 + \BB_0$, where we define
\beqn\label{eq:A0B0}
\AA_0 f := (a_{ij}* f)\partial_{ij}\mu - (c * f)\mu,
\qquad
\BB_0 f := (a_{ij}* \mu)\partial_{ij}f - (c * \mu)f.
\eeqn
Consider a smooth positive function $\chi \in \CC^\infty_c(\R^3_v)$ such that $0\leq \chi(v) \leq 1$, $\chi(v) \equiv 1$ for $|v|\leq 1$ and $\chi(v) \equiv 0$ for $|v|>2$. For any $R\geq 1$ we define $\chi_R(v) := \chi(R^{-1}v)$ and in the sequel we shall consider the function $M\chi_R$, for some constant $M>0$.

Then, we make the final decomposition of the operator $\Lambda$ as $\Lambda = \AA + \BB$ with
\beqn\label{eq:AB}
\AA  := \AA_0 + M \chi_R ,
\qquad
\BB  := \BB_0 - v\cdot \nabla_x - M\chi_R,
\eeqn
where $M>0$ and $R>0$ will be chosen later (see Lemma~\ref{lem:hypo}).

\subsection{Preliminaries}

We have the following results concerning the matrix $\bar a_{ij}(v)$.
\begin{lem}\label{lem:bar-aij}
The following properties hold:

\begin{enumerate}[(a)]

\item The matrix $\bar a(v)$ has a simple eigenvalue $\ell_1(v)>0$ associated with the eigenvector $v$ and a double eigenvalue $\ell_2(v)>0$ associated with the eigenspace $v^{\perp}$. Moreover,
when $|v|\to +\infty$ we have
$$
\ell_1(v) \sim  2 \la v \ra^\gamma \quad\text{and}\quad
\ell_2(v) \sim  \la v \ra^{\gamma+2} .
$$

\item The function $\bar a_{ij}$ is smooth, for any multi-index $\beta\in \N^3$
$$
|\partial^\beta \bar a_{ij}(v)|
\leq C_\beta \la v \ra^{\gamma+2-|\beta|}
$$
and
$$
\bal
\bar a_{ij}(v) \xi_i \xi_j &= \ell_1(v) |P_v \xi|^2 + \ell_2(v)|(I-P_v)\xi|^2 \\
&\ge c_0 \big\{ \la v \ra^{\gamma} |P_v \xi|^2 + \la v \ra^{\gamma+2}|(I-P_v)\xi|^2  \big\},
\eal
$$
for some constant $c_0>0$ and where $P_v$ is the projection on $v$, i.e. $ P_v \xi_i = \left( \xi \cdot \frac{v}{|v|} \right) \frac{v_i}{|v|}$.

\item We have
$$
\bar a_{ii}(v) = \mathrm{tr} (\bar a(v)) = \ell_1(v) + 2 \ell_2(v)
= 2 \int_{\R^3} |v-v_*|^{\gamma+2} \mu(v_*)\, dv_*
\qquad\text{and}\qquad
\bar b_i(v) = - \ell_1(v)\, v_i.
$$

\item If $|v|>1$, we have
$$
|\partial^\beta \ell_{1}(v)|\leq C_\beta \la v \ra^{\gamma-|\beta|}\qquad\text{and}\qquad |\partial^\beta \ell_{2}(v)|\leq C_\beta \la v \ra^{\gamma+2-|\beta|}.
$$
\end{enumerate}

\end{lem}

\begin{proof}
We just give the proof of item $(d)$ since $(a)$ comes from \cite[Propositions 2.3 and 2.4, Corollary 2.5]{DL}, $(b)$ is \cite[Lemma 3]{Guo} and $(c)$ is evident.
For item $(d)$, the estimate of $|\partial^\beta \ell_{2}(v)|$ directly comes from $(a)$ and \cite[Lemma 2]{Guo}. For $\ell_{1}(v)$, using $(b)$ and $(c)$,
$$
\partial_{v} \bar b_{i}(v)=\partial_{v} \big(-\ell_{1}(v)v_{i}\big)\,,
$$
and hence
$$
|\partial_{v} \ell_{1}(v)||v|\leq C \big( |\ell_{1}(v)|+|\partial_{v} \bar b_{i}(v)|   \big) \leq
C \la v \ra^{\gamma}\,,
$$
note that $|v|>1$, we thus have
$$
|\partial_{v} \ell_{1}(v)|\leq
C |v|^{-1} \la v \ra^{\gamma}\leq C\la v \ra^{\gamma-1}\,.
$$
The high order estimate is similar and hence we omit the details.
\end{proof}

The following elementary lemma will be useful in the sequel (see \cite[Lemma 2.5]{KC1} and \cite[Lemma 2.3]{KC2}).
\begin{lem}\label{lem:Jalpha}
Let $J_\alpha(v) := \int_{\R^3} |v-w|^\alpha \mu(w)\, dw$, for $-3 < \alpha \leq 3$. Then it holds:
\begin{enumerate}[(a)]

\item If $2 < \alpha \le 3$, then $J_\alpha(v) \le |v|^\alpha + C_\alpha |v|^{\alpha/2}+ C_\alpha$, for some constant $C_\alpha >0$.

\item If $0 \le \alpha \le 2$, then $J_\alpha(v) \le |v|^\alpha + C_\alpha$, for some constant $C_\alpha >0$.

\item If $-3 < \alpha < 0$, then $J_\alpha (v) \le C \la v \ra^\alpha$ for some constant $C>0$.

\end{enumerate}

\end{lem}

We define the function $\varphi_{m,p}$ as
\beqn\label{eq:def-varphi}
\varphi_{m,p} (v) :=
 \bar a_{ij}(v) \frac{\partial_{ij} m }{m}
+(p-1) \bar a_{ij}(v) \frac{\partial_i m}{m} \frac{\partial_j m}{m}
+ 2 \bar b_i (v) \frac{\partial_i m}{m}
+ \left( \frac1p - 1 \right) \bar c (v),
\eeqn
and also the function $\tilde\varphi_{m,p}$ given by
\beqn\label{eq:def-tildevarphi}
\bal
\tilde\varphi_{m,p} (v) &:=
\left( \frac{2}{p}-1\right)\bar a_{ij}(v) \frac{\partial_{ij} m }{m}
+\left(2-\frac{2}{p}\right) \bar a_{ij}(v) \frac{\partial_i m}{m} \frac{\partial_j m}{m} \\
&\quad
+ \frac{2}{p} \, \bar b_i (v) \frac{\partial_i m}{m} + \left( \frac1p - 1 \right) \bar c (v),
\eal
\eeqn
and hereafter, in order to treat both weight functions at the same time, we remind the notation: $\sigma = 0$ when $m= \la v \ra^k$ and $\sigma = s$ when $m = e^{r \la v \ra^s}$.

We prove the following result concerning $\varphi_{m,p}$ and $\tilde\varphi_{m,p}$.

\begin{lem}\label{lem:varphi}
Consider {\bf (H1)}, {\bf (H2)} or {\bf (H3)}, and let $\varphi_{m,p}$ and $\tilde\varphi_{m,p}$ be defined in \eqref{eq:def-varphi} and~\eqref{eq:def-tildevarphi} respectively. Then we have:

\smallskip

\noindent
{$\bullet$} Assume $\sigma \in[0,2)$:

(1) For all positive $\lambda < \lambda_{m,p}$ and $\delta \in (0,\lambda_{m,p} - \lambda)$ we can choose $M$ and $R$ large enough such that
$$
\varphi_{m,p} (v) - M \chi_R (v) \leq -\lambda  - \delta \la v \ra^{\gamma+\sigma}.
$$
$$
\tilde\varphi_{m,p} (v) - M \chi_R (v) \leq -\lambda  - \delta \la v \ra^{\gamma+\sigma}.
$$

\smallskip

(2) For all positive $\lambda < \lambda_{m,p}$ and $\delta \in (0,\lambda_{m,p} - \lambda)$ we can choose $M$ and $R$ large enough such that
$$
\varphi_{m,p} (v) - M \chi_R (v) + M \partial_j \chi_R(v) \leq -\lambda  - \delta \la v \ra^{\gamma+\sigma}.
$$
$$
\tilde\varphi_{m,p} (v) - M \chi_R (v) + M \partial_j \chi_R(v) \leq -\lambda  - \delta \la v \ra^{\gamma+\sigma}.
$$

\smallskip

\noindent
{$\bullet$} Assume $\sigma =2$: The same conclusion as before holds for $\tilde\varphi_{m,p}$. Moreover, concerning $\varphi_{m,p}$, the previous estimates also hold if we restrict $r \in (0,1/(2p))$ in assumptions {\bf (H1)-(iii), (H2)-(iii), (H3)-(ii)}, and also modifying the value of the abscissa $\lambda_{m,p} = 4r(1-2rp)$ in {\bf (H3)-(iii)}.

\end{lem}

\begin{proof}[Proof of Lemma \ref{lem:varphi}]

{\noindent \it Step 1. Polynomial weight.}
Consider $m= \la v \ra^k$ under hypothesis {\bf (H1)} or {\bf (H2)}. On the one hand, we have
$$
\bal
\frac{\partial_i m}{m} = k v_i \la v \ra^{-2}, \qquad
\frac{\partial_i m}{m}\, \frac{\partial_j m}{m} = k^2 v_i v_j \la v \ra^{-4},\\
\frac{\partial_{ij} m}{m} = \delta_{ij} \,  k \la v \ra^{-2} + k(k-2) v_i v_j \la v \ra^{-4}.
\eal
$$
Hence, from definitions \eqref{eq:bc}-\eqref{eq:barabc} and Lemma~\ref{lem:bar-aij} we obtain
$$
\bal
\bar a_{ij} \,\frac{\partial_{ij} m}{m} = (\delta_{ij}\bar a_{ij}) \,  k \la v \ra^{-2} +  (\bar a_{ij} v_i v_j)  \, k(k-2)\la v \ra^{-4}
= \bar a_{ii} \,  k \la v \ra^{-2} +  \ell_1(v)  \, k(k-2) |v|^2 \la v \ra^{-4},
\eal
$$
where we recall that the eigenvalue $\ell_1(v) >0$ is defined in Lemma~\ref{lem:bar-aij}.
Moreover, arguing exactly as above we obtain
$$
\bal
\bar a_{ij} \,\frac{\partial_i m}{m}\, \frac{\partial_j m}{m} = (\bar a_{ij} v_i v_j)  \, k^2\la v \ra^{-4}
=   \ell_1(v)  \, k^2 |v|^2 \la v \ra^{-4}
\eal
$$
and also, using the fact that $\bar b_{i}(v) = - \ell_1(v) v_i$ from Lemma~\ref{lem:bar-aij},
$$
\bal
\bar b_{i} \,\frac{\partial_i m}{m} =  - \ell_1(v) v_i  \, k v_i  \la v \ra^{-2}
=   - \ell_1(v)  \, k |v|^2 \la v \ra^{-2}.
\eal
$$
On the other hand, from item (c) of Lemma~\ref{lem:bar-aij} and definitions \eqref{eq:bc}-\eqref{eq:barabc} we obtain that
$$
\bar a_{ii}(v) = \ell_1(v) + 2 \ell_2(v)
\quad\text{and}\quad
\bar c(v) = -2(\gamma+3) J_{\gamma}(v),
$$
where $J_\alpha$ is defined in Lemma~\ref{lem:Jalpha}.
It follows that
\beqn\label{eq:phi1}
\bal
\varphi_{m,p}(v)
&=
 2k \ell_2(v)(v) \la v \ra^{-2}
+  k \ell_1(v) \la v \ra^{-2}
+  k(k-2) \, \ell_1(v) \, |v|^2 \la v \ra^{-4}\\
&\quad
+(p-1)  k^2  \, \ell_1(v) \, |v|^2 \la v \ra^{-4}
 - 2   k \, \ell_1(v) \, |v|^2 \la v \ra^{-2}
+2(\gamma+3)\left(1-\frac1p \right)  J_{\gamma}(v).
\eal
\eeqn
Since $\ell_1(v) \sim 2 \la v \ra^\gamma$, $\ell_2(v) \sim \la v \ra^{\gamma+2}$ and $\ell_1(v) |v|^2 \sim 2 \ell_2(v)$ when $|v|\to +\infty$ thanks to Lemma~\ref{lem:bar-aij}, and also $J_\gamma(v) \sim \la v\ra^\gamma$ from Lemma~\ref{lem:Jalpha} (since in this case we have $\gamma \ge 0$), the dominant terms in \eqref{eq:phi1} are the first, fifth and sixth ones, all of order $\la v \ra^{\gamma}$.
Then we obtain
\beqn\label{cas:poly1}
\bal
\limsup_{|v|\to + \infty} \varphi_{m,p}(v)
&\le  -2\left[ k - (\gamma+3)(1-1/p)\right] \la v\ra^\gamma,
\eal
\eeqn
and recall that $k > (\gamma+3)(1 - 1/p)$.
Doing the same kind of computations, we obtain the same asymptotic for $\tilde\varphi_{m,p}$,
\beqn \label{cas:poly1bis}
\limsup_{|v| \to + \infty} \tilde\varphi_{m,p} (v) \le -2 [k - (\gamma+3)(1-1/p)] \la v \ra^{\gamma}.
\eeqn

\medskip
{\noindent \it Step 2. Stretched exponential weight.}
We consider now $m=\exp(r\la v \ra^s)$ satisfying {\bf (H1)}, {\bf (H2)} or {\bf (H3)}. In this case we have
$$
\bal
& \frac{\partial_i m}{m} = r s v_i \la v\ra^{s-2},
\quad
\frac{\partial_i m}{m}\frac{\partial_j m}{m} = r^2 s^2 v_i v_j \la v\ra^{2s-4}, \\
&
\frac{\partial_{ij} m}{m} = r s \la v\ra^{s-2}  \delta_{ij}
+ r s (s-2) v_i v_j \la v\ra^{s-4}  + r^2 s^2 v_i v_j \la v\ra^{2s-4} .
\eal
$$
Then we obtain
\beqn\label{eq:varphiexp}
\bal
\varphi_{m,p}(v) &=
 2 r s \,  \ell_{2}(v)  \la v\ra^{s-2}
+   r s \,  \ell_{1}(v)  \la v\ra^{s-2}
+ r s(s-2) \,  \ell_1(v)   |v|^2\la v\ra^{s-4}\\
&\quad
+ p r^2 s^2 \,  \ell_1(v)   |v|^2\la v\ra^{2s-4}
 - 2 r s \,  \ell_1(v) |v|^2\la v\ra^{s-2}
+2(\gamma+3)\left(1-\frac{1}{p}  \right)  \,  J_{\gamma}(v)
\eal
\eeqn
In the case $0< s <2$, arguing as in step 1, the dominant terms in \eqref{eq:varphiexp} when $|v| \to + \infty$ are the first and fifth one, both of order $\la v\ra^{\gamma+s}$. Then we obtain
\beqn\label{cas:exp1}
\limsup_{|v|\to + \infty} \varphi_{m,p}(v) \le -2 r s \la v\ra^{\gamma+s},
\eeqn
and recall that $s+\gamma>0$.
In the same way we obtain
\beqn \label{cas:exp1bis}
\limsup_{|v| \to + \infty} \tilde\varphi_{m,p} (v) \le -2 rs \la v \ra^{\gamma+s}.
\eeqn

In the case $s=2$, the dominant terms in \eqref{eq:varphiexp} when $|v| \to + \infty$ are the first, fourth and fifth ones, all of order $\la v \ra^{\gamma+2}$. Hence we get
\beqn\label{cas:exp2}
\limsup_{|v|\to + \infty} \varphi_{m,p}(v) \le -4r(1-2pr) \la v \ra^{\gamma+2}.
\eeqn
However, a similar computation gives
\beqn \label{cas:exp2bis}
\limsup_{|v| \to + \infty} \tilde\varphi_{m,p} (v) \le - 4r(1-2r) \la v \ra^{\gamma+2},
\eeqn
which is better than the asymptotic of $\varphi_{m,p}$.
Thus we need the condition $r < 1/2$ for $\tilde\varphi_{m,p}$ (which is better than the condition $r < 1/(2p)$ for $\varphi_{m,p}$).

\medskip
\noindent
{\it Step 3. Conclusion.}
Finally, thanks to the asymptotic behaviour in \eqref{cas:poly1}, \eqref{cas:exp1} and \eqref{cas:exp2}, for any $\lambda < \lambda_{m,p}$ we can choose $M$ and $R$ large enough such that $\varphi_{m,p}(v) - M \chi_R(v) \leq -\lambda - \delta \la v \ra^{\gamma+\sigma}$ for some $\delta>0$ small enough, which gives us point (1) of the lemma.

For the point (2) we use $\partial_j \chi_R (v) = R^{-1} \partial_j \chi(v/R)$ and write
$$
\varphi_{m,p}(v) - M \chi_R(v) + M \partial_j \chi_R(v)
\le \varphi_{m,p}(v) - M \chi_R(v) + M\frac{C_\chi}{R}\, {\bf 1}_{R \le |v| \le 2R} =: \Phi(v).
$$
We fix some $\bar\lambda \in (\lambda  , \lambda_{m,p} )$. First we choose $R_1$ large enough such that, for all $|v| \ge  R_1$, we have
$$
\varphi_{m,p}(v) + \delta \la v \ra^{\gamma+\sigma} \le - \bar \lambda
$$
for some $\delta>0$ small enough, which implies that, for any $|v| \ge 2 R_1$,
$$
\Phi(v)  + \delta \la v \ra^{\gamma+\sigma} = \varphi_{m,p}(v) + \delta \la v \ra^{\gamma+\sigma} \le - \bar\lambda.
$$
Then we choose $M>0$ large enough such that, for all $|v| \le  R_1$,
$$
\Phi(v) + \delta \la v \ra^{\gamma+\sigma} = \varphi_{m,p}(v) + \delta \la v \ra^{\gamma+\sigma} - M \chi_{R_1}(v) \le - \bar\lambda.
$$
Finally, we choose $R \ge R_1$ large enough such that, for any $R \le |v| \le 2 R$,
$$
\Phi(v) + \delta \la v \ra^{\gamma+\sigma} \le \varphi_{m,p}(v)+ \delta \la v \ra^{\gamma+\sigma} + M \frac{C_\chi}{R} \le - \bar\lambda + M \frac{C_\chi}{R}
\le - \lambda,
$$
and we easily observe that now for $R_1 \le |v| \le R$ we have
$$
\Phi(v) + \delta \la v \ra^{\gamma+\sigma} = \varphi_{m,p}(v) + \delta \la v \ra^{\gamma+\sigma} - M \chi_R (v) \le - \bar \lambda - M \le - \lambda,
$$
which concludes the proof for $\varphi_{m,p}$.
Concerning $\tilde \varphi_{m,p}$, in the same way, inequalities \eqref{cas:poly1bis}, \eqref{cas:exp1bis} and \eqref{cas:exp2bis} yield the result.
\end{proof}

\subsection{Hypodissipativity}
\label{ssec:hypo}
In this subsection we prove hypodissipativity properties for the operator $\BB$ on 
the admissible spaces $\EE$ defined in \eqref{def:EE}.

Hereafter we define the space $W^{1,p}_{v,*}(m)$, with $1 < p < \infty$, associated to the norm
$$
\| f \|_{W^{1,p}_{v,*}(m)}^p = \| f \|_{L^p_{x,v}(m \la v \ra^{(\gamma+\sigma)/p})}^p
+  \| P_v \nabla f^{p/2} \|_{L^2(m^{p/2} \la v \ra^{\gamma/2})}^2
+ \| (I-P_v) \nabla f^{p/2} \|_{L^2(m^{p/2} \la v \ra^{(\gamma+2)/2})}^2  ,
$$
as well as the space $W^{n,p}_x (W^{1,p}_{v,*}(m))$, with $ n \in \N$, by
\beqn\label{eq:WnpxW1pv*}
\| f \|_{W^{n,p}_x (W^{1,p}_{v,*}(m))}^p
= \sum_{0\le j\le n} \| \nabla^j_x f \|_{L^p_x (W^{1,p}_{v,*}(m)) }^p
= \sum_{0\le j\le n} \int_{\T^3_x} \| \nabla^j_x f \|_{W^{1,p}_{v,*}(m) }^p.
\eeqn

\begin{lem}\label{lem:hypo}
Consider hypothesis {\bf (H1)}, {\bf (H2)} or {\bf (H3)}. Let $p\in[1,+\infty]$ and $n \in \N$. Then, for any $\lambda < \lambda_{m,p}$ we can choose $M>0$ and $R>0$ large enough such that the operator $(\BB+\lambda)$ is dissipative in $W^{n,p}_x L^p_{v}(m)$, in the sense that
\beqn\label{eq:SBdecayLp}
\forall\, t \ge0, \quad \| \SS_\BB(t) \|_{\BBB(W^{n,p}_x L^p_{v}(m))} \le C e^{-\lambda t}.
\eeqn
Moreover there hold: if $p=1$
\beqn\label{eq:SBregL1t}
\int_0^\infty e^{\lambda  t} \, \| \SS_{\BB}(t) \|_{ \BBB (W^{n,1}_x L^1_{v}(m) , W^{n,1}_x L^{1}_{v}(m \la v \ra^{\gamma+\sigma})) )} \, dt < \infty,
\eeqn
and if $1 < p < \infty$
\beqn\label{eq:SBregLpt}
\int_0^\infty e^{\lambda p t} \, \| \SS_{\BB}(t) \|_{ \BBB ( W^{n,p}_x L^p_{v}(m) , W^{n,p}_x (W^{1,p}_{v,*}(m)))}^p \, dt < \infty,
\eeqn
\end{lem}

\Black

\begin{proof}[Proof of Lemma~\ref{lem:hypo}]
We only consider the case $n=0$, the general case being treated in the sama way since $\nabla_x$ commutes with $\BB$.

Let us denote $\Phi'(z) = |z|^{p-1}\mathrm{sign}(z)$ and consider the equation
$$
\partial_t f = \BB f = \BB_0 f - v\cdot \nabla_x f -M\chi_R f.
$$
For all $p \in [1,+\infty)$, we have
$$
\bal
\frac1p\frac{d}{dt} \| f \|_{L^p_{x,v}(m)} ^p
&=  \int (\BB f) \Phi'(f)\,m^p .
\eal
$$
From \eqref{eq:oplandau} and \eqref{eq:A0B0}, last integral is equal to
$$
\bal
&\int \bar a_{ij}(v) \partial_{ij} f(x,v) \Phi'(f) m^p
-\int \bar c(v) f(x,v) \Phi'(f) m^p  \\
&\quad
-\int v \cdot \nabla_x f(x,v) \Phi'(f) m^p
-\int M\chi_R(v) f(x,v) \Phi'(f) m^p \\
&=: T_1 + T_2 + T_3 + T_4.
\eal
$$
The term $T_3$ vanishes thanks to its divergence structure and terms $T_2$ and $T_4$ are easily computed, giving
$$
T_2 = - \int \bar c(v) |f(x,v)|^p m^p
\quad\text{and}\quad
T_4=- \int M\chi_R(v) |f(x,v)|^p m^p .
$$
Let us compute then the term $T_1$. Using that $\partial_{ij}f  \Phi'(f) = p^{-1}\partial_{ij}(|f|^p) - (p-1)\partial_i f \partial_j f |f|^{p-2}$ we obtain
$$
\bal
T_1 &= \frac{1}{p}\int \bar a_{ij}(v) \partial_{ij}(|f|^p) m^p
- (p-1) \int  \bar a_{ij}(v) \partial_i f \partial_j f \,|f|^{p-2}m^p .
\eal
$$
Performing two integrations by parts on the first integral of $T_1$ it yields
$$
\bal
  \int (\BB f) \Phi'(f)\,m^p
&=  - \frac{4}{p^2}(p-1) \int  \bar a_{ij}(v) \, \partial_i (f^{p/2}) \, \partial_j (f^{p/2}) \, m^p \\
&\quad+
\int \left\{  \varphi_{m,p}(v) - M\chi_R(v) \right\}  |f|^p \, m^p ,
\eal
$$
where $\varphi_{m,p}$ is defined in \eqref{eq:def-varphi}.
We can also get, by a similar computation,
$$
\bal
\int (\BB f) \Phi'(f)\,m^p
&=  - \frac{4}{p^2}(p-1) \int  \bar a_{ij}(v) \, \partial_i (m f^{p/2}) \, \partial_j (m f^{p/2})   \, m^{p-2} \\
&\quad+
\int \left\{  \tilde\varphi_{m,p}(v) - M\chi_R(v) \right\}  |f|^p \, m^p.
\eal
$$

Thanks to Lemma~\ref{lem:varphi}, for any $\lambda < \lambda_{m,p}$ we can choose $M$ and $R$ large enough such that $\varphi_{m,p}(v) - M\chi_R(v) \leq -\lambda + \delta \la v \ra^{\gamma+\sigma}$. Hence it follows, using Lemma~\ref{lem:bar-aij}, 
\beqn\label{eq:SSf2}
\bal
\frac1p\frac{d}{dt} \| f \|_{L^p(m)} ^p
&\leq  - c_0 (p-1) \int  \{ \la v \ra^{\gamma} |P_v \nabla_v (f^{p/2})|^2 + \la v \ra^{\gamma+2} |(I-P_v)\nabla_v (f^{p/2})|^2   \} \, m^p \\
&\quad
- \lambda \| f \|_{L^p(m)}^p
- \delta \| \la v \ra^{\frac{\gamma+\sigma}{p}} f \|_{L^p(m)}^p.
\eal
\eeqn
or
\beqn
\bal
\frac1p\frac{d}{dt} \| f \|_{L^p(m)} ^p
&\leq  - c_0 (p-1) \int  \{ \la v \ra^{\gamma} |P_v \nabla_v (m f^{p/2})|^2 + \la v \ra^{\gamma+2} |(I-P_v)\nabla_v (m f^{p/2} )|^2   \}  \, m^{p-2} \\
&\quad
- \lambda \| f \|_{L^p(m)}^p
- \delta \| \la v \ra^{\frac{\gamma+\sigma}{p}} f \|_{L^p(m)}^p,
\eal
\eeqn
from which we easily obtain \eqref{eq:SBdecayLp} for any $1 \le p < \infty$. For $p=\infty$, let $g=mf$, it is easy to check that $g$ satisfies the equation
$$
\partial_{t}g+v\cdot\nabla_{x}g=\bar{a}_{ij}(v)\partial_{ij}g -2\bar{a}_{ij}(v)\frac{\partial_{i}m}{m}\partial_{j}g+\tilde\varphi_{m,\infty} (v)g- M\chi_R(v)g\,,
$$
by the standard maximum principle argument (for example, see \cite{Wu2}), we have
$$
\| \SS_\BB(t) f\|_{L^\infty_{x,v}(m)} \leq e^{-\lambda t} \| f \|_{L^\infty_{x,v}(m)}.
$$
This completes the proof of \eqref{eq:SBdecayLp}.

The proof of \eqref{eq:SBregL1t} and \eqref{eq:SBregLpt} follows easily from \eqref{eq:SSf2} by keeping all the terms at the right-hand side and integrating in time.
\end{proof}

Define the operator $\BB_m$ by $\BB_m h := m \BB (m^{-1} h)$, more precisely
\beqn\label{Bm}
\bal
\BB_m h &= \bar a_{ij} \partial_{ij} h - 2\bar a_{ij} \frac{\partial_i m}{m} \partial_j h + \left\{ 2 \bar a_{ij}\frac{\partial_i m}{m}\frac{\partial_j m}{m} - \bar a_{ij} \frac{\partial_{ij} m}{m} - \bar c - M \chi_R  \right\} h - v \cdot \nabla_x h \\
&=: \bar a_{ij} \partial_{ij} h +\beta_j \partial_j h + ( \zeta - M \chi_R  ) h - v \cdot \nabla_x h.
\eal
\eeqn
Observe that if $f$ satisfies $\partial_t f = \BB f$, then $h := mf$ satisfies $\partial_t h = \BB_m h$. 
We then define the operator $\BB^*_m$, the (formal) adjoint operator of $\BB_m$ with respect to the usual scalar product $L^2_{x,v}$, by
\beqn\label{B*m}
\bal
\BB^*_m \phi &= \bar a_{ij} \partial_{ij} \phi + \left\{ 2 \bar b_j + 2 \bar a_{ij} \frac{\partial_i m}{m}   \right\} \partial_j \phi + \left\{\bar a_{ij} \frac{\partial_{ij} m}{m} + 2 \bar b_j \frac{\partial_j m}{m} - M \chi_R  \right\} \phi + v \cdot \nabla_x \phi \\
&=: \bar a_{ij} \partial_{ij} \phi + \beta^*_j \partial_i \phi + (\zeta^* - M\chi_R) \phi + v \cdot \nabla_x \phi.
\eal
\eeqn
Remark that, denoting $h_t := \SS_{\BB_m}(t) h_0$ and $\phi_t := \SS_{\BB^*_m}(t) \phi_0$, 
which verify the equations $\partial_t h_t = \BB_m h_t$ and $\partial_t \phi_t = \BB^*_m \phi_t$, there holds
$$
\la h_t, \phi_0 \ra_{H^n_x L^2_v} = \la h_0, \phi_t \ra_{H^n_x L^2_v}.
$$

\begin{lem}\label{lem:hypo4}
Consider hypothesis {\bf (H1)}, {\bf (H2)} or {\bf (H3)}, and let $n \in \N$.
Then, for any $\lambda < \lambda_{m,2}$, we can choose $M$ and $R$ large enough such that the operator $(\BB^*_m+\lambda)$ is hypo-dissipative in $H^n_x L^2_v$, in the sense that
\beqn\label{eq:SB*mdecay}
\forall\, t \ge0, \quad \| S_{\BB^*_m}(t) \|_{\BBB( H^n_x L^2_v )} \le C e^{-\lambda t}.
\eeqn
Moreover there holds
\beqn\label{eq:SBregH^-1}
\int_0^\infty e^{2 \lambda t} \, \| S_{\BB}(t) \|_{\BBB( H^n_x (H^{-1}_{v,*}(m)) ,H^n_x L^2_v(m) )}^2 \, dt \le \infty,
\eeqn
where we recall that $H^n_x (H^{-1}_{v,*}(m))$ is defined in \eqref{eq:HnxH-1v*}.

\end{lem}

\begin{proof}
We consider the case $n=0$, the others being the same because $\nabla_x$ commutes with $\BB^*_m$.

Let $\partial_t \phi = \BB^*_m \phi$, where we recall that $\BB^*_m$ is defined in \eqref{B*m}. We have
$$
\begin{aligned}
\int (\BB_m^* \phi) \, \phi
&= \int \left( \bar a_{ij} \frac{\partial_{ij} m}{m}  + 2 \bar b_j \frac{\partial_j m}{m}  - M \, \chi_R\right) \, \phi^2 \\
&\quad +  \int \left(  \bar a_{ij} \frac{\partial_j m}{m} + \bar b_i \right) \partial_i (\phi^2)  + \int v \cdot \nabla_x \phi \, \phi + \int \bar a_{ij} \partial_{ij} \phi \, \phi \\
&=: T_1 + T_2 + T_3 + T_4.
\end{aligned}
$$
Performing one integration by parts, we obtain
$$
T_2 =\int \left( - \bar a_{ij} \frac{\partial_{ij} m}{m} + \bar a_{ij} \frac{\partial_i m}{m} \frac{\partial_j m}{m}  - \bar b_j \frac{\partial_j m}{m}   -\bar c\right) \phi^2.
$$
The term $T_3$ gives no contribution thanks to its divergence structure in $x$. Moreover we deal with $T_4$ using that $\partial_{ij} \phi \, \phi = \frac{1}{2} \partial_{ij}(\phi^2) - \partial_i \phi \partial_j \phi$, which implies
$$
T_4 = - \int \bar a_{ij} \partial_i \phi \partial_j \phi + \frac12 \int \bar c \, \phi^2.
$$
Finally, we obtain that
\beqn\label{Bm*phi2}
\bal
\int (\BB_m^* \phi )\, \phi &=
- \int \bar a_{ij} \partial_i \phi \partial_j \phi +
\int \{ \tilde\varphi_{m,2} - M \chi_R \} \, \phi^2 \\
&\le -c_0 \int \left\{ \la v \ra^\gamma |P_v \nabla_v \phi|^2 + \la v \ra^{\gamma+2} |(I-P_v) \nabla_v \phi|^2 \right\} + \int \{ \tilde\varphi_{m,2}-M \chi_R \} \, \phi^2.
\eal
\eeqn
where we recall that $\tilde\varphi_{m,2}$ is defined in \eqref{eq:def-tildevarphi}. 

Thanks to Lemma~\ref{lem:varphi}, for any positive $\lambda < \lambda_{m,2}$ and $\delta \in (0, \lambda_{m,2}- \lambda)$, we can thus find $M,R>0$ large enough such that $ \tilde\varphi_{m,2}(v)-M \chi_R\leq -\lambda - \delta \la v \ra^{\gamma+\sigma}$. We can conclude that
$$
\bal
\frac12 \frac{d}{dt} \| \phi \|_{L^2}^2
&\le - \lambda \| \phi \|_{L^2}^2
- \delta \| \la v \ra^{\frac{\gamma+\sigma}{2}} \phi \|_{L^2}^2 \\
&\quad
- c_0 \left\{ \| \la v \ra^{\frac{\gamma}{2}} P_v \nabla_v \phi \|_{L^2(m)}^2 +
\| \la v \ra^{\frac{\gamma+2}{2}} (I- P_v) \nabla_v \phi \|_{L^2(m)}^2    \right\}.
\eal
$$
From this inequality we easily obtain \eqref{eq:SB*mdecay} and also the regularity estimate
$$
\int_0^\infty e^{2 \lambda t} \, \| \SS_{\BB^*_m} (t) \phi \|_{L^2_x (H^1_{v,*})}^2 \, dt \lesssim \| \phi \|_{L^2_x L^2_v}^2.
$$
Consider now the function $h$ that satisfies $\partial_t h = \BB_m h$. Using that $\la \SS_{\BB_m}(t) h , \phi \ra_{H^n_x L^2_v} = \la h , \SS_{\BB^*_m}(t) \phi \ra_{H^n_x L^2_v}$, this last estimate implies by duality (see \eqref{eq:HnxH-1v*})
$$
\int_0^\infty e^{2 \lambda t} \, \| \SS_{\BB_m} (t) h \|_{L^2_x L^2_v }^2 \, dt \lesssim \| h \|_{L^2_x (H^{-1}_{v,*}  )}^2.
$$
Finally we deduce \eqref{eq:SBregH^-1} by using the fact that $S_{\BB_m}(t) h = m S_{\BB}(t) f$.
\end{proof}

\Black

We now investigate hypodissipative properties of $\BB$ in high-order velocity spaces.

\begin{lem}\label{lem:hypo1}
Consider hypothesis {\bf (H1)}, {\bf (H2)} or {\bf (H3)}, $ \ell \in \N$ and $n \in \N^*$.
Then, for any $\lambda<\lambda_{m,1}$, we can choose $M>0$ and $R>0$ large enough such that the operator $\BB+ \lambda$ is hypo-dissipative in $W^{n,1}_{x} W^{\ell,1}_{v}(m)$, in the sense that
$$
\forall\, t \ge0, \quad \| \SS_\BB(t) \|_{\BBB(W^{n,1}_{x} W^{\ell,1}_{v}(m) )} \le C e^{-\lambda t}.
$$

\end{lem}

\begin{proof}[Proof of Lemma~\ref{lem:hypo1}]
Consider the equation
$$
\partial_t f = \BB f = \BB_0 f - v\cdot \nabla_x f -M\chi_R f.
$$
Remind that $\BB_0 f = Q(\mu,f)$ and remark that $x$-derivatives commute with the operator $\BB$, thus for any multi-index $\alpha, \beta \in \N^3$, we have
$$
\partial^\alpha_v \partial^\beta_x (\BB f )=\partial^\alpha_v ( \BB \partial^\beta_x f )
$$
and
$$
\partial^\alpha_v \BB_0 f = \partial^\alpha_v Q(\mu,f) = \sum_{\alpha_1+\alpha_2=\alpha}
C_{\alpha_1,\alpha_2} Q(\partial^{\alpha_1}_v \mu , \partial^{\alpha_2}_v f)
$$
and, writing $v \cdot \nabla_x f = v_i \partial_{x_i} f$,
$$
\bal
\partial^\alpha_v \BB f
&= \BB \partial^{\alpha}_v f + \sum_{\alpha_1+\alpha_2=\alpha, |\alpha_1| \ge 1}
C_{\alpha_1,\alpha_2} \big\{ Q(\partial^{\alpha_1}_v \mu , \partial^{\alpha_2}_v f)
- ( \partial^{\alpha_1}_v v_i) \partial_{x_i} (\partial^{\alpha_2}_v f)
- M (\partial^{\alpha_1}_v \chi_R) (\partial^{\alpha_2}_v f)     \big\}
\eal
$$
finally
$$
\bal
&\partial^\alpha_v \partial^\beta_x \BB f
= \BB (\partial^{\alpha}_v \partial^\beta_x f) \\
&\quad
+ \sum_{\alpha_1+\alpha_2=\alpha, |\alpha_1| = 1}
C_{\alpha_1,\alpha_2} \big\{ Q(\partial^{\alpha_1}_v \mu , \partial^{\alpha_2}_v \partial^\beta_x f)
- ( \partial^{\alpha_1}_v v_i) \partial_{x_i} (\partial^{\alpha_2}_v \partial^\beta_x f)
- M (\partial^{\alpha_1}_v \chi_R) (\partial^{\alpha_2}_v \partial^\beta_x f)     \big\}\\
&\quad
+ \sum_{\alpha_1+\alpha_2=\alpha, |\alpha_1| \ge 2}
C_{\alpha_1,\alpha_2} \big\{ Q(\partial^{\alpha_1}_v \mu , \partial^{\alpha_2}_v \partial^\beta_x f)
- M (\partial^{\alpha_1}_v \chi_R) (\partial^{\alpha_2}_v \partial^\beta_x f)     \big\}.\\
\eal
$$

\medskip

We shall treat in full details the case $\ell=n=1$, the others $ \ell,n \ge 2$ being treated in the same way.

\medskip
\noindent
\textbf{Case $\ell=n=1$ :}
\emph{Step 1. Derivatives in $x$.}
First, using the computation \eqref{eq:SSf2} for $p=1$, we have
\beqn\label{eq:f-Wl1}
\bal
\frac{d}{dt} \| f \|_{L^1_{x,v}(m)} =
\int \{ \varphi_{m,1}(v) - M\chi_R(v) \} \, |f| \,  m .
\eal
\eeqn
As explained before, the $x$-derivatives commute with the operator $\BB$, so for any multi-index $\beta \in \N^3$ we get from \eqref{eq:SSf2} that
\beqn\label{eq:fx-Wl1}
\bal
\frac{d}{dt} \| \partial^\beta_x f \|_{L^1_{x,v}(m)} =
\int \{ \varphi_{m,1}(v) - M\chi_R(v) \} |\partial^\beta_x f |\, m .
\eal
\eeqn

\medskip
\noindent
\emph{Step 2. Derivatives in $v$.}
We now consider the derivatives in $v$. For any $\alpha \in \N^3$ with $|\alpha|=1$, we compute the evolution of $v$-derivatives:
$$
\partial_t (\partial^\alpha_v f)  = \BB(\partial^\alpha_v f)
+ Q( \partial^\alpha_v \mu , f)
- (\partial^\alpha_v v_i) \partial_{x_i} f
-M(\partial^\alpha_v \chi_R) f.
$$
From the previous equation we deduce that
$$
\bal
\frac{d}{dt} \| \partial^\alpha_v f \|_{L^1_{x,v}(m)} &= \int
\big\{\BB(\partial^\alpha_v f)
+ Q( \partial^\alpha_v \mu , f)
- (\partial^\alpha_v v_i) \partial_{x_i} f
-M(\partial^\alpha_v \chi_R) f  \big\} \mathrm{sign}(\partial^\alpha_v f) \, m \\
&= : T_1 + T_2 + T_3 + T_4 + T_5,
\eal
$$
where
$$
\bal
&T_1 = \int \BB (\partial^\alpha_v f) \, \mathrm{sign}(\partial^\alpha_v f) \, m \\
&T_2 = \int  (\partial^\alpha_v \bar a_{ij}) \, \partial_{ij} f \, \mathrm{sign}(\partial^\alpha_v f) \, m \\
&T_3 = - \int (\partial^\alpha_v \bar c) \, f \, \mathrm{sign}(\partial^\alpha_v f) \, m \\
&T_4 = - \int (\partial^\alpha_v v_i) \partial_{x_i} f \, \mathrm{sign}(\partial^\alpha_v f) \, m    = 0 \\
&T_5 = - \int M (\partial^\alpha_v \chi_R ) f \, \mathrm{sign}(\partial^\alpha_v f) \, m .
\eal
$$
Again using the computation \eqref{eq:SSf2} of Lemma \ref{lem:hypo} for $p=1$, we have
$$
T_1 =
\int \{ \varphi_{m,1}(v) - M\chi_R(v) \} |\partial^\alpha_v f| \, m .
$$
Concerning $T_5$, we use the following fact on the derivative of $\chi_R$:
$$
\left| \partial^\alpha_v\chi_R (v) \right| = \frac{1}{R} \left| \partial^\alpha_v \chi \left( \frac{v}{R}\right) \right| \leq \frac{C}{R} \, {\mathbf 1}_{R \le |v| \le 2R},
$$
which implies that
$$
T_5 \leq M \frac{C}{R} \| {\mathbf 1}_{R \le |v| \le 2R} \, f \|_{L^1_{x,v}(m)}.
$$
Performing integration by parts, we get
$$
T_2 + T_3 =  - \int \partial^\alpha_v \bar a_{ij} \, \partial_i f \, \partial_j m \, \mathrm{sign}(\partial^\alpha_v f)
+ \int \partial^\alpha_v \bar b_j \, \partial_j m \, f \, \mathrm{sign}(\partial^\alpha_v f) =: A+B.
$$
When $m$ is a polynomial weight $m = \la v \ra^{k}$, we can easily estimate $T_2 + T_3$, thanks to another integration by parts, by
$$
T_2 + T_3
= \int  \{ (\partial^\alpha_v \bar a_{ij}) \, \partial_{ij} m  + 2 (\partial^\alpha_v \bar b_j) \, \partial_j m   \} \, f \, \mathrm{sign}(\partial^\alpha_v f)
\lesssim \| \la v \ra^{\gamma - 1} f \|_{L^1_{x,v}(m)},
$$
where we have used $|\partial^\alpha_v \bar a_{ij}| \le C \la v \ra^{\gamma+1}$, $|\partial^\alpha_v \bar b_{j}| \le \la v \ra^{\gamma}$, $|\partial_j m| \le C \la v \ra^{-1} m$ and $|\partial_{ij} m| \le C \la v \ra^{-2} m$.

We now investigate the case of (stretched) exponential weight $m = e^{r \la v \ra^s}$.
First, we can easily estimate the term $B$, since $\partial_j m = C v_j \la v \ra^{\sigma-2} m$, as
$$
B \lesssim \| \la v \ra^{\gamma+s - 1} f \|_{L^1_{x,v}(m)}.
$$
For the other term, integrating by parts again (first with respect to the $\partial^\alpha_v$-derivative then to the $\partial_i$-derivative), gives us
$$
A = -\int \left\{ \bar a_{ij} \frac{\partial_{ij} m}{m} + \bar b_{j} \frac{\partial_{j} m}{m} \right\} \, |\partial^\alpha_v f| \, m
+\int \bar a_{ij} \, \partial_i (\partial^\alpha_v m) \, \partial_j f \, \mathrm{sign}(\partial^\alpha_v f),
$$
and we investigate the last term in the right-hand side.
Recall that
$$
\bar a_{ij} \xi_i \xi_j = \ell_1(v) |P_v \xi |^2 + \ell_2(v) |(I-P_v) \xi|^2,
$$
we decompose $\partial_j f = P_v \partial_j f + (I-P_v) \partial_j f$ and similarly for $\partial_j(\partial^\alpha_v m)$, then a tedious but straightforward computation yields
$$
\bal
&\int \bar a_{ij} \, \partial_i (\partial^\alpha_v m) \, \partial_j f \, \mathrm{sign}(\partial^\alpha_v f) \\
&\qquad
= \int \left\{  rs \ell_1(v) \la v \ra^{s-2} + rs(s-2) \ell_1(v) |v|^2 \la v \ra^{s-4} + r^2 s^2 \ell_1(v) |v|^2 \la v \ra^{2s-4} \right\} \, P_v \partial^\alpha_v f \, \mathrm{sign}(\partial^\alpha_v f) \, m \\
&\qquad\quad
+\int rs \ell_2(v) \la v \ra^{s-2} \, (I-P_v)\partial^\alpha_v f \, \mathrm{sign}(\partial^\alpha_v f) \, m .
\eal
$$
Recall that $\varphi_{m,1}(v) = \bar a_{ij} \frac{\partial_{ij} m}{m} + 2\bar b_{j} \frac{\partial_{j} m}{m}$ (see eq.~\eqref{eq:def-varphi}), hence we obtain
$$
\bal
T_1 + A
&\le \int \left\{ \psi_{m,1}(v)   - M \chi_R(v) \right\} \, |\partial^\alpha_v f| \, m
\eal
$$
with
$$
\bal
\psi_{m,1}(v) &:= \bar b_{j} \frac{\partial_{j} m}{m} + rs \ell_2(v) \la v \ra^{s-2}
+ rs \ell_1(v) \la v \ra^{s-2} \\
&\quad
+ rs(s-2) \ell_1(v) |v|^2 \la v \ra^{s-4}
+r^2 s^2 \ell_1(v) |v|^2 \la v \ra^{2s-4}.
\eal
$$
Thanks to the asymptotic behaviour of $\ell_1(v)$ and $\ell_2(v)$ in Lemma~\ref{lem:bar-aij} and arguing as in Lemma~\ref{lem:varphi}, we obtain first that
\beqn\label{eq:tilde-varphi-asymptotic}
\left\{
\bal
& \limsup_{|v| \to +\infty} \psi_{m,1}(v) \le - rs \la v \ra^{\gamma+s}, \quad\text{if } 0<s<2; \\
&\limsup_{|v| \to +\infty} \psi_{m,1}(v) \le -2r (1-4r) , \quad\text{if }  s=2;
\eal
\right.
\eeqn
and then for any positive $\lambda < \lambda_{m,1}$ and $\delta \in (0, \lambda_{m,1} - \lambda)$ we can choose $M,R$ large enough such that $\psi_{m,1}(v) - M\chi_R(v) \le - \lambda - \delta \la v \ra^{\gamma+\sigma}$.

Putting together all the previous estimates of this step, and denoting $\varphi^\sigma (v) = \varphi_{m,1}(v)$ when $m = \la v \ra^k$ and $\varphi^\sigma (v) = \psi_{m,1}(v)$ when $m = e^{r \la v \ra^s}$, we obtain
\beqn\label{eq:fv-W11}
\frac{d}{dt} \| \partial^\alpha_v f \|_{L^1_{x,v}(m)}
\leq \int \{  \varphi^{\sigma}(v)  - M \chi_R(v) \} \, | \partial^\alpha_v f| \, m
+ \int \left\{ C \la v \ra^{\gamma+\sigma - 1} +  C \frac{M}{R} {\mathbf 1}_{R \le |v| \le 2R} \right\} \, |f| \, m.
\eeqn

\medskip
\noindent
\emph{Step 3. Conclusion.}
Consider the standard norm on $W^{1,1}_{x,v}(m)$
$$
\| f \|_{W^{1,1}_{x,v}(m)} = \|f\|_{L^1_{x,v}(m)} + \|\nabla_x f\|_{L^1_{x,v}(m)}
+  \|\nabla_v f\|_{L^1_{x,v}(m)}.
$$
Gathering the previous estimates \eqref{eq:f-Wl1}, \eqref{eq:fx-Wl1} and \eqref{eq:fv-W11},
 we finally obtain
$$
\bal
\frac{d}{dt} \| f \|_{W^{1,1}_{x,v}(m)}
&\leq  \int \left\{ \varphi_{m,1} (v)
+   C \la v \ra^{\gamma+  \sigma-1}
+   M \, \frac{C}{R} {\mathbf 1}_{R \le |v| \le 2R} - M\chi_R \right\} |f| \, m \\
&\quad
+  \int \{ \varphi_{m,1}(v)   - M\chi_R \} |\nabla_x f| \, m
+   \int \{ \varphi^{\sigma}(v) - M\chi_R \} |\nabla_v f| \, m.
\eal
$$
Remark that, since $\sigma \in[0,2]$, the function $\phi^0_m(v) := \varphi_{m,1}(v) +  C \la v \ra^{\gamma + \sigma-1} $ has the same asymptotic behaviour of $\varphi_{m,1}(v)$ (see eq.~\eqref{cas:poly1} and eq.~\eqref{cas:exp1}).
Then, arguing as in Lemma \ref{lem:varphi} (and \eqref{eq:tilde-varphi-asymptotic}), for any positive $\lambda<\lambda_{m,1}$ and $\delta \in (0, \lambda_{m,1} - \lambda)$, one may find $M>0$ and $R>0$ large enough such that
$$
\bal
\varphi_{m,1}(v) +  C \la v \ra^{\gamma + \sigma-1} + M \, \frac{C}{R} {\mathbf 1}_{R \le |v| \le 2R} - M\chi_R &\leq -\lambda - \delta \la v \ra^{\gamma+\sigma}, \\
\varphi_{m,1}(v)  - M\chi_R &\leq -\lambda - \delta \la v \ra^{\gamma+\sigma}, \\
\varphi^\sigma(v)  - M\chi_R &\leq -\lambda - \delta \la v \ra^{\gamma+\sigma}.
\eal
$$
This implies that
$$
\frac{d}{dt} \| f \|_{W^{1,1}_{x,v}(m)} \leq -\lambda \| f \|_{W^{1,1}_{x,v}(m)}
- \delta \| f \|_{W^{1,1}_{x,v}(m \la v \ra^{\gamma+\sigma})},
$$
which concludes the proof in the case $\ell=1$.

\medskip
\noindent
\textbf{Case $\ell \ge 2$ :}
The higher order derivatives are treated in the same way, so we omit the proof.
\end{proof}

\begin{lem}\label{lem:hypo2}
Consider hypothesis {\bf (H1)}, {\bf (H2)} or {\bf (H3)}, $ \ell \in \N$ and $n \in \N^*$.
Then, for any $\lambda<\lambda_{m,2}$, we can choose $M>0$ and $R>0$ large enough such that the operator $\BB+ \lambda$ is hypo-dissipative in $H^{n}_{x} H^{\ell}_{v}(m)$, in the sense that
$$
\forall\, t \ge0, \quad \| \SS_\BB(t) \|_{\BBB(H^{n}_{x} H^{\ell}_{v}(m) )} \le C e^{-\lambda t}.
$$

\end{lem}

\begin{proof}[Proof of Lemma \ref{lem:hypo2}]
Let us consider the equation $ \partial_t f = \BB f = \BB_0 f -M\chi_R f$.
Again we treat the case $\ell=1$ in full details, the others $\ell \ge 2$ being the same.

\medskip
\noindent
\textbf{Case $\ell=n=1$ :}
{\it Step 1. $L^2$ estimate.}
The $L^{2}_{x,v}(m)$ estimate is a special case of Lemma~\ref{lem:hypo}, from which we have
\beqn\label{eq:f-H1}
\bal
\frac{1}{2} \frac{d}{dt} \| f \|_{L^2_{x,v}(m)}^2
&\leq   - c_{0}\int  \{  \la v\ra^{\gamma} |P_v \nabla_{v} f|^{2}+
\la v\ra^{\gamma+2} |(I-P_v )\nabla_{v} f|^{2}  \} \, m^2 \\
&\quad
+  \int \{ \varphi_{m,2}(v) - M\chi_R(v) \}  f^2 \, m^2 .
\eal
\eeqn

\medskip
\noindent
{\it Step 2. $x$-derivatives.}
Recall that the $x$-derivatives commute with the equation, so for any $\beta \in \N^3$ we have
\beqn\label{eq:fx-H1}
\bal
\frac{1}{2} \frac{d}{dt} \| \partial^\beta_x f \|_{L^2_{x,v}(m)}^2
&\leq
- c_{0}\int  \{ \la v\ra^{\gamma} |P_v \nabla_{v} (\partial^\beta_x f )|^{2}+
\la v\ra^{\gamma+2} |(I-P_v )\nabla_{v} (\partial^\beta_x f )|^{2} \} \, m^2 \\
&\quad
+\int \{ \varphi_{m,2}(v) - M\chi_R(v) \}  | \partial^\beta_x f |^{2} \, m^2 .
\eal
\eeqn

\medskip
\noindent
{\it Step 3. $v$-derivatives.}
Let $\alpha \in \N^3$ with $|\alpha|=1$. We recall the equation satisfied by $\partial^\alpha_v f$
$$
\partial_t \partial^\alpha_v f  = \BB (\partial^\alpha_v f)+Q(\partial^\alpha_v \mu, f)
- (\partial^\alpha_v v_i)  \, \partial_{x_i} f-M(\partial^\alpha_v \chi_R) f.
$$
From last equation we deduce that
$$
\bal
\frac12 \frac{d}{dt} \| \partial^\alpha_v f \|^2_{L^2_{x,v}(m)} &= \int \left\{  \BB (\partial^\alpha_v f) + Q(\partial^\alpha_v  \mu, f) - (\partial^\alpha_v v_i)\,\partial_{x_i} f  - M (\partial^\alpha_v \chi_R ) f \right\} \, \partial^\alpha_v f \, m^2 \\
&= : T_1 + T_2 + T_3 + T_4 + T_5,
\eal
$$
where
$$
\bal
&T_1 = \int \BB (\partial^\alpha_v f) \, \partial^\alpha_v f \, m^2  \\
&T_2 =  \int (\partial^\alpha_v  \bar a_{ij}) \, \partial_{ij} f \, \partial^\alpha_v f \, m^2  \\
&T_3 = - \int (\partial^\alpha_v  \bar c) \, f \, \partial^\alpha_v f \, m^2  \\
&T_4 = - \int (\partial^\alpha_v v_i) \, \partial_{x_i} f \, \partial^\alpha_v f \, m^2  \\
&T_5 = - \int M (\partial^\alpha_v \chi_R ) f \, \partial^\alpha_v f \, m^2 .
\eal
$$
We have from Lemma~\ref{lem:hypo}
\beqn\label{eq:fv-H1-T1}
\bal
T_1
&\leq  - c_{0}\int  \{  \la v\ra^{\gamma} |P_v \nabla_{v} (\partial^\alpha_v f)|^{2}+
\la v\ra^{\gamma+2} |(I-P_v )\nabla_{v} (\partial^\alpha_v f)|^{2} \} \, m^2  \\
&\quad
+  \int \{ \varphi_{m,2}(v) - M\chi_R(v) \}  |\partial^\alpha_v f|^{2} \, m^2 .
\eal
\eeqn
The terms $T_{3}$, $T_{4}$ and $T_{5}$ are easy to estimate: for any $\eps>0$ we get
\beqn\label{eq:fv-H1-T4}
T_{4}\leq \eps \| \partial^\alpha_v f \|_{L^2_{x,v}(m)}^2+ C(\eps)\| \partial^\alpha_x f\|_{L^2_{x,v}(m)}^2,
\eeqn
\beqn\label{eq:fv-H1-T5}
T_{5}\leq M \frac{C}{R} \| {\mathbf 1}_{R \le |v| \le 2R} \, \partial^\alpha_v f \|_{L^2_{x,v}(m)}^2+M \frac{C}{R} \| {\mathbf 1}_{R \le |v| \le 2R} \, f\|_{L^2_{x,v}(m)}^2,
\eeqn
and using Lemma~\ref{lem:bar-aij}-(b),
\beqn\label{eq:fv-H1-T3}
\bal
T_{3}&\leq C \int\la v\ra^{\gamma-1} \, |f| \, |\partial^\alpha_v f| \, m^2\\
&\leq C\| \la v\ra^{\frac{\gamma-1}{2}} \partial^\alpha_v f \|_{L^2_{x,v}(m)}^2+C\|\la v\ra^{\frac{\gamma-1}{2}} f\|_{L^2_{x,v}(m)}^2 .
\eal
\eeqn
Let us now deal with the part $T_2$. Performing integrations by parts, we have:
$$
\bal
T_2 &=  \int (\partial^\alpha_v \bar a_{ij}) \, \partial_{ij} f \, \partial^\alpha_v f \, m^2  \\
&= - \int (\partial^{\alpha}_v \bar b_{j}) \, \partial_{j} f \, \partial^\alpha_v f \, m^2
- \int (\partial^\alpha_v \bar a_{ij}) \, \partial_{j} f \, \partial_i (\partial^\alpha_v f) \, m^2
 - \int (\partial^\alpha_v \bar a_{ij}) \, \partial_{j} f \, \partial^\alpha_v f \, \partial_i m^2  \\
&=: - \left(T_{21} + T_{22} + T_{23}\right).
\eal
$$
We first deal with $T_{21}$. Using Lemma~\ref{lem:bar-aij}, we have
\beqn\label{eq:fv-H1-T21}
\bal
T_{21}  &\leq C \int \la v \ra^{\gamma} \, |\partial_{j}f| \, |\partial^\alpha_v f| \, m^2\\
&\leq C \|\la v \ra^{\frac{\gamma}{2}} \nabla_{v}f \|^2_{L^2_{x,v}(m)}
=  C \|\la v \ra^{\frac{\gamma}{2}} P_v \nabla_{v}f \|^2_{L^2_{x,v}(m)} +  C \|\la v \ra^{\frac{\gamma}{2}} (I-P_v)\nabla_{v}f \|^2_{L^2_{x,v}(m)} .
\eal
\eeqn
As far as $T_{22}$ is concerned, the integration by parts gives,
$$
\bal
T_{22}&=-\int \partial^{\alpha}_v \big[(1-\chi)m^2\big] \, \bar{a}_{ij} \, \partial_{j}f \, \partial_{i} (\partial^\alpha_v f)  -\int  (1-\chi)m^2\, \overline{a}_{ij} \,  \partial_{j} (\partial^\alpha_v f) \,\partial_{i}(\partial^\alpha_v f) \\
&\quad
-\int  (1-\chi)m^2\,  \overline{a}_{ij} \, \partial_{j}f \, \partial_{i}(\partial^\alpha_v \partial^\alpha_v f)
- \int (\partial^\alpha_v \bar a_{ij} )\, \partial_{j} f \, \partial_i (\partial^\alpha_v f) \, \chi m^2 \\
&=: - \left(\widetilde{T}_{221}+\widetilde{T}_{222}+\widetilde{T}_{223}\right)+T_{220}.
\eal
$$
Let us estimate $\widetilde{T}_{222}+\widetilde{T}_{223}$, using integration by parts,
\begin{align*}
&\phantom{xx}{}\widetilde{T}_{222}+\widetilde{T}_{223}\\
&=\int  (1-\chi)m^{2}\Big[ \ell_1(v)\, P_{v} \nabla_{v}(\partial^\alpha_v \partial^\alpha_v f) \cdot P_{v} \nabla_{v} f
+\ell_2(v)\, (I-P_{v})\nabla_{v} ( \partial^\alpha_v \partial^\alpha_v f) \cdot (I-P_{v} )\nabla_{v }f    \Big]
\\
&\phantom{xx}{}
+\int(1-\chi)m^{2}\Big[ \ell_1(v)\, P_{v} \nabla_{v}(\partial^\alpha_v f) \cdot P_{v} \nabla_{v}(\partial^\alpha_v f)
+ \ell_2(v)\, (I-P_{v} )\nabla_v (\partial^\alpha_v f) \cdot (I-P_{v} )\nabla_{v} (\partial^\alpha_v f  )  \Big]
\\
&=-\widetilde{T}_{221}
-\int (\partial^{\alpha}_v \ell_1(v))\, P_{v} \nabla_{v}(\partial^\alpha_v f)  \cdot P_{v} \nabla_{v }f  \, (1-\chi)m^{2}
\\
&\phantom{xx}{}-\int (\partial^{\alpha}_v \ell_2(v)) \, (I-P_{v})\nabla_{v}(\partial^\alpha_v f) \cdot (I-P_{v} )\nabla_{v }f \,(1-\chi) m^{2}
\\
&\phantom{xx}{}-\int\big[\ell_1(v)-\ell_2(v)\big]\, (I-P_{v} )\partial^\alpha_v (\partial^\alpha_v f) \,
\frac{v \cdot \nabla_{v }f}{|v|^{2}} \, (1-\chi)m^{2}
\\
&\phantom{xx}{}-\int\big[\ell_1(v)-\ell_2(v)\big] \, (I-P_{v} )\nabla_{v } \partial^\alpha_v f \,
\frac{v \cdot \nabla_{v }g}{|v |^{2}} \, (1-\chi) m^{2}  \\
&=: -\widetilde{T}_{221}+T_{221}+...+T_{224}\,.
\end{align*}
This means $T_{22}=T_{220}+T_{221}+...+T_{224}$. In order to estimate $T_{22}$, we need to estimate $T_{22i}$ for $i=0, \dots,4$ (lemma \ref{lem:bar-aij} plays an important role in those estimates). First of all, we obtain
$$
\bal
T_{220}&\leq C \int_{|v|\leq 2}\la v\ra^{\gamma+1}  | \nabla_{v }f| \,  |\nabla_{v }(\partial^\alpha_v f)| \, |\chi|m^2 \\
&\leq
\eps \| \la v\ra^{\frac{\gamma}{2}} \nabla_{v }(\partial^\alpha_v f)\|_{L^2_{x,v}(m)}^2
+C(\eps) \| \la v\ra^{\frac{\gamma}{2}}  \nabla_{v }f\|_{L^2_{x,v}(m)}^2
\eal
$$
For $T_{221}$, we have
$$
\bal
T_{221}&\leq C \int_{|v|\geq 1}   \la v\ra^{\gamma-1}  |P_{v} \nabla_{v }f| \,  |P_{v} \nabla_{v } (\partial^\alpha_v f)| \, m^2 \\
&\leq
 \eps\| \la v\ra^{\frac{\gamma - 1}{2}} P_{v} \nabla_{v }(\partial^\alpha_v f)\|_{L^2_{x,v}(m)}^2+C(\eps)\|\la v\ra^{\frac{\gamma - 1}{2}} P_{v} \nabla_{v }f\|_{L^2_{x,v}(m)}^2\,.
\eal
$$
For $T_{222}$, we have
$$
\bal
T_{222}&\leq C \int_{|v|\geq 1 }\la v\ra^{\gamma+1}  |(I-P_{v} ) \nabla_{v }f| \, |(I-P_{v} ) \nabla_{v }(\partial^\alpha_v f)|  \, m^2\\
&\leq
 \eps\| \la v\ra^{\frac{\gamma+1}{2}} (I-P_{v} )\nabla_{v } (\partial^\alpha_v f)\|_{L^2_{x,v}(m)}^2+C(\eps)\|\la v\ra^{\frac{\gamma+1}{2}}  (I-P_{v} ) \nabla_{v }f\|_{L^2_{x,v}(m)}^2.
\eal
$$
For $T_{223}$, we obtain
$$
\bal
T_{223}&\leq C \int_{|v|\geq 1}  \big(  \la v\ra^{\gamma-1}+ \la v\ra^{\gamma+1} \big) |\nabla_{v }f|  \, |(I-P_{v} ) \nabla_{v } (\partial^\alpha_v f)|  \, m^2\\
&\leq  \eps\| \la v\ra^{\frac{\gamma+2}{2}}(I-P_{v} )\nabla_{v } (\partial^\alpha_v f) \|_{L^2_{x,v}(m)}^2
+C(\eps)\|\la v\ra^{\frac{\gamma}{2}} \nabla_{v }f\|_{L^2_{x,v}(m)}^2 .
\eal
$$
Finally, for $T_{224}$,
$$
\bal
T_{224}&\leq C \int_{ |v|\geq 1 } \big(  \la v\ra^{\gamma-1}+ \la v\ra^{\gamma+1} \big) |\nabla_{v } (\partial^\alpha_v f)| \,  |(I-P_{v} ) \nabla_{v }f| \, m^2 \\
&\leq  \eps\| \la v\ra^{\frac{\gamma}{2}} \nabla_{v } (\partial^\alpha_v f) \|_{L^2_{x,v}(m)}^2
+ C(\eps)\|\la v\ra^{\frac{\gamma+2}{2}}(I-P_{v} )\nabla_{v }f\|_{L^2_{x,v}(m)}^2
 \eal
$$
This completes the estimate of $T_{22}$ that we write, gathering previous bounds, as
\beqn\label{eq:fv-H1-T22}
\bal
T_{22}
&\le \eps  \| \la v\ra^{\frac{\gamma}{2}} P_v \nabla_v (\partial^\alpha_v f) \|_{L^2_{x,v}(m)}
+\eps \| \la v\ra^{\frac{\gamma+2}{2}} (I-P_v) \nabla_v (\partial^\alpha_v f) \|_{L^2_{x,v}(m)}\\
&\quad
C(\eps)  \| \la v\ra^{\frac{\gamma}{2}} P_v \nabla_v  f \|_{L^2_{x,v}(m)}
+C(\eps)  \| \la v\ra^{\frac{\gamma+2}{2}} (I-P_v) \nabla_v  f \|_{L^2_{x,v}(m)}.
\eal
\eeqn

Concerning $T_{23}$, we apply the same process as $T_{22}$: we first write
$$
\bal
T_{23} &=- \int (\partial^\alpha_v  \bar a_{ij}) \, \partial_{j} f \, \partial_i m^{2} \, \chi g
\\
&\phantom{xx}{}
-\int \partial^\alpha_v \ell_1(v) \, P_{v} \nabla_{v} m^2 \cdot P_{v} \nabla_{v }f \, (1-\chi) \, \partial^\alpha_v f
\\
&\phantom{xx}{}
-\int   \partial^\alpha_v \ell_2(v) \, (I-P_{v} )\nabla_{v }m^{2} \cdot (I-P_{v} )\nabla_{v }f \,(1-\chi) \, \partial^\alpha_v f
\\
&\phantom{xx}{}
-\int \big[\ell_1(v)-\ell_2(v)\big] \, (I-P_{v} )\partial^\alpha_v m^{2}  \,
\frac{v \cdot \nabla_{v }f }{|v |^{2}} \, (1-\chi) \, \partial^\alpha_v f
\\
&\phantom{xx}{}
-\int\big[\ell_1(v)-\ell_2(v)\big]\, (I-P_{v} )\partial^\alpha_v f  \,
\frac{v \cdot \nabla_{v }m^{2}}{|v |^{2}} \, (1-\chi) \, \partial^\alpha_v f  \\
&=: T_{230}+...+T_{234} .
\eal
$$
Note that $(I-P_{v} )\nabla_{v }m^{2}=0$, one can easily get $T_{232}=T_{233}=0$. Let us estimate the other terms, by Lemma \ref{lem:bar-aij}, we have
$$
\bal
T_{230}&\leq C \int_{|v|\leq 2}\la v\ra^{\gamma+\sigma}\,  | \nabla_{v }f| \,  |\partial^\alpha_v f| \, |\chi| \, m^2 \\
&\leq
 \eps\| \la v \ra^{\frac{\gamma}{2}} \partial^\alpha_v f \|_{L^2_{x,v}(m)}^2
 +C(\eps)\| \la v \ra^{\frac{\gamma}{2}}  \nabla_{v }f\|_{L^2_{x,v}(m)}^2
\eal
$$
also
$$
\bal
T_{231} &\leq C\int_{|v|>1 } \la v \ra^{\gamma+\sigma-2} \, |P_{v}\nabla_{v}f| \, |\partial^\alpha_v f| \, m^2 \\
&\leq C(\eps) \|\la v \ra^{\frac{\gamma}{2}} P_{v}\nabla_{v}f\|^2_{L^2_{x,v}(m)}
+ \eps \|\la v \ra^{\frac{\gamma+2\sigma-4}{2}} \partial^\alpha_v f \|^2_{L^2_{x,v}(m)} ,
\eal
$$
and
$$
\bal
T_{234}&\leq C\int_{|v|> 1 } \big(  \la v\ra^{\gamma+\sigma-2}+ \la v\ra^{\gamma+\sigma} \big)\, |(I-P_{v} )\nabla_{v }f| \, |\partial^\alpha_v f| \, m^2\\
&\leq C(\eps)   \|\la v \ra^{\frac{\gamma+2}{2}} (I-P_v) \nabla_v f \|^2_{L^2_{x,v}(m)}
+ \eps \|\la v \ra^{\frac{\gamma+2\sigma-2}{2}} \partial^\alpha_v f \|^2_{L^2_{x,v}(m)}.
\eal
$$
Gathering previous inequalities we complete the estimate of $T_{23}$
\beqn\label{eq:fv-H1-T23}
\bal
T_{23}
&\le \eps\| \la v \ra^{\frac{\gamma}{2}} \partial^\alpha_v f \|_{L^2_{x,v}(m)}^2
+  \eps \|\la v \ra^{\frac{\gamma+2\sigma-2}{2}} \partial^\alpha_v f \|^2_{L^2_{x,v}(m)}\\
&\quad
+ C(\eps) \|\la v \ra^{\frac{\gamma}{2}} P_{v}\nabla_{v}f\|^2_{L^2_{x,v}(m)}
+C(\eps)   \|\la v \ra^{\frac{\gamma+2}{2}} (I-P_v) \nabla_v f \|^2_{L^2_{x,v}(m)}.
\eal
\eeqn

Putting together \eqref{eq:fv-H1-T1} to \eqref{eq:fv-H1-T23} we get, using the fact that
$ 1 +  \la v \ra^{\gamma} + \la v \ra^{\gamma+2\sigma-2} \lesssim \la v \ra^{\gamma+\sigma}$,
\beqn\label{eq:fv-H1}
\bal
\frac12 \frac{d}{dt} \| \partial^\alpha_v f \|_{L^2_{x,v}(m)}^2
&\le - (c_0-\eps)  \int \big \{ \la v \ra^\gamma |P_v \nabla_v (\partial^\alpha_v f)|^2
+ \la v \ra^{\gamma+2} |(I-P_v) \nabla_v (\partial^\alpha_v f)|^2 \big\} \, m^2 \\
&\quad
+ \int \left\{ \varphi_{m,2}(v)
+ \eps \la v \ra^{\gamma+\sigma}
+ C \la v \ra^{\gamma-1}
+ M\frac{C}{R} {\mathbf 1}_{R \le |v| \le 2R}
- M \chi_R(v)  \right\} \, | \partial^\alpha_v f |^2\, m^2 \\
&\quad
+C(\eps)\int \big \{ \la v \ra^\gamma |P_v \nabla_v  f|^2
+ \la v \ra^{\gamma+2} |(I-P_v) \nabla_v  f|^2 \big\} \, m^2 \\
&\quad
+ \int \left\{  C \la v \ra^{\gamma-1} + M\frac{C}{R} {\mathbf 1}_{R \le |v| \le 2R}  \right\} \, |  f |^2\, m^2
+ C(\eps) \| \partial^\alpha_x f \|_{L^2_{x,v}(m)}^2.
\eal
\eeqn

\medskip
\noindent
{\it Step 4. Conclusion in the case $\ell=n=1$.}
We now introduce the following norm on $H^1_x H^1_v(m)$
$$
\| f \|_{\widetilde H^1(m)}^2 :=  \|f\|_{L^2_{x,v}(m)}^2 + \|\nabla_x f\|_{L^2_{x,v}(m)}^2 + \eta \, \|\nabla_v f\|_{L^2_{x,v}(m)}^2 ,
$$
which is equivalent to the standard $H^1_{x,v}(m)$-norm for any $\eta>0$.
Gathering estimates \eqref{eq:f-H1}, \eqref{eq:fx-H1} and \eqref{eq:fv-H1} of previous steps, we obtain
$$
\bal
\frac{1}{2} \frac{d}{dt} \| f \|_{\widetilde H^1(m)}^2
&\leq (-c_0 +  \eta \,  C(\eps))\int  \Big\{ \la v\ra^{\gamma} |P_v \nabla_{v} f|^{2} + \la v\ra^{\gamma+2} |(I-P_v )\nabla_{v} f|^{2} \Big\}\, m^2\\
& \qquad+  \int \left\{ \psi^0_m(v) +  \eta \, M \frac{C}{R} {\mathbf 1}_{R \le |v| \le 2R} - M\chi_R(v) \right\}   f^{2} \, m^2 \\
&\qquad -c_0 \sum_{|\beta|=1}\int  \Big\{ \la v\ra^{\gamma} |P_v \nabla_{v} (\partial^\beta_x f)|^{2} + \la v\ra^{\gamma+2} |(I-P_v )\nabla_{v} (\partial^\beta_x f)|^{2} \Big\} \, m^2\\
& \qquad+  \int \left\{ \psi^1_m(v)  - M\chi_R(v) \right\}  | \nabla_{x}f|^{2} \, m^2  \\
& \qquad + \eta  (-c_0 + \eps ) \sum_{|\alpha|=1} \int
\Big\{ \la v\ra^{\gamma} |P_v \nabla_{v} (\partial^\alpha_v f)|^{2}
+  \la v\ra^{\gamma+2} |(I-P_v )\nabla_{v} (\partial^\alpha_v f)|^{2} \Big\} \, m^2  \\
&\qquad+  \eta \int \left\{ \psi^2_m(v) + M \frac{C}{R} {\mathbf 1}_{R \le |v| \le 2R}  - M\chi_R(v) \right\}  |\nabla_v f |^{2} \, m^2 .
\eal
$$
where we have defined
$$
\bal
& \psi^0_m(v) := \varphi_{m,2}(v) + C\eta \la v \ra^{\gamma-1}, \\
& \psi^1_m(v) := \varphi_{m,2}(v) + \eta C(\eps), \\
& \psi^2_m(v) := \varphi_{m,2}(v) +\eps\la v \ra^{\gamma+\sigma} + C \la v \ra^{\gamma-1}.
\eal
$$
Let us fix any $\lambda < \lambda_{m,2}$. We first choose $\eps >0$ small enough so that $-c_0 + \eps  < 0$ and $-\lambda_{m,2} + \eps < - \lambda$. Then we choose $\eta>0$ small enough such that $ -c_0 + \eta \, C(\eps) \leq 0$ and $ -\lambda_{m,2} +  \eta C(\eps) < - \lambda$.
Hence the functions
$\psi^i_m$ have the same asymptotic behaviour than $\varphi_{m,2}$ (see \eqref{cas:poly1}, \eqref{cas:exp1} and \eqref{cas:exp2}). Then, using Lemma~\ref{lem:varphi}, for any $\lambda < \lambda_{m,2}$ and $\delta \in (0,\lambda_{m,2} - \lambda)$, one may find $M>0$ and $R>0$ large enough such that
$$
\bal
\psi^0_m(v)+  \eta \, M \frac{C}{R} {\mathbf 1}_{R \le |v| \le 2R}
- M\chi_R(v) &\leq -\lambda - \delta\la v \ra^{\gamma+\sigma}, \\
\psi^1_m(v)- M\chi_R(v)  &\leq -\lambda - \delta\la v \ra^{\gamma+\sigma}, \\
\psi^2_m(v) + M \frac{C}{R} {\mathbf 1}_{R \le |v| \le 2R}  - M\chi_R(v) &\leq -\lambda - \delta\la v \ra^{\gamma+\sigma}.
\eal
$$
This implies
$$
\bal
\frac{1}{2} \frac{d}{dt} \| f \|_{\widetilde H^1(m)}^2
&\le  -\lambda  \| f \|_{\widetilde H^1(m)}^2 - \delta \|  f \|_{\widetilde H^1(m \la v \ra^{(\gamma+\sigma)/2})}^2 \\
&\quad
- K \Big\{  \| \la v \ra^{\frac{\gamma}{2}} P_v \nabla_v f \|_{L^2(m)}^2 + \| \la v \ra^{\frac{\gamma+2}{2}}(I- P_v) \nabla_v f \|_{L^2(m)}^2  \Big\} \\
&\quad
- K \Big\{  \| \la v \ra^{\frac{\gamma}{2}} P_v \nabla_v (\nabla_x f) \|_{L^2(m)}^2 + \| \la v \ra^{\frac{\gamma+2}{2}}(I- P_v) \nabla_v (\nabla_x f) \|_{L^2(m)}^2  \Big\} \\
&\quad
- K \Big\{  \| \la v \ra^{\frac{\gamma}{2}} P_v \nabla_v (\nabla_v f) \|_{L^2(m)}^2 + \| \la v \ra^{\frac{\gamma+2}{2}}(I- P_v) \nabla_v (\nabla_v f) \|_{L^2(m)}^2  \Big\} ,
\eal
$$
and then
$$
\| \SS_\BB(t) f \|_{H^1_{x,v}(m)} \le C e^{-\lambda t} \| f \|_{H^1_{x,v}(m)}.
$$
This concludes the proof of the hypodissipativity of $\BB+\lambda$ in $H^1_{x,v}(m)$.

\bigskip

\medskip
\noindent
\textbf{Case $\ell \ge 2$ :}
The higher order derivatives are treated in the same way, introducing the (equivalent) norm on $H^n_x H^\ell_v(m)$
$$
\| f \|_{\widetilde{H^n_x H^\ell_v} (m)}^2 = \| f \|_{L^2(m)}^2
+ \sum_{1 \le |\alpha| + |\beta| \le \max(\ell,n); |\alpha| \le \ell; |\beta| \le n} \eta^{|\alpha|} \, \|  \partial_{v}^{\alpha} \partial_{x}^{\beta} f \|_{L^2(m)}^2,
$$
and choosing $\eta>0$ small enough as in the case $\ell=1$.
\end{proof}

\begin{lem}\label{lem:hypo3}
Consider hypothesis {\bf (H1)}, {\bf (H2)} or {\bf (H3)}, $ \ell \in \N$ and $n \in \N^*$,
and $p \in [1,2]$.
Then, for any $\lambda < \lambda_{m,p}$, we can choose $M>0$ and $R>0$ large enough such that the operator $\BB+\lambda$ is hypo-dissipative in $W^{n,p}_{x} W^{\ell,p}_{v}(m)$, in the sense that
$$
\forall\, t \ge0, \quad \| \SS_\BB(t) \|_{\BBB(W^{n,p}_{x} W^{\ell,p}_{v}(m))} \le C e^{-\lambda t}.
$$

\end{lem}

\begin{proof}
It is a consequence of Lemmas \ref{lem:hypo1} and \ref{lem:hypo2}, together with the Riesz-Thorin interpolation theorem.
\end{proof}

\subsection{Regularization}\label{ssec:reg}
We now turn to the boundedness of $\AA$ as well as regularization properties of $\AA \SS_\BB(t)$.
We recall the operator $\AA$ defined in \eqref{eq:AB}
$$
\AA f = \AA_0 f + M\chi_R f = (a_{ij}\ast f)\partial_{ij}\mu - (c\ast f)\mu + M\chi_R f,
$$
for $M$ and $R$ large enough chosen before.
Thanks to the smooth cut-off function $\chi_R$, for any $q\in[1,+\infty]$, $p\geq q$ and any weight function $m$ under the hypotheses {\bf (H1)}-{\bf (H2)}-{\bf (H3)}, we easily obtain
$$
\bal
\| M \chi_R f \|_{L^q_{x,v}( \mu^{-1/2} )}
\lesssim \|f\|_{L^q_x L^p_v(m)}.
\eal
$$
Taking derivatives we get an analogous estimate, for any $n,\ell \in \N$,
$$
\bal
\| M \chi_R f \|_{W^{n,q}_x W^{\ell,q}_v ( \mu^{-1/2} )}
\lesssim \| f \|_{W^{n,q}_x W^{\ell,p}_v ( m)}, \qquad
\eal
$$
Arguing by duality we also have
$$
\| M \chi_R f \|_{H^n_x H^{-1}_v( \mu^{-1/2} )} \lesssim \| f \|_{H^n_x H^{-1}_v (m)}.
$$
Finally we obtain
\beqn\label{chiRbounded}
M \chi_R \in
\left\{
\bal
\BBB \left( L^{p}_{x,v} ( m),  L^{p}_{x,v} ( \mu^{-1/2} ) \right),& \quad \forall\, p \in [1,\infty];\\
\BBB \left( W^{n,p}_x W^{\ell,p}_v ( m),  W^{n,p}_x W^{\ell,p}_v ( \mu^{-1/2} ) \right),& \quad \forall\, p \in [1,2], \, n \in \N^*, \, \ell \in \N .
\eal
\right.
\eeqn

We know obtain the boundedness of $\AA$.

\begin{lem}\label{lem:A0}
Consider {\bf (H1)}, {\bf (H2)} or {\bf (H3)} and a weight function $m$.
\begin{enumerate}[(i)]

\item For any $p \in[1,\infty]$, there holds
$$
\AA \in \BBB \left(  L^{p}_{x,v}(m) , L^{p}_{x,v}(\mu^{-1/2}) \right).
$$

\item For any $p \in[1,2]$ , $ n \in \N^*$ and $\ell \in \N $, there holds
$$
\AA \in \BBB \left(  W^{n,p}_x W^{\ell,p}_{v}(m) , W^{n,p}_x W^{\ell,p}_{v}(\mu^{-1/2}) \right).
$$


\end{enumerate}
In particular $ \AA \in \BBB(E) \cap \BBB(\EE)$ for any admissible space $\EE$ in \eqref{def:EE}.

\end{lem}

\begin{proof}
Thanks to \eqref{chiRbounded} we just need to consider the operator $\AA_0$. We write $$\AA_0 f = (a_{ij}*f)\partial_{ij}\mu - (c*f)\mu$$ and split the proof into several steps.

\smallskip
{\it Step 1.} Since $\gamma \in [-2,1]$ we have $|a_{ij}(v-v_*)| \lesssim \la v \ra^{\gamma+2} \la v_* \ra^{\gamma+2}$, which implies $ |(a_{ij}*f) (v)| \lesssim \la v \ra^{\gamma+2} \| f \|_{L^1_v(\la v \ra^{\gamma+2})}$. Therefore, for any $p \in [1,\infty]$, we have
$$
\| (a_{ij}*f)\partial_{ij}\mu \|_{L^p_v (\mu^{-1/2})} \lesssim \| f \|_{L^1_v(\la v \ra^{\gamma+2})},
$$
from which we can also easily deduce
$$
\| \partial^\alpha_v \partial^\beta_x (a_{ij}*f)\partial_{ij}\mu \|_{L^p_v (\mu^{-1/2})}
\lesssim \sum_{\alpha_1  \le \alpha} \| \partial^{\alpha_1}_v \partial^\beta_x f \|_{L^1_{v}(\la v\ra^{\gamma+2})}.
$$
Integrating in the $x$-variable, we finally get
$$
\| (a_{ij}*f)\partial_{ij}\mu \|_{W^{n,p}_x W^{\ell,p}_{v} (\mu^{-1/2})} \lesssim \| f \|_{W^{n,p}_x W^{\ell,1}_{v}(\la v \ra^{\gamma+2})}.
$$

\smallskip
{\it Step 2.} Assume $\gamma \in [0,1]$. In that case we have $|c(v-v_*)| \lesssim \la v \ra^\gamma \la v_* \ra^\gamma$ and the same argument as above gives
$$
\| (c*f) \mu \|_{W^{n,p}_x W^{\ell,p}_{v} (\mu^{-1/2})} \lesssim \| f \|_{W^{n,p}_x W^{\ell,1}_{v}(\la v \ra^{\gamma})}.
$$

\smallskip
{\it Step 3.} Assume $\gamma \in [-2,0)$.
We decompose $ c = c_+ + c_-$ with $c_+ = c {\mathbf 1}_{|\cdot| > 1}$ and $c_- = c {\mathbf 1}_{|\cdot| \le 1}$. For the non-singular term $c_+$ we easily get, for any $p\in[1,\infty]$,
$$
\| (c_+ * f) \mu \|_{L^p_v(\mu^{-1/2})} \lesssim \| f \|_{L^1_v}
$$
whence
$$
\| (c_+ * f) \mu \|_{W^{n,p}_x W^{\ell,p}_{v} (\mu^{-1/2})} \lesssim \| f \|_{W^{n,p}_x W^{\ell,1}_{v}}.
$$
We now investigate the singular term $c_-$. For any $p \in [1, 3/|\gamma| )$ we get
$$
\bal
\| (c_- * f) \mu \|_{L^p_v (\mu^{-1/2})}^p
&= \| (c_- * f) \mu^{1/2} \|_{L^p_v}^p
\lesssim \int_v \left|  \int_{v_*} |v-v_*|^{\gamma} \,{\mathbf 1}_{|v-v_*| \le 1} \, |f(v_*)| \right|^p \mu^{1/2}(v) \\
&\lesssim \int_{v_*} |f(v_*)|^p \left\{ \int_v |v-v_*|^{\gamma p } \,{\mathbf 1}_{|v-v_*| \le 1} \, \mu^{1/2}(v)   \right\} \\
&\lesssim \| f \|_{L^p_v (\la v \ra^\gamma)}^p,
\eal
$$
where we have used that $|\gamma| p < 3$ (so that the integral in $v$ is bounded) and Lemma~\ref{lem:Aalpha}. Taking derivatives and integrating in $x$ it follows
$$
\| (c_- * f) \mu \|_{W^{n,p}_x W^{\ell,p}_{v} (\mu^{-1/2})} \lesssim \| f \|_{W^{n,p}_x W^{\ell,p}_{v}(\la v \ra^\gamma)}, \quad \forall\, p \in [1,3/|\gamma|).
$$
Remark that by H\"older's inequality, for any $q \in ( 3/(3+\gamma), \infty]$ we have
$$
\bal
| (c_- * f)(v)|
&\lesssim \int_{v_*} |v-v_*|^{\gamma} \,{\mathbf 1}_{|v-v_*| \le 1} \, |f(v_*)|
\lesssim \left( \int_{v_*} |v-v_*|^{\gamma q'} \,{\mathbf 1}_{|v-v_*| \le 1} \right)^{1/q'} \| f \|_{L^q_v} \lesssim \| f \|_{L^q_v},
\eal
$$
which implies
$$
\| (c_- * f) \mu \|_{L^p_v(\mu^{-1/2})} \lesssim \| f \|_{L^q_v}, \quad \forall\, p \in [1,\infty],
$$
and similarly
$$
\| (c_- * f) \mu \|_{W^{n,p}_x W^{\ell,p}_v(\mu^{-1/2})} \lesssim \| f \|_{W^{n,p}_x W^{\ell,q}_v}, \quad \forall\, p \in [1,\infty].
$$
Observe that in particular the operator $T f = (c_- * f) \mu $ is a bounded operator from $W^{n,1}_x W^{\ell,1}_v(m)$ to $W^{n,1}_x W^{\ell,1}_v (\mu^{-1/2})$ and from $W^{n,\infty}_x W^{\ell,\infty}_v(m)$ to $W^{n,\infty}_x W^{\ell,\infty}_v(\mu^{-1/2})$, thus by interpolation also from $W^{n,p}_x W^{\ell,p}_v (m)$ to $W^{n,p}_x W^{\ell,p}_v (\mu^{-1/2})$ for any $p \in [1,\infty]$. This together with estimates of previous steps completes the proof of points (i) and (ii).
\end{proof}

We turn now to regularization properties of the semigroup $\SS_\BB$. We follow a technique introduced by H\'erau \cite{Herau} for Fokker-Plank equations (see also \cite[Section A.21]{Vi-hypo} and \cite{MM}).

\begin{lem}\label{lem:reg}
Consider hypothesis {\bf (H1)}, {\bf (H2)} or {\bf (H3)} and let $m_0 $ be some weight function with $\gamma+ \sigma >0$. Define
$$
m_1 :=
\left\{
\bal
& m_0 &\quad\text{if } \gamma \in [0,1];\\
& \la v \ra^{\frac{|\gamma|}{2}} m_0  & \quad \text{if } \gamma \in [-2,0).
\eal
\right.
\quad
m_2 :=
\left\{
\bal
& m_0 &\quad\text{if } \gamma \in [0,1];\\
& \la v \ra^{4 |\gamma|} m_0  & \quad \text{if } \gamma \in [-2,0).
\eal
\right.
$$
Then there hold:

\begin{enumerate}

\item From $L^2$ to $H^\ell$ for $\ell \ge 1$:
$$
\forall\, t \in (0,1], \qquad
\| \SS_\BB(t)  \|_{\BBB(L^2(m_1) , H^\ell(m_0))} \le C\, t^{-3\ell/2}
$$

\item From $L^1$ to $L^2$:
$$
\forall\, t \in (0,1], \qquad
\| \SS_\BB(t)  \|_{\BBB(L^1(m_2) , L^2(m_1))} \le C\, t^{-8} .
$$

\item From $L^2$ to $L^\infty$:
$$
\forall\, t \in (0,1], \qquad
\| \SS_\BB(t)  \|_{\BBB(L^2(m_2) , L^\infty(m_1))} \le C\, t^{-8} .
$$


\end{enumerate}

\end{lem}

\begin{proof}[Proof of Lemma \ref{lem:reg}]
We consider the equation $\partial_t f = \BB f$ and split the proof into three steps.

\medskip\noindent
{\it Step 1: from $L^2$ to $H^\ell$.} We only prove the case $\ell=1$, the other cases being treated in the same way.
Let us define
$$
\FF(t,f) := \| f\|_{L^2(m_1)}^2 + \alpha_1\, t \,\| \nabla_v f\|_{L^2(m_0)}^2 + \alpha_2 \,t^2 \, \la \nabla_x f, \nabla_v f\ra_{L^2(m_0)} + \alpha_3\, t^3 \, \| \nabla_x f \|_{L^2(m_0)}^2.
$$
We now choose $\alpha_i$, $i=1,2,3$ such that $0<\alpha_3 \le \alpha_2 \le \alpha_1 \le 1$ and $\alpha_2^2 \le 2 \alpha_1 \alpha_3$. Then, there holds
$$
2 \FF(t,f) \geq \alpha_3 \, t^3 \,  \| \nabla_{x,v} f \|^2_{L^2(m_0)}.
$$
Moreover, denoting $f_t = \SS_\BB(t) f$, we have
$$
\bal
\frac{d}{dt}\, \FF(t,f_t)
&= \frac{d}{dt} \| f_t\|_{L^2(m_1)}^2
+ \alpha_1 \,\| \nabla_v f_t\|_{L^2(m_0)}^2
+ \alpha_1\, t \,\frac{d}{dt}\| \nabla_v f_t\|_{L^2(m_0)}^2 \\
&\quad
+ 2 \alpha_2\, t\,  \la \nabla_x f_t , \nabla_v f_t \ra_{L^2(m_0)}
+ \alpha_2\, t^2\, \frac{d}{dt} \la \nabla_x f_t , \nabla_v f_t \ra_{L^2(m_0)} \\
&\quad
+ 3\alpha_3 \,t^2 \, \| \nabla_x f_t \|_{L^2(m_0)}^2
+ \alpha_3 \,t^3 \, \frac{d}{dt} \| \nabla_x f_t \|_{L^2(m_0)}^2.
\eal
$$
We need to compute
$$
\bal
\frac{d}{dt} \la \nabla_x f , \nabla_v f \ra_{L^2(m_0)}
&= \sum_{|\alpha|=1}
 \int \left\{\partial^\alpha_x(\BB f) \, (\partial^\alpha_v f)
 +  (\partial^\alpha_x f) \, \partial^\alpha_v (\BB f) \right\} \, m_0^2 .
\eal
$$
Let us denote $f_x := \partial^\alpha_x f$ and $f_v := \partial^\alpha_v f$ to simplify and recall that
$$
\bal
\partial^\alpha_x (\BB f) = \bar a_{ij} \partial_{ij} f_x - \bar c f_x - v\cdot \nabla_x f_x - M \chi_R f_x,
\eal
$$
and
$$
\bal
\partial^\alpha_v (\BB f)
&= \bar a_{ij} \partial_{ij} f_v - \bar c f_v - v\cdot \nabla_x f_v - M \chi_R f_v \\
&\quad
+ (\partial^\alpha_v \bar a_{ij}) \partial_{ij} f - (\partial^\alpha_v \bar c) f - f_x - M(\partial^\alpha_v \chi_R) f.
\eal
$$
Using the same computation as in Lemma~\ref{lem:hypo2}, we obtain
$$
\int \left\{ \partial^\alpha_x (\BB f) \, (\partial^\alpha_v f)
 +  (\partial^\alpha_x f) \, \partial^\alpha_v (\BB f) \right\} \, m_0^2
 = T_0 + T_1 + T_2 + T_3,
$$
where
$$
T_0 := - 2 \int \bar a_{ij} \, \partial_i f_x  \, \partial_j f_v \, m_0^2,
$$
$$
\bal
T_1 &:= \int \{ \varphi_{m_0,2}(v) - 2M \chi_R(v) \} \, f_x \, f_v \, m_0^2,
\eal
$$
$$
\bal
T_2 &:= -\int   \left\{  (\partial^\alpha_v\bar a_{ij}) \,\frac{\partial_{i} m_0^2}{m_0^2} + \partial^\alpha_v \bar b_j \right\} \, \partial_j f \, f_x \, m_0^2
-\int  \left\{ \partial^\alpha_v \bar c + M (\partial^\alpha_v \chi_R) \right\}  f \, f_x \, m_0^2
- \int |f_x|^2 \, m_0^2
\eal
$$
and
$$
T_3 := - \int (\partial^\alpha_v \bar a_{ij}) \, \partial_i f \, \partial_j f_x \, m_0^2.
$$
For the term $T_1$, from the proof of Lemma~\ref{lem:varphi} we get
$$
T_1 \lesssim  \int \la v \ra^{\gamma + \sigma} |f_x| \, |f_v| \, m_0^2
\lesssim \eps t \,\| \la v\ra^{\frac{\gamma + \sigma}{2}} \, \partial^\alpha_x f \|_{L^2(m_0)}^2 + \eps^{-1} t^{-1} \, \| \la v\ra^{\frac{\gamma + \sigma}{2}} \, \partial^\alpha_v f \|_{L^2(m_0)}^2.
$$
In a similar way, using $|\partial^\alpha_v \bar a_{ij} | \le C \la v \ra^{\gamma+1}$, $|\partial^\alpha_v \bar b_j| \le C \la v \ra^{\gamma}$ and $|\partial_i m^2 | \le C \la v \ra^{\sigma-1} m^2$, we obtain for the second term
$$
\bal
T_2 &\lesssim  \int \la v \ra^{\gamma+\sigma} |\nabla_v f| \, |f_x| \, m_0^2
+ \int  \left\{ \la v \ra^{\gamma-1} + \frac{M}{R}{\mathbf 1}_{R \le |v| \le 2R}   \right\}|f| \, |f_x| \, m_0^2
- \| \partial^\alpha_x f \|^2_{L^2(m_0)}\\
&\lesssim \eps t \int \left\{  \la v\ra^{\gamma+\sigma} + \la v \ra^{\gamma-1} + \frac{M}{R}{\mathbf 1}_{R \le |v| \le 2R} \right\} \, |\partial^\alpha_x f |^2 \, m_0^2
+ \eps^{-1} t^{-1} \int \left\{ \la v \ra^{\gamma-1} + \frac{M}{R}{\mathbf 1}_{R \le |v| \le 2R} \right\} \, |f|^2 \, m_0^2 \\
&\quad
+ \eps^{-1} t^{-1} \, \| \la v\ra^{\frac{\gamma + \sigma}{2}} \, \nabla_v f \|_{L^2(m_0)}^2
 - \| \partial^\alpha_x f \|^2_{L^2(m_0)}.
\eal
$$
We now investigate $T_0$ and, decomposing $\partial_i f_x = P_v \partial_i f_x + (I-P_v)\partial_i f_x$ and the same for $\partial_j f_v$, we easily get
$$
\bal
T_0
&\lesssim \eps t \big\{ \| \la v \ra^{\frac{\gamma}{2}} P_v \nabla_v(\partial^\alpha_x f)\|_{L^2(m_0)}^2 + \| \la v \ra^{\frac{\gamma+2}{2}} (I-P_v) \nabla_v(\partial^\alpha_x f)\|_{L^2(m_0)}^2 \big\} \\
&\quad
+ \eps^{-1} t^{-1} \big\{ \| \la v \ra^{\frac{\gamma}{2}} P_v \nabla_v(\partial^\alpha_v f)\|_{L^2(m_0)}^2 + \| \la v \ra^{\frac{\gamma+2}{2}} (I-P_v) \nabla_v(\partial^\alpha_v f)\|_{L^2(m_0)}^2 \big\}.
\eal
$$
For the remainder term $T_3$, arguing as in the proof of Lemma~\ref{lem:hypo2} (term $T_{22}$ in that lemma, see \eqref{eq:fv-H1-T22}) gives us
$$
\bal
T_3
&\lesssim \eps t \big\{ \| \la v \ra^{\frac{\gamma}{2}} P_v \nabla_v(\partial^\alpha_x f)\|_{L^2(m_0)}^2 + \| \la v \ra^{\frac{\gamma+2}{2}} (I-P_v) \nabla_v(\partial^\alpha_x f)\|_{L^2(m_0)}^2 \big\} \\
&\quad
+ \eps^{-1} t^{-1} \big\{ \| \la v \ra^{\frac{\gamma}{2}} P_v \nabla_v f\|_{L^2(m_0)}^2 + \| \la v \ra^{\frac{\gamma+2}{2}} (I-P_v) \nabla_v  f\|_{L^2(m_0)}^2 \big\}.
\eal
$$
Finally, putting together previous estimates we obtain
$$
\bal
&\int \{ \nabla_x (\BB f) \nabla_v f  + \nabla_x f \nabla_v(\BB f) \} \, m_0^2 \\
&\lesssim \eps t \Big\{  \| \la v \ra^{\frac{\gamma+\sigma}{2}} \nabla_x f \|_{L^2(m_0)}^2
+  \| \la v \ra^{\frac{\gamma}{2}} P_v \nabla_v (\nabla_x f) \|_{L^2(m_0)}^2
+  \| \la v \ra^{\frac{\gamma+2}{2}} (I-P_v) \nabla_v (\nabla_x f) \|_{L^2(m_0)}^2 \Big\} \\
&\quad
+C \eps^{-1} t^{-1} \Big\{  \| \la v \ra^{\frac{\gamma+\sigma}{2}} \nabla_v f \|_{L^2(m_0)}^2
+  \| \la v \ra^{\frac{\gamma}{2}} P_v \nabla_v (\nabla_v f) \|_{L^2(m_0)}^2
+  \| \la v \ra^{\frac{\gamma+2}{2}} (I-P_v) \nabla_v (\nabla_v f) \|_{L^2(m_0)}^2 \Big\} \\
&\quad
+C \eps^{-1} t^{-1} \Big\{  \| \la v \ra^{\frac{\gamma}{2}} P_v \nabla_v  f \|_{L^2(m_0)}^2
+  \| \la v \ra^{\frac{\gamma+2}{2}} (I-P_v) \nabla_v f \|_{L^2(m_0)}^2 \Big\} \\
&\quad
+ C \eps^{-1} t^{-1} \| f \|_{L^2(m_0)}^2  - \| \nabla_x f \|_{L^2(m_0)}^2.
\eal
$$
Using Cauchy-Schwarz inequality, we also write the following
$$
\bal
2 \alpha_2 t \la \nabla_x f, \nabla_v f \ra_{L^2(m_0)} \leq  \alpha_2 \left(\eps t^2 \|\nabla_x f \|^2_{L^2(m_0)} + C \eps^{-1} \|\nabla_v f \|^2_{L^2(m_0)} \right).
\eal
$$
Moreover, picking up estimates of Lemma~\ref{lem:hypo2}, it follows that: for any $0<\lambda<\lambda_{m,2}$ and $0 < \delta < \lambda_{m,2}-\lambda$, there are $M,R>0$ large enough such that,
$$
\bal
\int (\BB f) f \, m_1^2
&
\le - c_0 \Big\{ \| \la v\ra^{\frac{\gamma}{2}} \, P_v \nabla_v f \|_{L^2(m_1)}^2
+ \| \la v\ra^{\frac{\gamma+2}{2}} \, (I-P_v) \nabla_v f \|_{L^2(m_1)}^2  \Big\} \\
&\quad
- \lambda \| f \|_{L^2(m_1)}^2 - \delta \| \la v \ra^{\frac{\gamma+\sigma}{2}} f \|_{L^2(m_1)}^2,
\eal
$$
also, for some $\eps_0>0$ to be chosen later,
$$
\bal
\int \nabla_v (\BB f) \nabla_v f \, m_0^2
&\le - c_0 \Big\{ \| \la v\ra^{\frac{\gamma}{2}} \, P_v \nabla_v (\nabla_v f) \|_{L^2(m_0)}^2
+ \| \la v\ra^{\frac{\gamma+2}{2}} \, (I-P_v) \nabla_v ( \nabla_v f) \|_{L^2(m_0)}^2  \Big\}\\
&\quad
- \lambda \| \nabla_v f \|_{L^2(m_0)}^2 - \delta \| \la v \ra^{\frac{\gamma+\sigma}{2}} \nabla_v f_t \|_{L^2(m_0)}^2 \\
&\quad
+ C \Big\{ \| \la v\ra^{\frac{\gamma}{2}} \, P_v \nabla_v f \|_{L^2(m_0)}^2
+ \| \la v\ra^{\frac{\gamma+2}{2}} \, (I-P_v) \nabla_v f \|_{L^2(m_0)}^2  \Big\} \\
&\quad
+ C \| f \|_{L^2(m_0)}^2
+C \eps_0^{-1} t^{-1} \| \nabla_v f \|_{L^2(m_0)}^2 + C \eps_0 t \| \nabla_x f \|_{L^2(m_0)}^2,
\eal
$$
and finally
$$
\bal
\int \nabla_x (\BB f) \nabla_x f\, m_0^2
&\le - c_0 \Big\{ \| \la v\ra^{\frac{\gamma}{2}} \, P_v \nabla_v (\nabla_x f) \|_{L^2(m_0)}^2
+ \| \la v\ra^{\frac{\gamma+2}{2}} \, (I-P_v) \nabla_v ( \nabla_x f) \|_{L^2(m_0)}^2  \Big\}\\
&\quad
- \lambda \| \nabla_x f \|_{L^2(m_0)}^2 - \delta \| \la v \ra^{\frac{\gamma+\sigma}{2}} \nabla_x f \|_{L^2(m_0)}^2.
\eal
$$
We choose
$$
\eps_0 = \eps^2, \quad
\alpha_1 := \eps^{5/2}, \quad \alpha_2 := \eps^4, \quad \alpha_3:= \eps^{9/2}.
$$
Therefore, for any $t \in [0,1]$, we can gather previous estimates to obtain
$$
\bal
&\frac{d}{dt}\, \FF(t,f_t) \\
&\le + \left( - c_0 + C \eps^{1/2}  + C \eps^{5/2} + C \eps^{3}  \right) \Big\{ \| \la v\ra^{\frac{\gamma}{2}} \, P_v \nabla_v f_t \|_{L^2(m_1)}^2
+ \| \la v\ra^{\frac{\gamma+2}{2}} \, (I-P_v) \nabla_v f_t \|_{L^2(m_1)}^2  \Big\} \\
&\quad
+ t \eps^{5/2} \left( - c_0   + C \eps^{1/2} \right) \Big\{ \| \la v\ra^{\frac{\gamma}{2}} \, P_v \nabla_v (\nabla_v f_t) \|_{L^2(m_0)}^2
+ \| \la v\ra^{\frac{\gamma+2}{2}} \, (I-P_v) \nabla_v (\nabla_v f_t) \|_{L^2(m_0)}^2  \Big\} \\
&\quad
+ t^3  \eps^{9/2} \left( - c_0   + C \eps^{1/2} \right) \Big\{ \| \la v\ra^{\frac{\gamma}{2}} \, P_v \nabla_v ( \nabla_x f_t) \|_{L^2(m_0)}^2
+ \| \la v\ra^{\frac{\gamma+2}{2}} \, (I-P_v) \nabla_v ( \nabla_x f_t) \|_{L^2(m_0)}^2  \Big\} \\
&\quad
- \lambda \| f_t \|_{L^2(m_1)}^2 - \delta \| \la v \ra^{\frac{\gamma+ \sigma}{2}} f_t \|_{L^2(m_1)}^2 + C t ( \eps^{5/2}  + \eps^{3}) \| f_t \|_{L^2(m_0)}^2 \\
&\quad
- \lambda \eps^{5/2} t  \| \nabla_v f_t \|_{L^2(m_0)}^2
- t  \eps^{5/2} \left( \delta  - C \eps^{1/2}  \right) \| \la v \ra^{\frac{\gamma+ \sigma}{2}} \nabla_v f_t \|_{L^2(m_0)}^2   \\
&\quad
- t^2 \left( \lambda \eps^{9/2} t - C \eps^{9/2}  - \eps^{5}   - C \eps^{9/2} t + \eps^{4}  \right)  \| \nabla_x f_t \|_{L^2(m_0)}^2
- t^3  \eps^{9/2} \left( \delta  - \eps^{1/2}   \right) \|  \la v \ra^{\frac{\gamma+ \sigma}{2}} \nabla_x f_t \|_{L^2(m_0)}^2.
\eal
$$
We then choose $\eps >0$ small enough such that the following conditions are fulfilled:
$$
\left\{
\bal
- c_0  + C \eps^{1/2} + C \eps^{5/2}+ C \eps^{3}  & < -K < 0 , \\
-c_0 + C \eps^{1/2}  & < -K < 0 ,\\
-\lambda + Ct (\eps^{5/2} + \eps^3) & < -K < 0, \\
\delta - C \eps^{1/2} & < -K < 0, \\
C \eps^{9/2}  + \eps^{5}    + C \eps^{9/2} - \eps^{4}  & < -K < 0 .
\eal
\right.
$$
We have then proved that, for any $t \in [0,1]$,
$$
\bal
\frac{d}{dt} \FF(t, f_t)
&\le - K' \Big\{ \| f_t \|_{L^2(m_1)}^2 + \|   \nabla_v f_t \|_{L^2(m_0)}^2 + t^2 \| \nabla_x f \|_{L^2(m_0)}^2 \Big\}
- \delta \| \la v \ra^{\frac{\gamma+\sigma}{2}} f_t \|_{L^2(m_1)}^2,
\eal
$$
which implies
$$
C t^3  \| \nabla_{x,v} f_t \|^2_{L^2(m_0)} \le \FF(t,f_t) \le \FF(0,f_0) = \|f_0\|^2_{L^2(m_1)}.
$$
We deduce
$$
\forall \, t \in (0,1], \quad \| \nabla_{x,v} \SS_\BB(t) f \|_{L^2(m_0)} \le C \, t^{-3/2} \, \|f_0\|_{L^2(m_1)},
$$
and the proof of point (1) for $\ell=1$ is complete.

\medskip
\noindent
{\it Step 2: from $L^1$ to $L^2$.}
We define,
$$
\bal
\GG(t,f_t)
&:= \| f_t \|_{L^1(m_2)}^2
+ \alpha_0 \, t^{N} \widetilde \FF(t,f_t),   \\
\widetilde \FF(t,f_t) &:=
  \| f_t \|_{L^2(m_1)}^2
+ \alpha_1 \, t^2 \| \nabla_v f_t \|_{L^2(m_0)}^2 \\
&\quad
+ \alpha_2 \, t^{4}  \la \nabla_x f_t , \nabla_v f_t \ra_{L^2(m_0)}
+ \alpha_3 \, t^{6}  \| \nabla_x f_t \|_{L^2(m_0)}^2,
\eal
$$
for some $N$ to be chosen later.
Thanks to H\"older and Sobolev inequalities (in $\T^3_x \times \R^3_v$), there holds
$$
\| \la v \ra^q g \|_{L^2}^2 \lesssim \| \nabla_{x,v} g \|_{L^2}^{3/2} \, \| \la v \ra^{4q} g \|_{L^1}^{1/2},
$$
which implies that
\beqn\label{eq:nash}
\bal
\| f \|_{L^2(m_1)}^2
&\lesssim  \|  f \|_{L^1(m_2)}^{1/2} \, \| \nabla_{x,v} (m_0 f) \|_{L^2}^{3/2}   \\
&\lesssim C_\eps t^{-15} \| f \|_{L^1(m_2)}^2
+ \eps t^{5} \| \nabla_{x,v} f \|_{L^2(m_0)}^2
+ \eps t^{5} \| \la v \ra^{\sigma-1} f \|_{L^2(m_0)}^2 \\
&\lesssim C_\eps t^{-15} \| f \|_{L^1(m_2)}^2
+ \eps t^{5} \| \nabla_{x,v} f \|_{L^2(m_0)}^2
+ \eps t^{5} \| \la v \ra^{\frac{\gamma+\sigma}{2}} f \|_{L^2(m_1)}^2,
\eal
\eeqn
where we have used in last line that $ \la v \ra^{\sigma - 1} m_0 \lesssim \la v \ra^{\frac{\gamma+\sigma}{2}} m_1$.
Arguing as in step 1, we have
$$
\bal
\frac{d}{dt} \widetilde \FF(t, f_t)
&\le - K' \Big\{ \| f_t \|_{L^2(m_1)}^2 + \|   \nabla_v f_t \|_{L^2(m_0)}^2 + t^4 \| \nabla_x f \|_{L^2(m_0)}^2 \Big\}
- \delta \| \la v \ra^{\frac{\gamma+\sigma}{2}} f_t \|_{L^2(m_1)}^2.
\eal
$$
Putting together previous estimates it follows
$$
\bal
\frac{d}{d t} \GG (t,f_t)
&\le - K \| f_t \|_{L^1(m_2)}^2 + \alpha_0 N t^{N-1} \widetilde \FF(t,f)
\\&\quad - K' \alpha_0 t^N \Big\{ \| f_t \|_{L^2(m_1)}^2 +  \|  \nabla_v f_t \|_{L^2(m_0)}^2 + t^4 \| \nabla_x f \|_{L^2(m_0)}^2 \Big\}  - \delta \alpha_0 t^N \| \la v \ra^{\frac{\gamma+\sigma}{2}} f_t \|_{L^2(m_1)}^2\\
&\le - K \| f_t \|_{L^1(m_2)}^2 \\
&\quad
+ \alpha_0 N t^{N-1} \| f_t \|_{L^2(m_1)}^2
+ C \alpha_0 N t^{N+1}  \| \nabla_v f_t \|_{L^2(m_0)}^2
+ C \alpha_0 N t^{N+5}  \| \nabla_x f_t \|_{L^2(m_0)}^2 \\
&\quad
- K' \alpha_0 t^N \Big\{ \| f_t \|_{L^2(m_1)}^2 +  \|  \nabla_v f_t \|_{L^2(m_0)}^2 + t^4 \| \nabla_x f \|_{L^2(m_0)}^2 \Big\}
- \delta \alpha_0 t^N \| \la v \ra^{\frac{\gamma+\sigma}{2}} f_t \|_{L^2(m_1)}^2.
\eal
$$
Choose $t_*\in (0,1)$ so that $N t^{N+1} \ll K' t^{N}$ then, for any $t \in [0, t_*]$,
$$
\bal
\frac{d}{d t} \GG (t,f_t)
&\le - K \| f_t \|_{L^1(m_2)}^2
+ C \alpha_0 t^{N-1} \| f_t \|_{L^2(m_1)}^2 - \delta \alpha_0 t^N \| \la v \ra^{\frac{\gamma+\sigma}{2}} f_t \|_{L^2(m_1)}^2 \\
&\quad
- K'' \alpha_0 t^N \Big\{  \| \nabla_v f_t \|_{L^2(m_0)}^2 + t^4 \| \nabla_x f \|_{L^2(m_0)}^2 \Big\}.
\eal
$$
Thanks to \eqref{eq:nash}, for any $t \in [0, t_*]$, we get
$$
\bal
\frac{d}{d t} \GG (t,f_t)
&\le -  (K - C_\eps  \alpha_0 t^{N-16})\| f_t \|_{L^1(m_2)}^2
-  \alpha_0 t^N (\delta - C \eps)\| \la v \ra^{\frac{\gamma+\sigma}{2}} f_t \|_{L^2(m_1)}^2\\
&\quad
- \alpha_0 t^{N+4} (K'' - C \eps ) \| \nabla_{x,v} f \|_{L^2(m_0)}^2
\eal
$$
Taking $N = 16$ and choosing $\eps>0$ small enough then $\alpha_0 >0$ small enough, we get $\frac{d}{d t} \GG (t,f_t) \le 0$ then
$$
\forall\, t \in [0,t_*], \quad
C t^{16} \| f_t \|_{L^2(m_1)}^2 \le \GG (t,f_t) \le \GG (0,f_0) = \| f_0 \|_{L^1(m_2)}^2.
$$
This ends the proof of point (2), using the fact that the norm is propagated for $t > t_*$.

\medskip
\noindent
{\it Step 3: From $L^2$ to $L^\infty$.} Arguing by duality as in Lemma~\ref{lem:hypo4}, the proof follows as in step 2.
\end{proof}

We define the convolution $\SS_1 * \SS_2$ by
$$
(\SS_1 * \SS_2)(t) := \int_0^t \SS_1(\tau) \, \SS_2(t-\tau) \, d\tau,
$$
and, for $n \in \N^*$, we define $\SS^{(*n)}$ by $ \SS^{(n)} = \SS * \SS^{(*(n-1))}$ with $\SS^{(*1)} = \SS$.

\Black

\begin{cor}\label{cor:reg}
Consider hypothesis {\bf (H1)}, {\bf (H2)} or {\bf (H3)}, and spaces $\EE_0, \EE_1$ of the type $E$ or $\EE$ defined in \eqref{def:E} and \eqref{def:EE}.
Then for any $\lambda' < \lambda  < \lambda_{m,p}$ (where $\lambda$ is defined in Lemmas~\ref{lem:hypo}, \ref{lem:hypo4}, \ref{lem:hypo1}, \ref{lem:hypo2} or \ref{lem:hypo3}) there exists $N \in \N$ such that
$$
\| (\AA\SS_\BB)^{(*N)}(t)  \|_{\BBB(\EE_1 , \EE_0 )} \leq C\, e^{- \lambda' t}, \qquad \forall\, t\geq 0.
$$
\end{cor}

\begin{proof}
It is a consequence of the hypodissipativity properties of $\BB$ (Lemmas~\ref{lem:hypo}, \ref{lem:hypo1}, \ref{lem:hypo2} and \ref{lem:hypo3}), the boundedness of the operator $\AA$ (Lemma~\ref{lem:A0}), and the regularization properties in Lemma~\ref{lem:reg}, together with \cite[Lemma 2.4]{MM} and \cite[Lemma 2.17]{GMM}.
\end{proof}

\subsection{Proof of Theorem~\ref{thm:extension}}
Thanks to the estimates proven in previous section, we can now turn to the proof of Theorem~\ref{thm:extension}.

\begin{proof}[Proof of Theorem~\ref{thm:extension}]
Let $\EE$ be an admissible space defined in \eqref{def:EE} and consider $\ell_0 \ge 1$ large enough such that $E := H^{\ell_0}_{x,v} (\mu^{-1/2})$ defined in \eqref{def:E} satisfies $E \subset \EE$.
Recall that in the small/reference space $E$ we already have a spectral gap in 
Theorem~\ref{theo:gapE}.

Then the proof of Theorem~\ref{thm:extension} is a consequence of the hypo-dissipative properties of $\BB$ in Lemmas~\ref{lem:hypo}, \ref{lem:hypo1}, \ref{lem:hypo2}, \ref{lem:hypo3}, the boundedness of $\AA$ in Lemma~\ref{lem:A0} and the regularizing properties of $(\AA \SS_\BB)^{(*N)}$ in Corollary~\ref{cor:reg},
with which we are able to apply
the ``extension theorem'' from \cite[Theorem 2.13]{GMM} (see also \cite[Theorem 1.1]{MM}).
\end{proof}

\subsection{Proof of Theorem~\ref{thm:SL-regularity}}
We give in this subsection a regularity estimate for the semigroup~$\SS_\LL$.

\begin{proof}[Proof of Theorem~\ref{thm:SL-regularity}]
A key argument in the proof of \cite[Theorem 2.13]{GMM} in order to obtain the exponential decay (that gives point $(iii)$ in Theorem~\ref{thm:extension}) is the following factorization of the semigroup, for any $\ell \in \N^*$,
\beqn\label{SL-factorization}
\SS_{\Lambda}(t) (I-\Pi_0) = \sum_{j=0}^{\ell-1} ((I-\Pi_0) \SS_{\BB} * (\AA \SS_\BB)^{(*j)}  )(t) + (\SS_{\Lambda}(I-\Pi_0) * (\AA \SS_{\BB})^{(*\ell)} )(t),
\eeqn
which has been used with $\ell = N$ given by Corollary~\ref{cor:reg}. We now turn to the proof of \eqref{eq:SLambda-reg}, and recall that $\EE = H^n_x L^2_v(m)$ and $\EE_{-1} = H^n_x (H^{-1}_{v,*}(m))$. For sake of simplicity, in what follows, we denote $e_\lambda(t) := e^{\lambda t}$. We write \eqref{SL-factorization} with $\ell = N+1$
$$
\SS_{\Lambda}(t) (I-\Pi_0) = \sum_{j=0}^{N} ((I-\Pi_0) \SS_{\BB} * (\AA \SS_\BB)^{(*j)}  )(t) + (\SS_{\Lambda}(I-\Pi_0) * (\AA \SS_{\BB})^{(*N)} * (\AA \SS_{\BB}) )(t) ,
$$
so that, for any $\lambda_D < \lambda_{m,2}$ and any $\lambda < \lambda_1$, where $\lambda_1 \le \min \{ \lambda_0 , \lambda_D \}$ is given by Theorem~\ref{thm:extension}, we have
\beqn\label{eq:SjN}
e^{\lambda t} \SS_{\Lambda}(t) (I-\Pi_0) = \sum_{j=0}^{N} \SS_j(t) + \SS_{N+1}(t)
\eeqn
with
$$
\SS_j (t) =  \left( (I-\Pi_0) e_\lambda \SS_{\BB} * (e_\lambda\AA \SS_\BB)^{(*j)}\right) (t), \quad j = 0, \ldots, N,
$$
and
$$
\SS_{N+1}(t) = \left( e_\lambda\SS_{\Lambda}(I-\Pi_0) * (e_\lambda\AA \SS_{\BB})^{(*N)}   * (e_\lambda\AA \SS_\BB) \right) (t).
$$
We now prove that $ \| e^{\lambda t} \SS_{\Lambda}(t) (I-\Pi_0)\|_{\BBB(\EE_{-1} , \EE )} \in L^2_t(\R_+)$ by evaluating each term in \eqref{eq:SjN}, which in turn completes the proof of \eqref{eq:SLambda-reg}. 
Using Lemma~\ref{lem:A0}, we easily observe that thanks to Lemmas~\ref{lem:hypo} and \ref{lem:hypo4} there hold
$$
\| e^{\lambda t} \AA \SS_\BB (t) \|_{ \BB (\EE_{-1}, \EE )} \le C \| e^{\lambda t} \SS_\BB (t) \|_{ \BB (\EE_{-1}, \EE )} ,
$$
and also
$$
\| e^{\lambda t} \AA \SS_\BB (t) \|_{ \BB (\EE, \EE )} \le C \| e^{\lambda t} \SS_\BB (t) \|_{ \BB (\EE, \EE )} \le C e^{- (\lambda_D - \lambda) t},
$$
from which we first obtain
$$
\| e^{\lambda t} \AA \SS_\BB (t) \|_{ \BB (\EE_{-1}, \EE )} \in L^2_t (\R_+), \quad
\| e^{\lambda t} \AA \SS_\BB (t) \|_{ \BB (\EE, \EE )} \in L^1_t (\R_+).
$$
Therefore we deduce
$$
\| \SS_0(t) \|_{ \BB (\EE_{-1}, \EE )} = \| e^{\lambda t} \SS_\BB (t) \|_{ \BB (\EE_{-1}, \EE )} \in L^2_t (\R_+), 
$$
and, for $j=1, \dots, N$, 
$$
\| \SS_j(t) \|_{ \BB (\EE_{-1}, \EE )} 
 \le C \| e^{\lambda t} \SS_\BB (t) \|_{ \BB (\EE, \EE )} * \| (e_\lambda \AA \SS_\BB)^{(*(j-1))} (t) \|_{ \BB (\EE, \EE )}
* \| e^{\lambda t} \AA \SS_\BB (t) \|_{ \BB (\EE_{-1}, \EE )} ,
$$
which implies by induction
$$
\| \SS_j(t) \|_{ \BB (\EE_{-1}, \EE )}  \in L^1_t (\R_+) * L^1_t (\R_+) * L^2_t (\R_+) \subset L^2_t (\R_+).
$$
For the last term we first observe that, thanks to Theorem~\ref{theo:gapE},
$$
\| e^{\lambda t}\SS_{\Lambda}(t)(I-\Pi_0) \|_{\BBB(E,E)} \le C e^{-(\lambda_0 - \lambda)t} \in L^1_t(\R_+),
$$
and also, thanks to Corollary~\ref{cor:reg},
$$
\| (e_\lambda\AA \SS_{\BB})^{(*N)}(t) \|_{\BB(\EE, E)} \le C e^{- (\lambda_D - \lambda) t}
\in L^1_t(\R_+).
$$
These estimates finally yield
$$
\bal
&\| \SS_{N+1}(t) \|_{ \BB (\EE_{-1}, \EE )} \\
&\qquad \le C \| e^{\lambda t}\SS_{\Lambda}(t)(I-\Pi_0) \|_{\BBB(E,E)} * \| (e_\lambda\AA \SS_{\BB})^{(*N)}(t) \|_{\BB(\EE, E)}  * \| e^{\lambda t}\AA \SS_\BB(t) \|_{\BBB(\EE_{-1} , \EE)}
\in L^2_t (\R_+),
\eal
$$
which completes the proof of \eqref{eq:SLambda-reg}.
\end{proof}

\Black

\section{The nonlinear equation}\label{sec:NL}

This section is devoted to the proof of Theorem~\ref{main1}.
We develop a perturbative Cauchy theory for the (nonlinear) Landau equation
using the estimates on the linearized operator obtained in the previous section.

\subsection{Functional spaces}
We recall the following definitions
$$
\| f \|_{H^1_{v,*}(m)}^2 = \| \la v \ra^{\frac{\gamma+\sigma}{2}} f \|_{L^2_v(m)}^2
+ \| \la v \ra^{\frac{\gamma}{2}} P_v \nabla_v f \|_{L^2_v(m)}^2
+ \| \la v \ra^{\frac{\gamma+2}{2}} (I-P_v) \nabla_v f \|_{L^2_v(m)}^2,
$$
and we also define the (stronger) norm
$$
\| f \|_{H^1_{v,**}(m)}^2 =  \| \la v \ra^{\frac{\gamma+2}{2}} f \|_{L^2_v(m)}^2
+ \| \la v \ra^{\frac{\gamma}{2}} P_v \nabla_v f \|_{L^2_v(m)}^2
+ \| \la v \ra^{\frac{\gamma+2}{2}} (I-P_v) \nabla_v f \|_{L^2_v(m)}^2.
$$
Recall the space $\HH^3_x L^2_v(m)$ defined in \eqref{HH3xL2v} associated to the norm
$$
\bal
\| f \|_{\HH^3_x L^2_v(m)}^2 = \sum_{0 \le j \le 3} \| \nabla^j_x f \|_{L^2_x L^2_v (m \la v \ra^{-j(1 - \sigma/2)})}^2,
\eal
$$
and also the space $\HH^3_x ( H^1_{v,*}(m) )$ defined in \eqref{HH3xH1v*} by
$$
\bal
\| f \|_{\HH^3_x (H^1_{v,*}(m))}^2 
&= \sum_{0 \le j \le 3} \| \nabla^j_x f \|_{L^2_x ( H^1_{v,*} (m \la v \ra^{-j(1 - \sigma/2)}))}^2. \\
&= \sum_{0 \le j \le 3} \int_{\T^3_x} \| \nabla^j_x f \|_{ H^1_{v,*} (m \la v \ra^{-j(1 - \sigma/2)})}^2. 
\eal
$$
We define in a similar way the space $\HH^3_x ( H^1_{v,**}(m))$ using the norm $H^1_{v,**}(m)$ (instead of $H^1_{v,*}(m)$).
We also define the negative Sobolev space $\HH^3_x ( H^{-1}_{v,*}(m) )$ by duality in the following way
\beqn\label{HH3xH-1v*}
\bal
\| f \|_{\HH^3_x ( H^{-1}_{v,*}(m) )} 
&:= \sup_{\| \phi \|_{\HH^3_x ( H^1_{v,*}(m) )} \le 1} \;
 \la f , \phi \ra_{\HH^3_x L^2_v(m)} \\
&:= \sup_{\| \phi \|_{\HH^3_x ( H^1_{v,*}(m) )} \le 1}
 \sum_{0 \le j \le 3} \; \la \nabla^j_x f , \nabla^j_x \phi \ra_{L^2_x L^2_v (m \la v \ra^{-j(1 - \sigma/2)})}.
\eal
\eeqn

\medskip

The results on the linearized operator $\Lambda$ in Theorems~\ref{thm:extension} and \ref{thm:SL-regularity} are stated for spaces of the type $H^3_x L^2_v (m)$, but they can be easily adapted for the spaces $\HH^3_x L^2_v (m)$ above, more precisely we have:

\begin{cor}\label{cor:decay&regularity}
Consider hypothesis {\bf (H1), (H2) or (H3)} and some weight function $m$, with the additional assumption $k > \gamma + 5 + 3/2$ in the case of polynomial weight $m = \la v \ra^k$. Then for any $\lambda < \lambda_{m,2}$ and any $\lambda_1 \le \min \{ \lambda_0, \lambda  \}$, there exists a constant $C>0$ such that
$$
\forall \, t \ge 0, \, \forall f \in \HH^3_x L^2_v (m), \quad
\| \SS_{\Lambda}(t) (I - \Pi_0) f \|_{\HH^3_x L^2_v (m)} \le C \, e^{- \lambda_1 t} \, \|  (I - \Pi_0) f \|_{\HH^3_x L^2_v (m)} .
$$ 
Moreover, for any $\lambda < \lambda_1$,
$$
\int_0^\infty e^{2 \lambda t} \, \| \SS_{\Lambda}(t) (I - \Pi_0) f \|_{\HH^3_x L^2_v (m)}^2 \, dt \le C \| (I - \Pi_0) f \|_{\HH^3_x (H^{-1}_{v,*} (m) )}^2.
$$

\end{cor}

\Black

\subsection{Dissipative norm for the linearized equation}
We construct now a norm for which the linearized semigroup $\SS_\Lambda(t)$ is dissipative, with a rate as close as we want to the optimal rate decay from Theorem~\ref{thm:extension}, and also has a stronger dissipativity property.

\begin{prop}\label{prop:dissipative}
Consider some weight function $m$ satisfying {\bf (H0)}, and let $X := \HH^3_x L^2_v(m)$ and $Y := \HH^3_x (H^1_{v,*}(m))$. Consider another weight function $\bar m$ satisfying {\bf (H1)-(H2)-(H3)} with $\bar m \lesssim m \la v \ra^{-(1-\sigma/2)}$ and denote $\bar X := \HH^3_x L^2_{v}( \bar m)$.

Define for any $\eta >0$ and any $\lambda_2 < \lambda_1$ (where $\lambda_1 >0$ is the optimal rate in Theorem~\ref{thm:extension}) the equivalent norm on $X$
\beqn\label{norm-diss}
\Nt f \Nt_{X}^2 := \eta \| f \|_{X}^2 + \int_0^\infty \| \SS_{\Lambda}(\tau) e^{\lambda_2 \tau} f \|_{\bar X}^2 \, d\tau.
\eeqn
Then there is $\eta>0$ small enough such that the solution $f_t = \SS_\Lambda(t) f$ to the linearized equation satisfies, for any $t \ge 0$ and some constant $K>0$,
$$
\frac12\frac{d}{dt} \Nt \SS_{\Lambda}(t) f \Nt_{X}^2
\le - \lambda_2 \Nt \SS_{\Lambda}(t) f \Nt_{X}^2  - K \| \SS_{\Lambda}(t) f \|_{Y}^2, \quad \forall\, f \in X, \, \Pi_0 f=0.
$$
\end{prop}

\begin{proof}
First we remark that the norm $\Nt \cdot \Nt_{\HH^3_x L^2_v(m)}$ is equivalent to the norm $\| \cdot \|_{\HH^3_x L^2_v(m)}$ defined in \eqref{HH3xL2v} for any $\eta >0$ and any $\lambda_2 < \lambda_1$. Indeed, using Corollary~\ref{cor:decay&regularity}, we have
$$
\bal
\eta \| f \|_{\HH^3_x L^2_v(m)}^2
&\le \Nt f \Nt_{\HH^3_x L^2_v(m)}^2 = \eta \| f \|_{\HH^3_x L^2_v(m)}^2  + \int_0^\infty \| \SS_{\Lambda}(\tau) e^{\lambda_2 \tau} f \|_{\HH^3_x L^2_v(\bar m)}^2 \, d\tau \\
&\le \eta \| f \|_{\HH^3_x L^2_v(m)}^2 + \int_0^\infty C^2 e^{- 2 (\lambda_1 - \lambda_2) \tau}  \| f \|_{\HH^3_x L^2_v(\bar m)}^2 \, d\tau \le (\eta + C) \| f \|_{\HH^3_x L^2_v(m)}^2.
\eal
$$
We now compute, denoting $f_t = \SS_\Lambda(t) f$,
$$
\bal
\frac12\frac{d}{dt} \Nt f_t \Nt_{\HH^3_x L^2_v(m)}^2
 = \eta \la \Lambda f_t \,  f_t \ra_{\HH^3_x L^2_v(m)}
 + \frac12 \int_0^\infty \frac{\partial}{\partial t} \| \SS_\Lambda(\tau) e^{\lambda_2 t} f_{t} \|_{\HH^3_x L^2_v(\bar m)}^2 \, d\tau
 =: I_1 + I_2.
\eal
$$
For $I_1$ we write $\Lambda = \AA + \BB$.
Arguing exactly as in Section~\ref{sec:lin}, more precisely Lemma~\ref{lem:A0}, we first obtain that $\AA \in \BB({\HH^3_x L^2_v(\bar m)}, \HH^3_x L^2_v(\mu^{-1/2}))$, whence
$$
\la \AA f_t , f_t \ra_{\HH^3_x L^2_v(m)} \le C \| f_t \|_{_{\HH^3_x L^2_v(\bar m)}}.
$$
Moreover, repeating the estimates for the hypodissipativity of $\BB$ in Lemmas~\ref{lem:hypo} and \ref{lem:hypo2} we easily get, for any $\lambda_2 \le \lambda < \lambda_{m,2}$ and some $K>0$,
$$
\bal
\la \BB f , f \ra_{\HH^3_x L^2_v(m)}
&\le - \lambda \| f \|_{\HH^3_x L^2_v(m)}^2 - K \| f \|_{\HH^3_x (H^1_{v,*}(m))}^2,
\eal
$$
therefore it follows
$$
I_1
\le  - \lambda \eta \| f_t \|_{\HH^3_x L^2_v(m)}^2 - \eta K \| f_t \|_{\HH^3_x (H^1_{v,*}(m))}^2 + \eta C \| f_t \|_{\HH^3_x L^2_v(\bar m)}^2 .
$$
The second term is computed exactly
$$
\bal
I_2
&= \frac12 \int_0^\infty \frac{\partial}{\partial t} \| \SS_{\Lambda}(\tau + t) e^{\lambda_2 \tau} f \|_{\HH^3_x L^2_v(\bar m)}^2 \, d\tau  \\
&= \frac12 \int_0^\infty \frac{\partial}{\partial \tau} \| \SS_{\Lambda}(\tau + t) e^{\lambda_2 \tau} f \|_{\HH^3_x L^2_v(\bar m)}^2 \, d\tau
- \lambda_2 \int_0^\infty \| \SS_\Lambda(\tau) e^{\lambda_2 \tau} f_{t} \|_{\HH^3_x L^2_v(\bar m)}^2 \, d\tau \\
&= \frac12\left[  \| \SS_\Lambda(\tau) e^{\lambda_2 \tau} f_{t} \|_{\HH^3_x L^2_v(\bar m)}^2  \right]_{\tau=0}^{\tau=+\infty} - \lambda_2 \int_0^\infty \| \SS_\Lambda(\tau) e^{\lambda_2 \tau} f_{t} \|_{\HH^3_x L^2_v(\bar m)}^2 \, d\tau \\
&= -\frac12  \| f_{t} \|_{\HH^3_x L^2_v(\bar m)}^2  - \lambda_2 \int_0^\infty \| \SS_\Lambda(\tau) e^{\lambda_2 \tau} f_{t} \|_{\HH^3_x L^2_v(\bar m)}^2 \, d\tau
\eal
$$
where we have used the semigroup decay from Corollary~\ref{cor:decay&regularity}.

Gathering previous estimates and using that $\lambda \ge \lambda_2$
we obtain
$$
\bal
I_1 + I_2
&\le -\lambda_2 \left\{ \eta \| f_t \|_{\HH^3_x L^2_v(m)}^2 + \int_0^\infty \| \SS_\Lambda(\tau) e^{\lambda_2 \tau} f_{t} \|_{\HH^3_x L^2_v(\bar m)}^2 \, d\tau  \right\} \\
&\quad
 - \eta K \| f_t \|_{\HH^3_x (H^1_{v,*}(m))}^2 + \eta C\| f_t \|_{\HH^3_x L^2_v(\bar m)}^2
- \frac12\| f_t \|_{\HH^3_x L^2_v(\bar m)}^2.
\eal
$$
We complete the proof choosing $\eta>0$ small enough.
\end{proof}

\subsection{Nonlinear estimates}
We prove in this section some estimates for the nonlinear operator $Q$.
We will use the following auxiliary results.
\begin{lem}\label{lem:Aalpha}
Let $-3<\alpha<0$ and $\theta>3$. Then
$$
A_\alpha(v) := \int_{\R^3} |v-v_*|^\alpha \, \la v_* \ra^{-\theta} \, dv_* \lesssim \la v \ra^\alpha.
$$
\end{lem}

\begin{proof}
Let $|v| \le 1/2$, thus $|v_*| + 1/2 \le 1 + |v-v_*|$ and we get
$$
A_\alpha(v) = \int_{\R^3} |v_*|^\alpha \, \la v-v_* \ra^{-\theta} \, dv_*
\lesssim \int_{\R^3} |v_*|^\alpha \, \la v_* \ra^{-\theta} \, dv_* \lesssim \la v \ra^\alpha.
$$
Consider now $|v| >1/2$ and split the integral into two regions: $|v-v_*| > \la v \ra/4$ and $|v-v_*| \le \la v \ra/4$. For the first region we obtain
$$
\int_{\R^3} {\mathbf 1}_{|v-v_*| > \frac{\la v \ra}{4}} \, |v-v_*|^\alpha \, \la v_* \ra^{-\theta} \, dv_* \lesssim \la v \ra^\alpha \int_{\R^3} \la v_* \ra^{-\theta} \, dv_*
\lesssim \la v \ra^\alpha.
$$
For the second region, $|v| >1/2$ and $|v-v_*| \le \la v \ra/4$ imply $|v_*| \ge |v|/4$, hence
$$
\int_{\R^3} {\mathbf 1}_{|v-v_*| \le \frac{\la v \ra}{4}} \, |v-v_*|^\alpha \, \la v_* \ra^{-\theta} \, dv_*
\lesssim \la v \ra^{-\theta} \int_{\R^3} {\mathbf 1}_{|v-v_*| \le \frac{\la v \ra}{4}} \, |v-v_*|^\alpha  \, dv_*
\lesssim \la v \ra^{-\theta + \alpha + 3}
\lesssim \la v \ra^{\alpha}.
$$

\end{proof}

\begin{lem}\label{lem:abc*f}
There holds:
\begin{enumerate}[(i)]

\item For any $\theta > \gamma + 4 + 3/2$
$$
|(a_{ij}*f)(v) \, v_i v_j|
+ |(a_{ij}*f)(v) \, v_i |
+ |(a_{ij}*f)(v)  |   \lesssim \la v \ra^{\gamma+2} \, \| f \|_{L^2_v (\la v \ra^{\theta})}.
$$

\item For any $\theta' > (\gamma + 1)_+ + 3/2$ (where $x_+ := \max \{ x,0\}$)
$$
|(b_{j}*f)(v) |  \lesssim \la v \ra^{\gamma+1} \, \| f \|_{L^2_v (\la v \ra^{\theta'})}.
$$

\item If $\gamma \in [0,1]$, for any $\theta'' > \gamma + 3/2$
$$
|(c*f)(v)| \lesssim \la v \ra^{\gamma} \, \| f \|_{L^2_v (\la v \ra^{\theta''})}.
$$

\item If $\gamma \in [-2,0)$, for any $p > \frac{3}{3+\gamma}$ and $\theta'' > 3(1-1/p)$
$$
|(c*f)(v)| \lesssim \la v \ra^{\gamma} \, \| f \|_{L^p_v (\la v \ra^{\theta''})}.
$$
In particular, when $\gamma \in (-3/2,0)$ we can choose $p=2$ and $\theta''>3/2$; and when $\gamma \in [-2,-3/2]$ we can choose $p=4$ and $\theta''>9/4$.

\end{enumerate}
\end{lem}

\begin{proof}
Recall that $0$ is an eigenvalue of the matrix $a_{ij}$ so that $a_{ij}(v-v_*) v_i = a_{ij}(v-v_*) v_{*i} $ and $a_{ij}(v-v_*) v_i v_j = a_{ij}(v-v_*) v_{*i} v_{*j}$. Using this we can easily obtain, for any $\theta > \gamma + 4 +3/2$,
$$
\bal
| (a_{ij} * f)(v) \, v_i v_j |
&= |\int_{v_*} a_{ij}(v-v_*) v_i v_j f_*|
= \left|\int_{v_*} a_{ij}(v-v_*) v_{*i} v_{*j} f_*\right| \\
&\lesssim \int_{v_*} \la v \ra^{\gamma+2} \la v_* \ra^{\gamma+4}  |f_*|
\lesssim \la v \ra^{\gamma + 2} \| f \|_{L^1_v (\la v \ra^{\gamma+4})} \\
&\lesssim \la v \ra^{\gamma + 2} \| f \|_{L^2_v (\la v \ra^{\theta})}.
\eal
$$
In a similar way we get
$$
| (a_{ij} * f)(v) \, v_i  | \lesssim \la v \ra^{\gamma+2} \| f \|_{L^2_v (\la v \ra^{\theta-1})},
$$
and we easily have, since $\gamma \in [-2,1]$,
$$
| (a_{ij} * f)(v)   | \lesssim \la v \ra^{\gamma+2} \| f \|_{L^2_v (\la v \ra^{\theta-2})}.
$$
For the term $(b*f)$, we recall that $b_i(z) = -2 |z|^\gamma z_i$ and we separate into two cases. When $\gamma \in [-1,1]$ we have, for any $\theta' > \gamma+1 + 3/2$,
$$
\bal
| (b_{i} * f)(v)  | &\lesssim  \int_{v_*} |v-v_*|^{\gamma+1} \, |f_*|
\lesssim \int_{v_*} \la v \ra^{\gamma+1} \la v_* \ra^{\gamma+1}  |f_*| \\
&\lesssim \la v \ra^{\gamma + 1} \| f \|_{L^1_v (\la v \ra^{\gamma+1})}
\lesssim \la v \ra^{\gamma + 1} \| f \|_{L^2_v (\la v \ra^{\theta'})}.
\eal
$$
When $\gamma \in [-2,-1)$ we use Lemma~\ref{lem:Aalpha} to obtain, for any $\theta' >3/2$,
$$
\bal
| (b_{i} * f)(v)  |
&\lesssim  \int_{v_*} |v-v_*|^{\gamma+1} \, \la v_* \ra^{-\theta'}  \la v_* \ra^{\theta'} |f_*|
\lesssim \left(\int_{v_*} |v-v_*|^{2(\gamma+1)} \, \la v_* \ra^{-2\theta'} \right)^{1/2} \| f \|_{L^2_v (\la v \ra^{\theta'})} \\
&\lesssim \la v \ra^{\gamma+1} \, \| f \|_{L^2_v (\la v \ra^{\theta'})}.
\eal
$$
Finally for the last term $(c*f)$, recall that $c(z) = -2(\gamma+3)|z|^\gamma$ and separate into two cases. When $\gamma \in [0,1]$ then, for any $\theta'' > \gamma+3/2$,
$$
\bal
| (c * f)(v)  | &\lesssim  \int_{v_*} |v-v_*|^{\gamma} \, |f_*|
\lesssim \int_{v_*} \la v \ra^{\gamma} \la v_* \ra^{\gamma}  |f_*| \\
&\lesssim \la v \ra^{\gamma } \| f \|_{L^1_v (\la v \ra^{\gamma})}
\lesssim \la v \ra^{\gamma} \| f \|_{L^2_v (\la v \ra^{\theta''})}.
\eal
$$
When $\gamma \in [-2,0)$ we use Lemma~\ref{lem:Aalpha} to obtain, for any $p > \frac{3}{3+\gamma}$ and for any $\theta'' >3(1-1/p)$,
$$
\bal
| (c * f)(v)  |
&\lesssim  \int_{v_*} |v-v_*|^{\gamma} \, \la v_* \ra^{-\theta''}  \la v_* \ra^{\theta''} |f_*|
\lesssim \left(\int_{v_*} |v-v_*|^{\gamma \frac{p}{p-1}} \, \la v_* \ra^{-\theta'' \frac{p}{p-1}} \right)^{(p-1)/p} \| f \|_{L^p_v (\la v \ra^{\theta''})} \\
&\lesssim \la v \ra^{\gamma} \, \| f \|_{L^2_v (\la v \ra^{\theta''})},
\eal
$$
thanks to $|\gamma| p/(p-1) <3$.
\end{proof}

We now prove nonlinear estimates for the Landau operator $Q$.

\begin{lem}\label{lem:NL1}
Consider hypothesis {\bf (H1)}, {\bf (H2)} or ${\bf (H3)}$.
\begin{enumerate}[(i)]

\item For any $\theta > \gamma+4 + 3/2$, there holds
$$
\la  Q(f,g),h \ra_{L^{2}_v(m)}
\lesssim \| f \|_{L^2_v(\la v \ra^{\theta})} \, \| g \|_{H^1_{v,**}(m)} \, \| h \|_{H^1_{v,*}(m)}.
$$

\item For any $\theta > \gamma+4 + 3/2$ and $\theta'>9/4$, there holds
$$
\la  Q(f,g),g \ra_{L^{2}_v(m)}
\lesssim
\| f \|_{L^2_v(\la v \ra^{\theta})} \, \| g \|_{H^1_{v,*}(m)}^2, \quad \text{if } \gamma\in (-3/2,1] ;
$$
and
$$
\bal
\la  Q(f,g),g \ra_{L^{2}_v(m)}
&\lesssim
\| f \|_{L^2_v(\la v \ra^{\theta})} \, \| g \|_{H^1_{v,*}(m)}^2
+ \| f \|_{H^1_v(\la v \ra^{\theta'})} \, \| g \|_{L^2_{v}(m)}^2 \quad \text{if } \gamma\in [-2,-3/2].
\eal
$$

\end{enumerate}
\end{lem}

\begin{proof}
We write
$$
\bal
\la Q(f,g) , h \ra_{L^2_v(m)}
&= \int \partial_j \{ (a_{ij} * f) \partial_{i} g  - (b_j*f) g \} \, h \, m^2 \\
&= - \int (a_{ij}*f) \partial_i g \, \partial_j h \, m^2 - \int (a_{ij}*f) \partial_i g \, \partial_j m^2 \, h \\
&\quad
+ \int (b_{j}*f)  g \, \partial_j h \, m^2
+ \int (b_j*f) g \, h \, \partial_j m^2 \\
&=: T_1 + T_2 + T_3 + T_4.
\eal
$$

\smallskip
{\it Step 1. Point $(i)$.}
We estimate each term separately.

{\it Step 1.1.} For the first term, since the estimate for $|v|\le 1$ is evident, we only consider the case $|v| >1$. We decompose $\partial_i g = P_v \partial_i g + (I-P_v) \partial_i g$ and similarly for $\partial_j h$, where we recall that $P_v \partial_i g = v_i |v|^{-2} (v \cdot \nabla_v g)$. We hence write
$$
\bal
T_1 &= \int (a_{ij}*f) \, \{ P_v \partial_i g \, P_v \partial_j h + P_v \partial_i g \, (I-P_v) \partial_j h + (I-P_v) \partial_i g \, P_v \partial_j h
 + (I-P_v) \partial_i g \, (I-P_v) \partial_j h \}\, m^2 \\
 &=: T_{11} + T_{12} + T_{13} + T_{14}.
\eal
$$
Therefore we have, using Lemma \ref{lem:abc*f},
$$
\bal
T_{11}
& = \int (a_{ij}*f) v_i v_j \, \frac{(v \cdot \nabla_v g)}{|v|^2} \, \frac{(v \cdot \nabla_v h)}{|v|^2} \, m^2 \\
&\lesssim \| f \|_{L^2_v(\la v \ra^\theta)} \int \la v \ra^{\gamma+2}  |v|^{-2}\, |\nabla_v g| \, |\nabla_v h|  \, m^2 \\
&\lesssim \| f \|_{L^2_v(\la v \ra^\theta)}  \, \| \la v \ra^{\frac{\gamma}{2}} \nabla_v g \|_{L^2_v(m)}
\, \| \la v \ra^{\frac{\gamma}{2}} \nabla_v h \|_{L^2_v(m)}.
\eal
$$
Moreover
$$
\bal
T_{12}
& = \int (a_{ij}*f) v_i  \, \frac{(v \cdot \nabla_v g)}{|v|^2} \, \{(I-P_v)\partial_j h\} \, m^2 \\
&\lesssim \| f \|_{L^2_v(\la v \ra^\theta)} \int \la v \ra^{\gamma+2}  |v|^{-1}\, |\nabla_v g| \, |(I-P_v)\nabla_v h|  \, m^2 \\
&\lesssim \| f \|_{L^2_v(\la v \ra^\theta)}  \, \| \la v \ra^{\frac{\gamma}{2}} \nabla_v g \|_{L^2_v(m)}
\, \| \la v \ra^{\frac{\gamma+2}{2}} (I-P_v)\nabla_v h \|_{L^2_v(m)},
\eal
$$
and similarly
$$
\bal
T_{13}
&\lesssim \| f \|_{L^2_v(\la v \ra^\theta)}  \, \| \la v \ra^{\frac{\gamma+2}{2}} (I-P_v) \nabla_v g \|_{L^2_v(m)}
\, \| \la v \ra^{\frac{\gamma}{2}} \nabla_v h \|_{L^2_v(m)}.
\eal
$$
For the term $T_{14}$ we obtain
$$
\bal
T_{14}
& = \int (a_{ij}*f)   \, \{ (I-P_v)\partial_i g \} \, \{ (I-P_v)\partial_j h\} \, m^2 \\
&\lesssim \| f \|_{L^2_v(\la v \ra^\theta)} \int \la v \ra^{\gamma+2}   |(I-P_v)\nabla_v g| \, |(I-P_v)\nabla_v h|  \, m^2 \\
&\lesssim \| f \|_{L^2_v(\la v \ra^\theta)}  \, \| \la v \ra^{\frac{\gamma+2}{2}} (I-P_v) \nabla_v g \|_{L^2_v(m)}
\, \| \la v \ra^{\frac{\gamma+2}{2}} (I-P_v)\nabla_v h \|_{L^2_v(m)}.
\eal
$$

{\it Step 1.2.}
Let us investigate the second term $T_2$, and again we only consider $|v|>1$. Since $\partial_j m^2 = C v_j \la v \ra^{\sigma-2} m^2$, where we recall that $\sigma=0$ when $m=\la v \ra^k$ and $\sigma = s$ when $m=e^{r\la v \ra^s}$, the same argument as for $T_1$ gives us
$$
\bal
T_2 &= \int (a_{ij}*f) \, \{ P_v \partial_i g \, \partial_j m^2
+  (I-P_v) \partial_i g \, \partial_j m^2  \}\, h \\
 &=: T_{21} + T_{22}.
\eal
$$
Then we have
$$
\bal
T_{21} &= C \int (a_{ij}*f) v_i v_j \la v \ra^{\sigma-2} \, \frac{(v \cdot \nabla_v g)}{|v|^2} \, h \, m^2 \\
&\lesssim \| f \|_{L^2_v(\la v \ra^\theta)} \int \la v \ra^{\gamma+2} \la v \ra^{\sigma-2} \, |v|^{-1} \, |\nabla_v g| \, |h| \, m^2 \\
&\lesssim \| f \|_{L^2_v(\la v \ra^\theta)} \, \| \la v \ra^{\frac{\gamma+\sigma-2}{2}} \nabla_v g \|_{L^2_v(m)} \,
\| \la v \ra^{\frac{\gamma+\sigma}{2}}  h \|_{L^2_v(m)},
\eal
$$
and we recall that $\gamma+\sigma-2 \le \gamma$. For the other term we get
$$
\bal
T_{21} &= C \int (a_{ij}*f)  v_j \la v \ra^{\sigma-2} \, \{ (I-P_v)\partial_i g \} \, h \, m^2 \\
&\lesssim \| f \|_{L^2_v(\la v \ra^\theta)} \int \la v \ra^{\gamma+2} \la v \ra^{\sigma-2}  \, |(I-P_v)\nabla_v g| \, |h| \, m^2 \\
&\lesssim \| f \|_{L^2_v(\la v \ra^\theta)} \, \| \la v \ra^{\frac{\gamma+\sigma}{2}} (I-P_v)\nabla_v g \|_{L^2_v(m)} \,
\| \la v \ra^{\frac{\gamma+\sigma}{2}}  h \|_{L^2_v(m)},
\eal
$$
and recall that $\gamma + \sigma \le \gamma + 2$.

{\it Step 1.3.} For the term $T_4$,
$$
\bal
T_4 &= C \int (b_j * f) \, v_j \la v \ra^{\sigma-2} \, g \, h \, m^2 \\
&\lesssim \| f \|_{L^2_v(\la v \ra^\theta)} \int \la v \ra^{\gamma+1} \la v \ra^{\sigma-1} \, |g| \, |h| \, m^2 \\
&\lesssim \| f \|_{L^2_v(\la v \ra^\theta)}  \, \| \la v \ra^{\frac{\gamma+\sigma}{2}} g \|_{L^2_v(m)} \, \| \la v \ra^{\frac{\gamma+\sigma}{2}}  h \|_{L^2_v(m)}.
\eal
$$
Remark that up to now we have obtained
$$
T_1 + T_2 + T_4 \lesssim \| f \|_{L^2_v( \la v \ra^\theta)} \, \| g \|_{H^1_{v,*} (m)} \, \| h \|_{H^1_{v,*}(m)},
$$
however in the estimate of the term $T_3$ (see below) we will get a worst estimate (with the norm $\| g \|_{H^1_{v,**} (m)}$ instead of $\| g \|_{H^1_{v,*} (m)}$).

{\it Step 1.4.} We finally investigate the term $T_3$ and we get
$$
\bal
T_3 &\lesssim \| f \|_{L^2_v( \la v \ra^\theta)} \int \la v \ra^{\gamma+1} \, |g| \, |\nabla_v h| \, m^2 \\
&\lesssim \| f \|_{L^2_v( \la v \ra^\theta)} \, \| \la v \ra^{\frac{\gamma+2}{2}} g \|_{L^2_v(m)} \, \| \la v \ra^{\frac{\gamma}{2}}  \nabla_v h \|_{L^2_v(m)} \\
&\lesssim \| f \|_{L^2_v( \la v \ra^\theta)} \, \| g \|_{H^1_{v,**}(m)} \, \| \la v \ra^{\frac{\gamma}{2}}  \nabla_v h \|_{L^2_v(m)}.
\eal
$$
We complete the proof of point $(i)$ gathering previous estimates.

\smallskip
{\it Step 2. Point $(ii)$.}
Arguing as in Step 1, with $h$ replaced by $g$, we already have
$$
T_1 + T_2 + T_4 \lesssim \| f \|_{L^2_v( \la v \ra^\theta)} \, \| g \|_{H^1_{v,*} (m)}^2,
$$
and we only estimate the term $T_3$. Integrating by parts we get
$$
\bal
T_3 &= \int (b_j *f) \,g \, \partial_j g  \, m^2
&= - \frac12 \int (c*f) \, g^2 \, m^2 - \frac12 \int (b_j*f)\, \partial_j m^2 \, g^2 =: I + II.
\eal
$$
The term $II$ can be estimated exactly as $T_4$. For $I$, thanks to Lemma~\ref{lem:abc*f}, we obtain
$$
I \lesssim \| f \|_{L^2_v( \la v \ra^\theta)} \, \| \la v \ra^{\frac{\gamma}{2}} g \|_{L^2_v(m)}^2, \quad\text{if } \gamma \in (-3/2,1];
$$
and
$$
\bal
I &\lesssim \| f \|_{L^2_v( \la v \ra^\theta)} \, \| \la v \ra^{\frac{\gamma}{2}} g \|_{L^2_v(m)}^2
+ \| f \|_{L^4_v( \la v \ra^{\theta'})} \, \| \la v \ra^{\frac{\gamma}{2}} g \|_{L^2_v(m)}^2 , \quad\text{if } \gamma \in [-2,-3/2]; \\
&\lesssim \| f \|_{L^2_v( \la v \ra^\theta)} \, \| \la v \ra^{\frac{\gamma}{2}} g \|_{L^2_v(m)}^2
+ \| f \|_{H^1_v( \la v \ra^{\theta'})} \, \| \la v \ra^{\frac{\gamma}{2}} g \|_{L^2_v(m)}^2
\eal
$$
and that concludes the proof.
\end{proof}

\begin{lem}\label{lem:Qf}
Let assumption {\bf (H0)} be in force.

\begin{enumerate}[(i)]

\item There holds
$$
\la  Q(f,g), h \ra_{\HH^3_x L^2_v(m)}
\lesssim  \| f \|_{\HH^3_x L^2_v(m)} \, \| g \|_{\HH^3_x (H^1_{v,**}(m))} \, \| h \|_{\HH^3_x (H^1_{v,*}(m))},
$$
therefore
$$
\| Q(f,g) \|_{\HH^3_x (H^{-1}_{v,*}(m))} \lesssim \| f \|_{\HH^3_x L^2_v(m)} \, \| g \|_{\HH^3_x (H^1_{v,**}(m))}.
$$

\item There holds
$$
\bal
\la  Q(f,g),g \ra_{\HH^3_x L^2_v(m)}
&\lesssim  \| f \|_{\HH^3_x L^2_v(m)} \, \| g \|_{\HH^3_x ( H^1_{v,*}(m) )}^2 
+ \| f \|_{\HH^3_x ( H^1_{v,*}(m)) } \, \| g \|_{\HH^3_x L^2_{v}(m)}^2.
\eal
$$

\end{enumerate}

\end{lem}

\begin{proof}
We only prove point $(ii)$. Point $(i)$ can be proven in the same manner, using the estimate of Lemma~\ref{lem:NL1}-$(i)$ instead of Lemma~\ref{lem:NL1}-$(ii)$ as we shall do next.

We write
$$
\la  Q(f,g),g \ra_{\HH^3_x L^2_v(m)} = \la  Q(f,g),g \ra_{L^2_x L^2_v(m)}
+ \sum_{1 \le |\beta| \le 3} \la \partial^\beta_x Q(f,g), \partial^\beta_x g \ra_{L^2_x L^2_v(m \la v \ra^{-|\beta|(1-\sigma/2)})},
$$
and
$$
\partial^\beta_x Q(f,g) = \sum_{\beta_1 + \beta_2 = \beta} C_{\beta_1,\beta_2} Q ( \partial^{\beta_1}_x f , \partial^{\beta_2}_x g).
$$
The proof of the lemma is a consequence of Lemma~\ref{lem:NL1} together with the following inequalities, that we shall use in the sequel when integrating in $x \in \T^3$,
\beqn\label{eq:Sob_x}
\| u \|_{L^\infty (\T^3_x)} \lesssim \| u \|_{H^2 (\T^3_x)}, \quad
\| u \|_{L^6 (\T^3_x)} \lesssim \| u \|_{H^1 (\T^3_x)}, \quad
\| u \|_{L^3 (\T^3_x)} \lesssim \|  u \|_{H^1 (\T^3_x)}^{1/2} \, \| u \|_{L^2 (\T^3_x)}^{1/2}.
\eeqn

\smallskip\noindent
{\it Step 1.}
Using Lemma~\ref{lem:NL1}-$(ii)$ and \eqref{eq:Sob_x} we easily get, for $\theta > \gamma + 4 + 3/2$ and $\theta' > 9/4$,
$$
\bal
\la  Q(f,g),g \ra_{L^2_x L^2_v(m)} 
&\lesssim 
\int_{\T^3_x} \| f \|_{L^2_v(\la v \ra^{\theta})} \, \| g \|_{ H^1_{v,*}(m)}^2
+ \| f \|_{ H^1_v(\la v \ra^{\theta'})} \, \| g \|_{L^2_{v}(m)}^2 \\
&\lesssim 
\| f \|_{H^2_x L^2_v(\la v \ra^{\theta})} \, \| g \|_{L^2_x (H^1_{v,*}(m))}^2
+ \| f \|_{H^2_x (H^1_v(\la v \ra^{\theta'}))} \, \| g \|_{L^2_x L^2_{v}(m)}^2.
\eal
$$

\smallskip\noindent
{\it Step 2. Case $|\beta|=1$.} Arguing as in the previous step, from Lemma~\ref{lem:NL1}-$(ii)$ and \eqref{eq:Sob_x}, it follows
$$
\bal
&\la  Q(f, \partial^{\beta}_x g ), \partial^{\beta}_x g \ra_{L^2_x L^2_v( m \la v \ra^{-(1 - \sigma/2)} )} \\ 
&\qquad
\lesssim \int_{\T^3_x} \| f \|_{L^2_v(\la v \ra^{\theta})} \, \| \nabla_x g \|_{ H^1_{v,*}(m \la v \ra^{-(1 - \sigma/2)})}^2   + \| f \|_{H^1_v(\la v \ra^{\theta'})} \, \| \nabla_x g \|_{L^2_{v}(m \la v \ra^{-(1 - \sigma/2)})}^2 \\
&\qquad
\lesssim \| f \|_{H^2_x L^2_v(\la v \ra^{\theta})} \, \| \nabla_x g \|_{L^2_x (H^1_{v,*}(m \la v \ra^{-(1 - \sigma/2)}))}^2   + \| f \|_{H^2_x (H^1_v(\la v \ra^{\theta'}))} \, \| \nabla_x g \|_{L^2_x L^2_{v}(m \la v \ra^{-(1 - \sigma/2)})}^2.
\eal
$$
Moreover, using now Lemma~\ref{lem:NL1}-$(i)$, we get
$$
\bal
&\la  Q( \partial^{\beta}_x f, g ), \partial^{\beta}_x g \ra_{L^2_x L^2_v(m \la v \ra^{-(1 - \sigma/2)})} \\
&\qquad
\lesssim \int_{\T^3_x} \| \nabla_x f \|_{L^2_v(\la v \ra^{\theta})}
\, \| g \|_{H^1_{v,**}(m \la v \ra^{-(1 - \sigma/2)})}
\, \| \nabla_x g \|_{ H^1_{v,*}(m \la v \ra^{-(1 - \sigma/2)})} \\
&\qquad
\lesssim \| \nabla_x f \|_{H^2_x L^2_v(\la v \ra^{\theta})}
\, \| g \|_{L^2_x (H^1_{v,**}(m \la v \ra^{-(1 - \sigma/2)}))}
\, \| \nabla_x g \|_{L^2_x (H^1_{v,*}(m \la v \ra^{-(1 - \sigma/2)}))}.
\eal
$$

\smallskip\noindent
{\it Step 3. Case $|\beta|=2$.} When $\beta_2=\beta$ we have
$$
\bal
&\la  Q(f, \partial^{\beta}_x g ), \partial^{\beta}_x g \ra_{L^2_x L^2_v( m \la v \ra^{-2(1 - \sigma/2)} )} \\
&\qquad
\lesssim \int_{\T^3_x}  \| f \|_{L^2_v(\la v \ra^{\theta})} \, \| \nabla^2_x g \|_{H^1_{v,*}(m \la v \ra^{-2(1 - \sigma/2)})}^2
+ \| f \|_{H^1_v(\la v \ra^{\theta'})} \, \| \nabla^2_x g \|_{L^2_{v}(m \la v \ra^{-2(1 - \sigma/2)})}^2 \\
&\qquad
\lesssim \| f \|_{H^2_x L^2_v(\la v \ra^{\theta})} \, \| \nabla^2_x g \|_{L^2_x (H^1_{v,*}(m \la v \ra^{-2(1 - \sigma/2)}))}^2  
+ \| f \|_{H^2_x (H^1_v(\la v \ra^{\theta'}))} \, \| \nabla^2_x g \|_{L^2_x L^2_{v}(m \la v \ra^{-2(1 - \sigma/2)})}^2.
\eal
$$
If $|\beta_1| = |\beta_2|=1$ then we obtain
$$
\bal
&\la  Q( \partial^{\beta_1}_x f, \partial^{\beta_2}_x g ), \partial^{\beta}_x g \ra_{L^2_x L^2_v(m \la v \ra^{-2(1 - \sigma/2)})} \\
&\qquad
\lesssim \int_{\T^3_x}  \| \nabla_x f \|_{L^2_v(\la v \ra^{\theta})}
\, \| \nabla_x g \|_{H^1_{v,**}(m \la v \ra^{-2(1 - \sigma/2)})}
\, \| \nabla^2_x g \|_{H^1_{v,*}(m \la v \ra^{-2(1 - \sigma/2)})} \\
&\qquad
\lesssim \| \nabla_x f \|_{H^2_x L^2_v(\la v \ra^{\theta})}
\, \| \nabla_x g \|_{L^2_x (H^1_{v,**}(m \la v \ra^{-2(1 - \sigma/2)}))}
\, \| \nabla^2_x g \|_{L^2_x (H^1_{v,*}(m \la v \ra^{-2(1 - \sigma/2)}))}.
\eal
$$
Finally, when $\beta_1=\beta$ we get
$$
\bal
& \la  Q( \partial^{\beta}_x f, g ), \partial^{\beta}_x g \ra_{L^2_x L^2_v(m \la v \ra^{-2(1 - \sigma/2)})}  \\
&\quad\lesssim \int_{\T^3_x} \| \nabla^2_x f \|_{L^2_v(\la v \ra^{\theta})}
\, \|  g \|_{H^1_{v,**}(m \la v \ra^{-2(1 - \sigma/2)})}
\, \| \nabla^2_x g \|_{H^1_{v,*}(m \la v \ra^{-2(1 - \sigma/2)})} \\
&\quad\lesssim \| \nabla^2_x f \|_{L^6_x L^2_v(\la v \ra^{\theta})}
\, \|  g \|_{L^3_x (H^1_{v,**}(m \la v \ra^{-2(1 - \sigma/2)}))}
\, \| \nabla^2_x g \|_{L^2_x (H^1_{v,*}(m \la v \ra^{-2(1 - \sigma/2)}))}\\
&\quad\lesssim \| \nabla^2_x f \|_{H^1_x L^2_v(\la v \ra^{\theta})}
\, \|  g \|_{L^2_x (H^1_{v,**}(m \la v \ra^{-2(1 - \sigma/2)}))}^{1/2}
\, \|  g \|_{H^1_x (H^1_{v,**}(m \la v \ra^{-2(1 - \sigma/2)}))}^{1/2}
\, \| \nabla^2_x g \|_{L^2_x (H^1_{v,*}(m \la v \ra^{-2(1 - \sigma/2)}))}.
\eal
$$

\smallskip\noindent
{\it Step 4. Case $|\beta|=3$.} When $\beta_2=\beta$ we obtain
$$
\bal
\la  Q(f, \partial^{\beta}_x g ), \partial^{\beta}_x g \ra_{L^2_x L^2_v( m \la v \ra^{-3(1 - \sigma/2)} )}
&\lesssim \| f \|_{H^2_x L^2_v(\la v \ra^{\theta})} \, \| \nabla^3_x g \|_{L^2_x (H^1_{v,*}(m \la v \ra^{-3(1 - \sigma/2)}))}^2  \\
&\quad + \| f \|_{H^2_x (H^1_v(\la v \ra^{\theta'}))} \, \| \nabla^3_x g \|_{L^2_x L^2_{v}(m \la v \ra^{-3(1 - \sigma/2)})}^2 .
\eal
$$
If $|\beta_1|=1$ and $|\beta_2|=2$ then
$$
\bal
&\la  Q( \partial^{\beta_1}_x f, \partial^{\beta_2}_x g ), \partial^{\beta}_x g \ra_{L^2_x L^2_v(m \la v \ra^{-3(1 - \sigma/2)})}  \\
&\quad\lesssim \int_{\T^3_x} \| \nabla_x f \|_{L^2_v(\la v \ra^{\theta})}
\, \| \nabla^2_x g \|_{H^1_{v,**}(m \la v \ra^{-3(1 - \sigma/2)})}
\, \| \nabla^3_x g \|_{H^1_{v,*}(m \la v \ra^{-3(1 - \sigma/2)})} \\
&\quad\lesssim \| \nabla_x f \|_{H^2_x L^2_v(\la v \ra^{\theta})}
\, \| \nabla^2_x g \|_{L^2_x (H^1_{v,**}(m \la v \ra^{-3(1 - \sigma/2)}))}
\, \| \nabla^3_x g \|_{L^2_x (H^1_{v,*}(m \la v \ra^{-3(1 - \sigma/2)}))}.
\eal
$$
When $|\beta_1|=2$ and $|\beta_2|=1$ then we get
$$
\bal
&\la  Q( \partial^{\beta_1}_x f, \partial^{\beta_2}_x g ), \partial^{\beta}_x g \ra_{L^2_x L^2_v(m \la v \ra^{-3(1 - \sigma/2)})} \\
&\qquad\lesssim \int_{\T^3_x} \| \nabla^2_x f \|_{L^2_v(\la v \ra^{\theta})}
\, \| \nabla_x g \|_{H^1_{v,**}(m \la v \ra^{-3(1 - \sigma/2)})}
\, \| \nabla^3_x g \|_{H^1_{v,*}(m \la v \ra^{-3(1 - \sigma/2)})} \\
&\qquad\lesssim \| \nabla^2_x f \|_{H^1_x L^2_v(\la v \ra^{\theta})}
\, \| \nabla_x g \|_{L^2_x (H^1_{v,**}(m \la v \ra^{-3(1 - \sigma/2)}))}^{1/2}
\, \| \nabla_x g \|_{H^1_x (H^1_{v,**}(m \la v \ra^{-3(1 - \sigma/2)}))}^{1/2}
\, \| \nabla^3_x g \|_{L^2_x (H^1_{v,*}(m \la v \ra^{-3(1 - \sigma/2)}))}.
\eal
$$
Finally, when $\beta_1=\beta$, it follows
$$
\bal
&\la  Q( \partial^{\beta}_x f, g ), \partial^{\beta}_x g \ra_{L^2_x L^2_v(m \la v \ra^{-3(1 - \sigma/2)})} \\
&\qquad\lesssim \int_{\T^3_x} \| \nabla^3_x f \|_{L^2_v(\la v \ra^{\theta})}
\, \|  g \|_{H^1_{v,**}(m \la v \ra^{-3(1 - \sigma/2)})}
\, \| \nabla^3_x g \|_{H^1_{v,*}(m \la v \ra^{-3(1 - \sigma/2)})} \\
&\qquad\lesssim \| \nabla^3_x f \|_{L^2_x L^2_v(\la v \ra^{\theta})}
\, \|  g \|_{H^2_x (H^1_{v,**}(m \la v \ra^{-3(1 - \sigma/2)}))}
\, \| \nabla^3_x g \|_{L^2_x (H^1_{v,*}(m \la v \ra^{-3(1 - \sigma/2)}))}.
\eal
$$

\smallskip\noindent
{\it Step 5. Conclusion.} We can conclude the proof gathering previous estimates and remarking that, for any $n=0,1,2$, there holds
$$
\| \la v \ra^{\frac{\gamma+2}{2}} \, \nabla^n_x g \|_{L^2_x L^2_{v}(m \la v \ra^{-(n+1)(1-\sigma/2 )} )} =
\| \la v \ra^{ \frac{\gamma+\sigma}{2}} \, \nabla^n_x g \|_{L^2_x L^2_{v}(m \la v \ra^{-n(1-\sigma/2 )} )},
$$
which implies
$$
\| \nabla^n_x g \|_{L^2_x ( H^1_{v,**}(m \la v \ra^{-(n+1)(1-\sigma/2)}))} \lesssim
\| \nabla^n_x g \|_{L^2_x ( H^1_{v,*}(m \la v \ra^{-n(1-\sigma/2)}))},
$$
and observing also that
$$
\| f \|_{H^3_x L^2_v(\la v \ra^{\theta})} \lesssim \| f \|_{\HH^3_x L^2_v(m)}
$$
and
$$
\| f \|_{H^2_x ( H^1_v(\la v \ra^{\theta'}))} \lesssim \| f \|_{\HH^3_x (H^1_{v,*}(m))}.
$$
\end{proof}

\subsection{Proof of Theorem \ref{main1}}

We consider the Cauchy problem for the perturbation  $f= F - \mu$. The equation satisfied by  $f=f(t,x,v)$ is
\beqn\label{eq:f}
\left\{
\bal
\partial_t f  &= \Lambda f + Q(f,f) \\
f_{|t=0} &= f_0= F_0 - \mu.
\eal
\right.
\eeqn
From the conservation laws (see \eqref{eq:conserv} and \eqref{laws}), for all $t >0$, $\Pi_0 f_t = 0$ since $\Pi_0 f_0 = 0$, more precisely $\int f_t(x,v) \, dx\, dv = \int v_j f_t(x,v) \, dx\, dv = \int |v|^2 f_t(x,v) \, dx\, dv=0$, and also $\Pi_0 Q(f_t,f_t) = 0$.

\medskip
Hereafter we fix some weight function $m$ that satisfies hypothesis {\bf (H0)}. We also fix a weight function $m_0$ satisfying the assumptions of Corollary~\ref{cor:decay&regularity} (i.e.\ $m_0$ satisfies {\bf (H1)}, {\bf (H2)} or {\bf (H3)} with the additional condition $k_0 > \gamma+5 +3/2$ if $m_0 = \la v \ra^{k_0}$) such that $m_0 \lesssim m \la v \ra^{-(1-\sigma/2)}$. Observe that this is always possible under the assumptions on $m$.

\medskip

We will construct solutions on $L^\infty_t (\HH^3_x L^2_v(m))$ under a smallness assumption on the initial data $\| f_0 \|_{\HH^3_x L^2_v(m)} \le \e_0$.
We introduce the notation to simplify
$$
\left\{
\bal
& X := \HH^3_x L^2_v (m), \quad
Y := \HH^3_x  (H^1_{v,*} (m)), \quad
Y' := \HH^3_x (H^{-1}_{v,*}(m)),
\\
& X_0 := \HH^3_x L^2_v (m_0 ), \quad
Y_0 := \HH^3_x (H^1_{v,*} (m_0)), \quad
Y'_0 := \HH^3_x (H^{-1}_{v,*} (m_0)), \quad
Z_0 := \HH^3_x (H^1_{v,**} (m_0)), 
\eal
\right.
$$
where we recall that these spaces are defined in \eqref{HH3xL2v}-\eqref{HH3xH1v*}-\eqref{HH3xH-1v*}, and we also remark that $\| f \|_{Z_0} \lesssim \| f \|_{Y}$.

\medskip

We split the proof of Theorem~\ref{main1} into three parts: Theorem~\ref{thm:existence}, Theorem~\ref{thm:uniq} and Theorem~\ref{thm:decay} below.

\subsubsection{A priori estimates}
We start proving a stability estimate.

\begin{prop}\label{prop:stab}
Any solution $f=f_t$ to \eqref{eq:f} satisfies, at least formally, the following differential inequality: for any $\lambda_2 < \lambda_1$ there holds
$$
\frac12\frac{d}{dt} \Nt f \Nt_{X}^2 \le
- \lambda_2 \Nt f \Nt_{X}^2
-\big( K - C\Nt f \Nt_{X} \big) \| f \|_{Y}^2,
$$
for some constants $K,C>0$.
\end{prop}

\begin{proof}[Proof of Proposition \ref{prop:stab}]
Recall that the norm $\Nt \cdot \Nt_{X}$ is defined in Proposition~\ref{prop:dissipative} and it is equivalent to the $\| \cdot \|_X$-norm.
Thanks to~\eqref{eq:f} we write
$$
\bal
\frac12\frac{d}{dt} \Nt f \Nt_{X}^2
&= \eta \la f , \Lambda f \ra_{X} + \int_0^\infty \la \SS_\Lambda(\tau) e^{\lambda_2 \tau} f , \SS_\Lambda(\tau)e^{\lambda_2 \tau} \Lambda f \ra_{X_0} \, d\tau     \\
&\quad
+ \eta \la f , Q(f,f) \ra_{X} + \int_0^\infty \la \SS_\Lambda(\tau)e^{\lambda_2 \tau} f , \SS_\Lambda(\tau)e^{\lambda_2 \tau}  Q(f,f) \ra_{X_0} \, d\tau \\
&=: I_1+I_2 + I_3 + I_4.
\eal
$$
For the linear part $I_1+I_2$, we already have from Proposition~\ref{prop:dissipative} that, for any $\lambda_2 < \lambda_1$,
$$
I_1+I_2 \le - \lambda_2 \Nt f \Nt_{X}^2 - K \| f \|_{Y}^2 .
$$
Let us investigate the nonlinear part. 
For the term $I_3$, Lemma~\ref{lem:Qf}-$(ii)$ gives us directly
$$
\bal
I_3
&\lesssim \| f \|_{X} \, \| f \|_{Y}^2 +
\| f \|_{X}^2 \, \| f \|_{Y} 
\lesssim \Nt f \Nt_{X} \, \| f \|_{Y}^2.
\eal
$$
For the last term $I_4$, we use the fact that $\Pi_0 f_t = 0$ and $\Pi_0 Q(f_t,f_t)=0$ for all $t\ge 0$, together with Corollary~\ref{cor:decay&regularity} to get
$$
\bal
&\int_0^\infty \la \SS_\Lambda(\tau) e^{\lambda_2 \tau} f , \SS_\Lambda(\tau) e^{\lambda_2 \tau}  Q(f,f) \ra_{X_0} \, d\tau \\
&\qquad\le \int_0^\infty \| \SS_\Lambda(\tau) e^{\lambda_2 \tau} f \|_{X_0} \, \| \SS_\Lambda(\tau) e^{\lambda_2 \tau} Q(f,f) \|_{X_0} \, d\tau \\
&\qquad\le \left( \int_0^\infty \| \SS_\Lambda(\tau) e^{\lambda_2 \tau} f \|_{X_0}^2 \, d\tau \right)^{1/2} \left( \int_0^\infty \| S_\Lambda(\tau) e^{\lambda_2 \tau} Q(f,f) \|_{X_0}^2 \, d\tau \right)^{1/2}  \\
&\qquad\lesssim \left( \int_0^\infty  e^{-2 (\lambda_1 - \lambda_2) \tau} \| f \|_{X_0}^2 \, d\tau \right)^{1/2} \left( \int_0^\infty e^{2 \lambda_2 \tau} \| S_\Lambda(\tau)  Q(f,f) \|_{X_0}^2 \, d\tau \right)^{1/2}  \\
&\qquad\lesssim  \| f \|_{X_0} \,  \| Q(f,f) \|_{Y'_0}  .
\eal
$$
From Lemma~\ref{lem:Qf}-$(i)$ we have
$$
\| Q(f,f) \|_{Y'_0} \lesssim \| f \|_{X_0} \, \| f \|_{Z_0}.
$$
Therefore, using that $m_0 \lesssim m \la v \ra^{-(1-\sigma/2)}$ so that $\| f \|_{Z_0} \lesssim \| f \|_{Y}$, we obtain
$$
I_4 \lesssim \| f \|_{X} \, \| f \|_{Y}^2
\lesssim  \Nt f \Nt_{X} \, \| f \|_{Y}^2,
$$
and the proof is complete.
\end{proof}

\Black

We prove now an a priori estimate on the difference of two solutions to \eqref{eq:f}.

\begin{prop}\label{prop:conv}
Consider two solutions $f$ and $g$ to \eqref{eq:f} associated to initial data $f_0$ and $g_0$, respectively. Then, at least formally, the difference $f-g$ satisfies the following differential inequality
$$
\bal
\frac12\frac{d}{dt} \Nt f  - g \Nt_{X_0}^2 
&\le
- K \| f- g \|_{Y_0}^2 
+ C \Nt g \Nt_{X_0} \, \| f-g \|_{Y_0}^2 \\
&\quad
+ C \big( \| g \|_{Y_0} + \| f \|_{Y} \big)\, \Nt f-g \Nt_{X_0} \, \| f-g \|_{Y_0},
\eal
$$
for some constants $K,C>0$.
\end{prop}

\begin{proof}
We write the equation safisfied by $f-g$:
$$
\left\{
\bal
\partial_t (f-g) &= \Lambda(f-g) + Q(g,f-g) + Q(f-g,f), \\
(f-g)_{|t=0} &= f_0 - g_0.
\eal
\right.
$$
Denote $\overline{X}_0 := \HH^3_x L^2_v(\bar m_0)$ where $ \bar m_0 \lesssim m_0 \la v \ra^{-(1-\sigma/2)}$ (see \eqref{norm-diss}).
Then we compute
$$
\bal
\frac12\frac{d}{dt} \Nt f_t - g_t \Nt_{X_0}^2
&= \eta \la (f - g) , \Lambda (f - g) \ra_{X_0} + \int_0^\infty \la S_\Lambda(\tau) e^{\lambda_2 \tau} (f - g) , S_\Lambda(\tau)e^{\lambda_2 \tau} \Lambda (f - g) \ra_{\overline X_0} \, d\tau     \\
&\quad
+ \eta \la (f - g) , Q(g,f - g) \ra_{X_0} + \int_0^\infty \la S_\Lambda(\tau)e^{\lambda_2 \tau} (f - g) , S_\Lambda(\tau)e^{\lambda_2 \tau}  Q(g,f-g) \ra_{\overline X_0} \, d\tau \\
&\quad
+ \eta \la (f - g) , Q(f - g, f ) \ra_{X_0} + \int_0^\infty \la S_\Lambda(\tau)e^{\lambda_2 \tau} (f - g) , S_\Lambda(\tau)e^{\lambda_2 \tau}  Q(f - g,f) \ra_{\overline X_0} \, d\tau \\
&=: T_1+T_2 + T_3 + T_4 + T_5 + T_6.
\eal
$$
Arguing as in Proposition \ref{prop:stab} we easily obtain, 
$$
T_1+T_2 \le 
- K \| f- g \|_{Y_0}^2,
$$
and also
$$
T_3+T_4 \lesssim \Nt g \Nt_{X_0} \, \| f-g \|_{Y_0}^2 +
\| g \|_{Y_0} \, \Nt f-g \Nt_{X_0} \, \| f-g \|_{Y_0}.
$$
Moreover, for the last part $T_5 + T_6$, arguing as in Proposition~\ref{prop:stab} and using Lemma~\ref{lem:Qf}-$(i)$, we get
$$
T_5 + T_6 
\lesssim \Nt f-g \Nt_{X_0} \, \| f \|_{Z_0} \, \| f-g \|_{Y_0}
\lesssim \Nt f-g \Nt_{X_0} \, \| f \|_{Y} \, \| f-g \|_{Y_0},
$$
which completes the proof.
\end{proof}

\subsubsection{Cauchy problem in the close-to-equilibrium setting}
Thanks to the a priori estimates in Proposition~\ref{prop:stab} and Proposition~\ref{prop:conv}, we are now able to construct solutions to \eqref{eq:f} on $L^\infty_t (X) = L^\infty_t (\HH^3_x L^2_v (m))$, assuming a smallness condition on the initial data.

\begin{thm}\label{thm:existence}
There is a constant $\e_0 = \e_0(m)>0$ such that, if $\Nt f_0 \Nt_{X} \le \e_0$ then there exists a global weak solution $f$ to \eqref{eq:f} that satisfies, for some constant $C>0$,
$$
\| f \|_{L^\infty ([0,\infty); X)} + \| f \|_{L^2 ([0,\infty); Y)} \le C \e_0.
$$
Moreover, if $F_0 = \mu+ f_0 \ge 0$ then $F(t) = \mu + f(t) \ge 0$.
\end{thm}

\begin{proof}
The proof follows a standard argument by introducing an iterative scheme and using the estimates established in Propositions~\ref{prop:stab} and \ref{prop:conv}, thus we only sketch it.

\smallskip

For any integer $n \ge 1$ we define the iterative scheme
$$
\left\{
\bal
 \partial_t f^n  &= \Lambda f^n + Q(f^{n-1},f^n) \\
 f^n_{|t=0} &= f_0
\eal
\right.
\quad \forall \, n \ge 1, \quad \text{and}\quad
\left\{
\bal
\partial_t f^0 &= \Lambda f^0  \\
f^0_{|t=0} &= f_0
\eal
\right..
$$
Firstly, the functions $f^n$ are well defined on $X$ for all $t \ge 0$ thanks to the semigroup theory in Theorem~\ref{thm:extension} and Corollary~\ref{cor:decay&regularity}, and the stability estimates proven below.

\medskip
\noindent
{\it Step 1. Stability of the scheme.} 
We first prove the stability of the scheme on $X$.
Thanks to Propositions~\ref{prop:stab}, we prove by induction that, if $\e_0>0$ is small enough, there holds
\beqn\label{An(t)}
\forall\, n \ge 0, \forall \, t \ge 0, \quad
A_n(t) :=  \Nt f^n_t \Nt_{X}^2 + K \int_0^t \| f^n_\tau \|_{Y}^2 \, d\tau \le 2 \e_0^2.
\eeqn

\medskip
\noindent
{\it Step 2. Convergence of the scheme.} We now turn to the convergence of the scheme in $X_0$. Denote $d^n = f^{n+1} - f^n$ that satisfies
$$
\left\{
\bal
\partial_t d^n & = \Lambda d^n + Q(f^n,d^n)  + Q(d^{n-1}, f^n), \quad \forall\, n \in \N^*;\\
\partial_t d^0 & = \Lambda d^0 + Q(f^0,f^1).
\eal
\right.
$$
Thanks to Proposition~\ref{prop:stab}, Proposition~\ref{prop:conv} and estimate \eqref{An(t)},
we then prove by induction that, for $\e_0 >0$ small enough, it holds
\beqn\label{Bnt}
\forall\, t \ge 0, \forall\, n \ge 0, \quad
B_n(t) := \Nt d^n_t \Nt_{X_0}^2 + K \int_0^t \| d^n_\tau \|_{Y_0}^2 \, d\tau \le (C'\e_0)^{2n},
\eeqn
for some constant $C'>0$ that does not depend on $\e_0$.

Therefore the sequence $(f^n)_{n \in \N}$ is a Cauchy sequence in $L^\infty( [0,\infty) ;X_0) = L^\infty([0,\infty); \HH^3_x L^2_v(m_0))$, and its limit $f$ satisfies \eqref{eq:f} in a weak sense. We then deduce that
$$
\| f \|_{L^\infty ([0,\infty); X)} + \| f \|_{L^2 ([0,\infty); Y)} \le C \e_0,
$$
by passing to the limit $n\to \infty$ in \eqref{An(t)}.
Moreover, since $F_0 = \mu + f_0 \ge 0$ we easily obtain that $F(t) = \mu + f(t) \ge 0$ (see e.g. \cite{Guo}).
\end{proof}

We can now address the problem of uniqueness.


\begin{thm}\label{thm:uniq}
There is a constant $\e_0 = \e_0(m)>0$ such that, if $\Nt f_0 \Nt_{X} \le \e_0$ then there exists a unique global weak solution $f \in L^\infty([0,\infty); X) \cap L^2 ([0,\infty); Y)$ to \eqref{eq:f} such that
$$
\| f \|_{L^\infty([0,\infty); X)} + \| f \|_{L^2([0,\infty); Y)} \le C \e_0.
$$
\end{thm}

\begin{proof}
Let $f$ and $g$ be two solutions to \eqref{eq:f} with same initial data $g_0=f_0$ that satisfy
$$
\| f \|_{L^\infty([0,\infty); X)} + \| f \|_{L^2([0,\infty); Y)} \le C \e_0.
$$
and
$$
\| g \|_{L^\infty([0,\infty); X)} + \| g \|_{L^2([0,\infty); Y)} \le C \e_0.
$$
The difference $f-g$ satisfies then
$$
\partial_t (f-g) = \Lambda(f-g) + Q(g,f-g) + Q(f-g,f),
$$
with $f_0=g_0$.
We then compute the standard $L^2_x L^2_v(m_0)$-norm of the difference $f-g$
$$
\bal
\frac12 \frac{d}{dt} \| f-g \|_{L^2_x L^2_v(m_0)}^2
&= \la \Lambda (f-g), f-g \ra_{L^2_x L^2_v(m_0)}
+ \la Q(g,f-g),f-g \ra_{L^2_x L^2_v(m_0)} \\
&\quad
+ \la Q(f-g,f),f-g \ra_{L^2_x L^2_v(m_0)}.
\eal
$$
We write $\Lambda = \AA + \BB$ so that we obtain
$$
\la \Lambda (f-g), f-g \ra_{L^2_x L^2_v(m_0)} \le - K \| f-g \|_{L^2_x (H^1_{v,*}(m_0))}^2
+ C \| f-g \|_{L^2_x L^2_v(m_0)}^2.
$$
Moreover, Lemma~\ref{lem:NL1}-(ii) together with \eqref{eq:Sob_x} gives
$$
\bal
\la Q(g,f-g),f-g \ra_{L^2_x L^2_v(m_0)}
&\le C \| g \|_{H^2_x L^2_v (m_0)} \,
\| f-g \|_{L^2_x (H^1_{v,*}(m_0))}^2 + C \| g \|_{H^2_x (H^1_v(m_0))} \, \| f-g \|_{L^2_x L^2_v(m_0)}^2,
\eal
$$
whence, integrating in time,
$$
\bal
&\int_0^t \la Q(g_\tau,f_\tau-g_\tau),f_\tau-g_\tau \ra_{L^2_x L^2_v(m_0)} \, d\tau \\
&\quad
\le C \sup_{\tau \in[0,t]} \| g_\tau \|_{H^2_x L^2_v (m_0)}  \int_0^t
\| f_\tau - g_\tau \|_{L^2_x (H^1_{v,*}(m_0))}^2 \\
&\qquad
+ C  \left( \int_0^t \| g_\tau \|_{H^2_x (H^1_v(m_0))}^2 \right)^{1/2} \left( \sup_{\tau \in[0,t]} \| f_\tau - g_\tau \|_{L^2_x L^2_v(m_0)}^2 + \int_0^t \| f_\tau - g_\tau \|_{L^2_x L^2_v(m_0)}^2 \right).
\eal
$$
Thanks to Lemma~\ref{lem:NL1}-(i) it follows
$$
\bal
\la Q(f-g,f),f-g \ra_{L^2_x L^2_v(m_0)}
&\le C \| f - g \|_{L^2_x L^2_v(m_0)} \, \| f \|_{H^2_x (H^1_{v,**}(m_0))} \, \| f-g \|_{L^2_x (H^1_{v,*}(m_0))} ,
\eal
$$
which integrating in time gives
$$
\bal
&\int_0^t \la Q(f_\tau-g_\tau,f_\tau),f_\tau-g_\tau \ra_{L^2_x L^2_v(m_0)} \, d\tau \\
&\le C \left(\sup_{\tau \in [0,t]} \| f_\tau - g_\tau \|_{L^2_x L^2_v(m_0)}\right)\int_0^t  \| f_\tau \|_{H^2_x (H^1_{v,**}(m_0))} \, \| f_\tau-g_\tau \|_{L^2_x (H^1_{v,*}(m_0))} \\
&\le C \left(\int_0^t \| f_\tau \|_{H^2_x (H^1_{v,**}(m_0))}^2 \right)^{1/2} \left( \sup_{\tau \in [0,t]} \| f_\tau - g_\tau \|_{L^2_x L^2_v(m_0)}^2 + \int_0^t \| f_\tau-g_\tau \|_{L^2_x (H^1_{v,*}(m_0))}^2 \right),
\eal
$$
and observe that $\| f \|_{L^2_t (H^2_x (H^1_{v,**}(m_0)))} \lesssim \| f \|_{L^2_t (Y)} \le C \e_0$.
Therefore
$$
\bal
&\| f_t-g_t \|_{L^2_x L^2_v(m_0)}^2 + K \int_0^t \| f_\tau - g_\tau \|^2_{L^2_x (H^1_{v,*}(m_0))} \, d\tau \\
&\quad \le C \int_0^t \| f_\tau - g_\tau \|_{L^2_x L^2_v(m_0)}^2 \, d\tau
+ C \e_0  \int_0^t \| f_\tau - g_\tau \|_{L^2_x (H^1_{v,*}(m_0))}^2 \, d\tau \\
&\quad\quad
+ C\e_0 \left( \sup_{\tau \in [0,t]} \| f_\tau - g_\tau \|_{L^2_x L^2_v(m_0)}^2 + \int_0^t \| f_\tau-g_\tau \|_{L^2_x (H^1_{v,*}(m_0))}^2 \, d\tau\right),
\eal
$$
and when $\e_0>0$ is small enough we conclude the proof of uniqueness by Gronwall's inequality.
\end{proof}

\subsubsection{Convergence to equilibrium in the close-to-equilibrium setting}

\begin{thm}\label{thm:decay}
There is a positive constant $\e_1 \le \e_0 $ so that, if $\Nt f_0 \Nt_{X} \le \e_1$, then
the unique global weak solution $f$ to \eqref{eq:f} (constructed in Theorems~\ref{thm:existence} and \ref{thm:uniq}) verifies an exponential decay: for any $\lambda_2 < \lambda_1$ there exists $C >0$ such that
$$
\forall \, t \ge 0, \quad
\| f(t) \|_{X} \le C \, e^{-\lambda_2 t} \, \| f_0 \|_{X},
$$
where we recall that $\lambda_1>0$ is the optimal rate given by the semigroup
decay in Theorem~\ref{thm:extension}.
\end{thm}

\begin{proof}
From Theorem \ref{thm:existence} we have
$$
\sup_{t \ge 0} \Nt f(t) \Nt_{X}^2 + \int_0^t \| f(\tau) \|_{Y}^2 \, d\tau \le C\e_1^2.
$$
Using Proposition~\ref{prop:stab} we get, if $\e_1>0$ is small enough so that $- K + C\e_1 \le - K/2$, and for any $\lambda_2 < \lambda_1$,
$$
\bal
\frac12\frac{d}{dt} \Nt f \Nt_{X}^2
&\le - \lambda_2 \Nt f \Nt_X^2 - (K - C\e_1) \| f \|_{Y}^2 \\
&\le - \lambda_2 \Nt f \Nt_X^2 - \frac{K}{2} \| f \|_{Y}^2,
\eal
$$
and then we deduce an exponential convergence
$$
\forall \, t \ge 0, \qquad
\Nt f(t) \Nt_{X} \le e^{-\lambda_2 t} \, \Nt f_0 \Nt_{X},
$$
which implies
$$
\forall \, t \ge 0, \qquad
\| f(t) \|_{X} \le C e^{- \lambda_2 t} \, \| f_0 \|_{X}.
$$
\end{proof}


\end{document}